\documentclass[11pt]{amsart}
\usepackage{amssymb, amscd, amsmath, amsthm, epsf, epsfig, latexsym,psfrag,color} 
\usepackage{hyperref}

\usepackage{float}
\usepackage{psfrag}
\usepackage{comment}

\usepackage{pb-diagram}
\usepackage{graphicx}

\newcommand{\TryPackage}[3]{\IfFileExists{#1.sty}{\usepackage{#1}#2}{#3}
}
\TryPackage{mathrsfs}{\renewcommand{\mathcal}{\mathscr}}{%
        \TryPackage{eucal}{}{}}

\newcommand{\bbi}{{{\bf i}}}
\newcommand{\bbj}{{{\bf j}}}
\newcommand{\bbk}{{{\bf k}}}

\newcommand{\ZZ}{{\mathbb Z}}
\newcommand{\RR}{{\mathbb R}}
\newcommand{\CC}{{\mathbb C}}

\newcommand{\QQ}{{\mathbb Q}}

\newcommand{\cC}{{\mathcal C}}
\newcommand{\cG}{{\mathcal G}}

\newcommand{\ccH}{{\mathcal H}}

\newcommand{\red}{{\hbox{\scriptsize \sl red}}}

\newcommand{\Hom}{\operatorname{Hom}}

\newcommand{\nat}{\natural}
\newcommand{\cs}{\operatorname{CS}}
\newcommand{\Stab}{\operatorname{Stab}}
\newcommand{\rank}{\operatorname{rank}}
\newcommand{\Real}{\operatorname{Re}}

\graphicspath{ {figures/} }

\theoremstyle{definition}

\newtheorem{df}{Definition}[section]

\theoremstyle{plain}

\newtheorem{theorem}{Theorem}
\newtheorem{thm}[df]{Theorem}

\newtheorem{cor}[df]{Corollary}
\newtheorem{lem}[df]{Lemma}
\newtheorem{prop}[df]{Proposition}

\title[The pillowcase and perturbations of traceless representations]{The pillowcase and perturbations of traceless representations of knot groups}
\author{Matthew Hedden}\author{Chris Herald} \author{Paul Kirk}

\thanks{The first author gratefully acknowledges support from NSF grant DMS-0906258,  NSF CAREER grant DMS-1150872, and an Alfred P. Sloan Research Fellowship.  The third author gratefully acknowledges support from 
NSF  grant DMS-1007196}

\address{Department of Mathematics, Michigan State University \newline
\hspace*{.375in} East Lansing, MI 48824} 
\email{\rm{mhedden@math.msu.edu}}

\address{Department of Mathematics,   University of Nevada, Reno \newline
\hspace*{.375in} Reno, NV 89557} 
\email{\rm{herald@unr.edu}}

\address{Department of Mathematics, Indiana University \newline
\hspace*{.375in}  Bloomington, IN 47405} 
\email{\rm{pkirk@indiana.edu}}


\makeatletter
\def\@tocline#1#2#3#4#5#6#7{\relax
  \ifnum #1>\c@tocdepth 
  \else
    \par \addpenalty\@secpenalty\addvspace{#2}%
    \begingroup \hyphenpenalty\@M
    \@ifempty{#4}{%
      \@tempdima\csname r@tocindent\number#1\endcsname\relax
    }{%
      \@tempdima#4\relax
    }%
    \parindent\z@ \leftskip#3\relax \advance\leftskip\@tempdima\relax
    \rightskip\@pnumwidth plus4em \parfillskip-\@pnumwidth
    #5\leavevmode\hskip-\@tempdima
      \ifcase #1
       \or\or \hskip 1em \or \hskip 2em \else \hskip 3em \fi%
      #6\nobreak\relax
    \dotfill\hbox to\@pnumwidth{\@tocpagenum{#7}}\par
    \nobreak
    \endgroup
  \fi}
\makeatother

\begin{document}

 \begin{abstract} We introduce explicit holonomy perturbations of the Chern-Simons functional on a 3-ball containing a pair of unknotted arcs.  These perturbations give us a concrete local method for making  the moduli spaces of flat singular $SO(3)$ connections  relevant to Kronheimer and Mrowka's singular instanton knot homology  non-degenerate.  The  mechanism for this study is a (Lagrangian)   intersection diagram which arises, through restriction of representations, from a tangle decomposition of a knot.  When one of the tangles is trivial, our perturbations allow us to study isolated intersections of two Lagrangians to  produce minimal generating sets for singular instanton knot homology. The  (symplectic) manifold where this intersection occurs corresponds to the traceless character variety of the four-punctured 2-sphere, which we identify with the familiar pillowcase. We investigate the  image in this pillowcase of the traceless representations of tangles obtained by removing a trivial tangle from 2-bridge knots and torus knots.  Using this, we compute the singular instanton homology of a variety of torus knots.  In many cases, our computations allow us to understand non-trivial differentials in the spectral sequence from Khovanov homology to singular instanton homology.  \end{abstract} 
 \maketitle

 \tableofcontents
 
 \section{Introduction}
 Kronheimer and Mrowka  have recently developed a variant of instanton homology for knots in 3-manifolds \cite{KM1, KM-khovanov, KM-filtrations} which they call {\em singular instanton knot homology}. In \cite{KM-khovanov},  they  construct a filtered chain  complex  whose total homology is the singular instanton homology of a knot $K\subset S^3$, and whose spectral sequence has $E_2$ page the Khovanov homology of $K$.  Inspired by early observations of similarities between their theory and Khovanov homology, Lewallen \cite{Lewallen}  showed that for 2-bridge knots  the Khovanov homology is isomorphic to the homology of the variety of  $SU(2)$ representations of the fundamental group of the knot complement sending the meridian to a traceless matrix.  For alternating knots, he showed that the Khovanov homology is isomorphic to the homology of the subvariety of binary dihedral representations.  
  
 In fact for alternating knots the Khovanov and singular
instanton homology groups are isomorphic (with the Khovanov bigrading appropriately collapsed to a  $\ZZ/4$ grading), a fact implied by the collapse of Kronheimer and Mrowka's spectral sequence  at the $E_2$ page.  In contrast, they show that  there are non-trivial higher differentials in the spectral sequence associated to the $(4,5)$ torus knot \cite[Section 11]{KM-filtrations}, and hence Khovanov and instanton homology do not  have the same rank, in general.  (Rasmussen noticed that  the Khovanov homology of the $(4,5)$ torus knot also has larger rank than its  Heegaard knot Floer homology groups. It is conjectured that there is similar spectral sequence in that context.) Zentner \cite{Zentner}  showed that for some alternating pretzel knots there are non-binary dihedral traceless representations (in contrast to 2-bridge knots), so that for these families one expects there to be non-trivial differentials in the singular instanton chain complex.

To make sense of this expectation, we should recall that the chain complex from which the  instanton homology is computed is  the Morse complex associated to a perturbation of a particular Chern-Simons functional.  This chain complex is $\ZZ/4$ graded, and is generated by certain gauge equivalence classes of perturbed-flat  connections. These connections live on an $SO(3)$ bundle on the complement of the link formed by adding a small meridional circle $H$ to $K$. The second Stiefel-Whitney class of the bundle is Poincar\'e dual to an arc $W$ connecting $K$ and $H$, as in Figure \ref{natural}.   The {\em unperturbed} flat moduli space can be identified with the conjugacy classes of  $SU(2)$ representations of the fundamental group of $S^3\setminus (K\cup H\cup W)$,  which take the meridians of $K$ and $H$ to traceless matrices in $SU(2)$, and  the meridian of $W$ to the non-trivial central element $-1$. A feature of these representation spaces is that, with the exception of the unknot in $S^3$, they are never non-degenerate; the corresponding (unperturbed) Chern-Simons functional is not Morse. Thus it is necessary to perturb the functional to identify the generators of the instanton complex.  A common method for perturbing the Chern-Simons functional is through the use of so-called {\em holonomy perturbations}, and such perturbations can be constructed quite generally in instanton Floer theories.  The use of such perturbations, however, obscures the connection between generators of the instanton chain complex and representations of the fundamental group.   In particular, it would be desirable to be able to effectively estimate the rank of a reduced instanton chain complex from a presentation of the fundamental group of the knot or link complement.  A general perturbation of the Chern-Simons functional will make this estimation impossible.

\begin{figure}
\centering
\def\svgwidth{1.2in}
 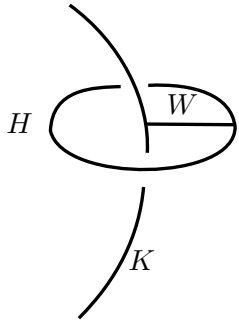
 \caption{A small segment of a knot $K$ and the meridional circle $H$ and arc $W$ used in the construction of the reduced singular instanton homology. \label{natural}}
 \end{figure}

The purpose of this article is to identify a certain perturbation explicitly and to use it to identify generators of the reduced instanton chain complex. The main conceptual step is to split a 3-manifold containing a knot along a {\em Conway sphere}; that is, a $2$--sphere which intersects the knot in four points.   The sphere decomposes the knot into two tangles, and determining generators of the instanton complex becomes an intersection problem for two Lagrangians in the relevant character variety of the $4$-punctured sphere.   The key to our result lies in picking  a Conway sphere for which one of the associated tangles is trivial, and constructing  concrete local (and hence universal) perturbations in this trivial piece.
 
 It turns out (Proposition \ref{s2prop}) that the relevant character variety for a 2-sphere with four marked points, $R(S^2,\{a,b,c,d\})$,  is a {\em pillowcase}, i.e. a topological 2-sphere with four singular points, understood as the quotient of the torus by the hyperelliptic involution or, equivalently, as the quotient of $\RR^2$:
 $$
 R(S^2,\{a,b,c,d\})\cong\RR^2/\sim,\ \ \ \  ~ (\gamma,\theta)\sim (-\gamma,-\theta) \sim(\gamma+2\pi m,\theta  +2\pi n).
 $$  
We describe paths in $ R(S^2,\{a,b,c,d\})$ by giving formulas for their lifts to $\RR^2$, e.g.~ $t\mapsto (\gamma(t), \theta(t))$.  In terms of the pillowcase, our main results, Theorem \ref{pertvar}  and Corollary \ref{corimage}, state the following.

\medskip

\begin{theorem} Let $A_1\cup A_2\subset B^3$ be a pair of unknotted arcs in the 3-ball, $H$ a small meridian of $A_1$, and $W$ an arc connecting $H$ and $A_1$. Let $P\subset B^3\setminus (A_1\cup A_2\cup H\cup W)$ be the perturbation curve    illustrated in Figure \ref{fig8}.    

Given small perturbation data $\pi=(P,\epsilon, f)$ (where $\epsilon>0$ and $f$ is an odd, $2\pi$-periodic function, e.g.~$f(\beta)=\sin(\beta))$, the perturbed reduced moduli space $R_\pi^\nat(B^2,A_1\cup A_2)$ is homeomorphic to a circle, and the image of the restriction map to the pillowcase 
$$\rho:R_\pi^\nat(B^2,A_1\cup A_2)\to R(S^2,\{a,b,c,d\})$$
is an immersion, given by 
$$\rho(\beta)=(\gamma(\beta), \theta(\beta))=
 (\beta+ \epsilon f(\beta)+ \tfrac\pi 2 ,  \beta- \epsilon  f(\beta)+ \tfrac\pi 2 ),~\beta\in S^1$$
as illustrated in Figure \ref{fig9}.  As $\epsilon\to 0$, this immersion limits to a generically 2-1 map onto the diagonal arc $\gamma=\theta$ in the pillowcase.
\end{theorem}

\medskip

\noindent A more precise statement can be found in the body of the paper. For all the results in this article it suffices to take $f(\beta)=\sin(\beta)$ in the perturbation data.

\medskip
In Section \ref{unreduced case} we treat the   case corresponding to the unreduced instanton homology $I^\sharp(Y,K)$.  The counterpart of Theorem \ref{pertvar} in this context is Theorem \ref{pertvar2}, which states in part:

\medskip

\begin{theorem} With perturbation data $\pi=(P_1\cup P_2, \epsilon, \sin(\beta))$ supported on the two curves $P_1$ and $P_2$ in Figure \ref{unreduced}, the perturbed unreduced moduli space $R^\sharp_\pi(B^3, A_1\cup A_2)$ is a disjoint union of two   circles, parameterized by $\beta\in S^1$ and $i=1,2$. The restriction map 
 $R^\sharp_\pi(B^3, A_1\cup A_2)\to R(S^2,\{m,b,c,n\})$ is an immersion, given by 
 $$\rho_i(\beta)= 
 \big(-\tau_i(\beta) + \frac{\pi}{2} +\beta +\epsilon\sin\beta,
 -\tau_i(\beta) + \frac{\pi}{2} +\beta -\epsilon\sin\beta\big), ~ i=1,2,$$
where $\tau_1(\beta)=\arcsin(-\tfrac{1}{2}\sin\beta)$ and $\tau_2(\beta)=\pi-\tau_1(\beta)$. As $\epsilon\to 0$, this immersion limits to a generically 4-1 map onto the diagonal arc $\gamma=\theta$ in the pillowcase.
\end{theorem}

\medskip

\noindent Figure \ref{unpertpic} illustrates the images in the pillowcase of the three curves, $\rho_1,\rho_2$ of Theorem 2, and $\rho$ of Theorem 1.

\medskip

As alluded to above, a useful  implication  of these results is that it reduces the problem of identifying generators of the singular instanton complexes for  a knot $K$ in a 3-manifold $Y$ to an intersection problem in the pillowcase.  More precisely, consider a 3-ball  which intersects $(Y,K)$ in a pair of unknotted arcs $A_1\cup A_2$. Setting  $(Y_0,K_0):=(Y\setminus B^3,K\setminus(A_1\cup A_2))$, we obtain a tangle decomposition
 \begin{equation}
\label{intlag}
(Y,K)=(Y_0,K_0)\cup_{(S^2,\{a,b,c,d\})} (B^3, A_1\cup A_2).
\end{equation}
This decomposition yields, upon passing to  the appropriate  moduli spaces,  an intersection diagram:

\begin{equation}\label{SVKD}
\begin{diagram}\node[2]{R(S^2,\{a,b,c,d\})}\\
\node{R(Y_0,K_0)}\arrow{ne}\node[2]{R^\nat_\pi(B^3, A_1\cup A_2)}\arrow{nw}\\
\node[2]{R_\pi^\nat(Y,K)}\arrow{nw}\arrow{ne}
\end{diagram}
\end{equation}

\noindent The  intersection  $R_\pi^\nat(Y,K)$  parameterizes the generators of the chain complex defining the instanton homology $I^\nat(Y,K)$, provided this is a   non-degenerate set. The unperturbed space $R^\nat(Y,K)$ is   never  non-degenerate, except for the unknot in $S^3$, in contrast to other examples of this method in low-dimensional topology such as  Casson's invariant \cite{AM}.

In the decomposition (\ref{intlag}),  the  non-trivial part of the $SO(3)$ bundle and the perturbation $\pi$ needed to make the set $R(Y,K)$ non-degenerate, has been placed entirely inside the simple space $B^3$.  Theorem \ref{pertvar} states that $R^\nat_\pi(B^3, A_1\cup A_2)  \to R(S^2,\{a,b,c,d\})$ is a smooth immersion of a circle, and identifies the image precisely.

It follows that the problem of describing the set $R^\nat(Y,K)$ of generators of the instanton homology chain complex  is reduced to understanding the space $R(Y_0, K_0)$   and its restriction 
to the pillowcase, a problem that involves only the fundamental group of the 2-stranded tangle complement $Y_0\setminus K_0$ and its peripheral structure.
Indeed, for simple knots like  2-bridge knots and torus knots, no further perturbations are needed. For general knots only perturbations in the knot complement (which have been studied in detail for a long time, see e.g.~ \cite{herald})  are required.

\medskip

With this understanding in place, we turn our attention to the problem of describing $R(Y_0,K_0)$ and its image in the pillowcase $R(S^2,\{a,b,c,d\})$ for 2-bridge knots and torus knots. More precisely, given a knot $K$ in $S^3$ and a 3-ball $B^3$ meeting $K$ in a pair of unknotted arcs $A_1\cup A_2$, the complement $Y_0$ of this 3-ball is again a 3-ball  and contains the two component tangle $K_0= K\setminus (A_1\cup A_2)$.  The moduli space $R(Y_0,K_0)$ is identified with the space of conjugacy classes of $SU(2)$ representations of $\pi_1(Y_0\setminus K_0)$ which send meridians of $K_0$ to traceless matrices. 
This space, and its restriction to the pillowcase, turns out to be a very interesting tangle invariant. In Section \ref{twobridgesection}  we identify $R(Y_0,K_0)$ and its image in the pillowcase for rational tangles (the tangles which glue with the trivial tangle to produce 2-bridge knots) and prove the following theorem. 

\medskip

\begin{theorem}  For  the 2-bridge knot  $K=K(p/q)$ and the 3-ball $B^3\subset S^3$ meeting $K$ in a pair of unknotted arcs as in Figure \ref{2-bridgeK}, the space $R(Y_0, K_0)$ is an arc, and the restriction $R(Y_0, K_0)\to R(S^2,\{a,b,c,d\})$ is the embedding
$$(\gamma,\theta)=(qt, (q-p)t), ~t\in [0,\pi]. $$
\end{theorem}

\medskip

Using Theorem \ref{pertvar}  and Corollary \ref{corimage} it follows that  the set $R^\nat_\pi(S^3, K)$ is a union of pairs of isolated non-degenerate points $x_{\ell_1},x_{\ell_2}$, $\ell=1,2,\cdots,\frac{p-1}{2}$    and one additional point $\alpha'$.  
In particular, the chain complex $CI^\nat(S^3,K)$ for the reduced instanton homology of $K(p/q)$ is generated by these $ 2(\frac{p-1}{2})+1=p$  points.  From \cite{KM-khovanov} we know that all differentials are zero so that these points generate the instanton homology.  While the instanton homology of 2-bridge knots is easily  determined by the spectral sequence of \cite{KM-khovanov}, we find it interesting to be able to produce an explicit complex with trivial differential (in contrast to the complex coming from the spectral sequence which has many more generators than the rank of  homology, or the highly degenerate unperturbed character variety studied in \cite{Lewallen}).

\medskip

In Section \ref{torusknotsect}
we analyze the corresponding representation spaces for tangles arising from torus knots. These are more complicated than the  spaces associated to 2-bridge knots. For a particular tangle decomposition $(Y_0,K_0)\cup (B^3, A_1\cup A_2)$ of  the $(p,q)$ torus knot, the space $R(Y_0,K_0)$   can be identified with a certain  singular  semi-algebraic curve in $\RR^2$, cut out by a 2-variable polynomial  determined entirely by the integers $p,q$ and $r,s$, where $pr+qs=1$.   We obtain several results (Theorems \ref{torusknotreps} and \ref{toruspt2}) which give different descriptions of $R(Y_0,K_0)$ for torus knots. See Figure \ref{fig34} for the explicit example of the $(3,4)$ torus knot, a knot whose   instanton chain complex (for all small enough perturbations) necessarily has a non-trivial differential. Our results show that for torus knots a reduced singular instanton chain complex can be found with $|\sigma(K)|+1$ generators, where $\sigma(K)$ denotes the signature of $K$.  Moreover, for small generic holonomy perturbations this is the fewest possible generators, even if the instanton homology has smaller rank.

\medskip

The pillowcase arises in another   context when studying the $SU(2)$ representation spaces of knot complements: as the character variety $\chi(T)$ of {\em all} $SU(2)$ representations of the fundamental group of a torus.
 We will have occasion to use both incarnations, $R(S^2,\{a,b,c,d\})$ and $\chi(T)$, of the pillowcase in this article. Indeed, our analysis of $R^\nat(Y,K)$ calls on  the examination of the restriction from $R(Y_0,K_0)$ to the  pillowcase  $R(S^2,\{a,b,c,d\})$,  as well as the restriction of the full character variety $\chi(Y,K)$ (the space of conjugacy classes of all representations of $\pi_1(Y\setminus K)$) to the  pillowcase  $\chi(T)$ associated to the peripheral torus.  
 
  This latter context is familiar and can be found in many places in the literature, starting with Klassen's article \cite{Klassen}, and, in the enlarged context of $SL(2,\CC)$ representation, as the variety defined by the A-polynomial of a knot \cite{CCGLS}.  The interplay between these two manifestations of the pillowcase is exploited in Section \ref{data}.  We combine our results with the foundational theorems of Kronheimer-Mrowka to  calculate and tabulate   the reduced instanton complex, the Khovanov homology, and the instanton homology  for various families of torus knots. In particular we give new examples of torus knots for which there  are   non-trivial differentials in the instanton chain complex, examples for which there are many higher non-trivial differentials in the spectral sequence, and non 2-bridge examples for which all differentials in the chain complex are zero.  For example the spectral sequence for the $(5,7)$ torus knot  drops in rank from 29 to 17 after the $E_2$ page, and the spectral sequence for the $(13,28)$ torus knot collapses at  Khovanov homology.

\medskip

It is worth contrasting our approach with the one taken in \cite{JacRub}, where  a knot in $S^3$ is described by a closed braid, and a Lagrangian intersection picture is obtained by cutting along a $2n$-punctured 2-sphere which separates the braid from a trivial braid used to close it. The purpose of that article, however, is to explore the symplecto-geometric aspects of their setup, whereas our emphasis is on singular instanton homology and its efficient computation.

  \medskip
  
  Readers interested in a quick geometric overview of the contents of this article are encouraged to examine Figures \ref{fig8} and \ref{fig9}, which  encapsulate the statement of Theorem 1. Figures \ref{unreduced} and \ref{unpertpic} illustrate  Theorem 2.
  The reader is also encouraged to compare Figures \ref{figtrefoil} and \ref{fig12} illustrating the Trefoil knot,   Figures  \ref{fig21} and \ref{fig34} illustrating the $(3,4)$ torus knot,  and Figures \ref{figchi(7-2)} and \ref{fig72} illustrating the knot $7_2$.  These figures illustrate how to determine generators of the singular instanton chain complexes of these knots from the intersection  of arcs of traceless and perturbed traceless representations in the pillowcase.

 \section{Unit quaternions}In this section we establish  notation and recall some basic facts about $SU(2)$. Identify $SU(2)$ with the unit quaternions, 
 $$SU(2)=\{a+b\bbi+c\bbj+d\bbk~|~a,b,c,d\in \RR, ~a^2+b^2+c^2=1\}.$$
 The inverse of a unit quaternion $q=a+b\bbi+c\bbj+d\bbk$ equals its conjugate $\bar q=a-b\bbi-c\bbj-d\bbk$. 
 The Lie algebra of $SU(2)$ is  identified with the {\em pure quaternions}
 $$su(2)=\{b\bbi+c\bbj+d\bbk~|~b,c,d\in \RR\}$$
with invariant inner product  $v\cdot w=-\Real(vw)$.
 We denote the exponential map $su(2)\to SU(2)$  by $q\mapsto e^q$.
	Let $C(\bbi)\subset SU(2)$ denote the conjugacy class of $\bbi$: this is the 2-sphere of {\em pure unit quaternions} 
$$C(\bbi)=\{b\bbi+c\bbj+d\bbk~|~b,c,d\in \RR, b^2+c^2+d^2=1\}.$$
We also call $C(\bbi)$ the {\em traceless} unit quaternions, since they correspond to the traceless $2\times 2$ matrices in the usual description of $SU(2)$.

Notice that $C(\bbi)$ lies in the Lie algebra $su(2)$, and that as a subset of $SU(2)$ it corresponds to the unit quaternions $q\in SU(2)$ satisfying $\Real(q)=0$. 
Furthermore,
$$e^{\nu Q}=\cos(\nu)+\sin(\nu)Q  \text{ for }\nu\in \RR,  \text{ and }  Q^2=-1\text{ for }Q\in C(\bbi).$$  In particular, every non-zero vector in $su(2)$ can be uniquely written in the form $tQ$ for some $t\in \RR_{> 0}, Q\in C(\bbi)$. Every element of $SU(2)$ can be written in the form $e^{tQ}$ for some $t\in \RR, Q\in C(\bbi)$; this representation is unique for $SU(2)\setminus \{\pm 1\}$ if we choose $0<t<\pi$.

We summarize a few well known and easily verified properties of the conjugation action of $SU(2)$ on itself in the following proposition.

\begin{prop} \label{basic} Consider the action of $SU(2)$ on itself by conjugation.
\begin{enumerate}
\item The stabilizer of any subgroup  $G\subset SU(2)$, 
$$\Stab(G)=\{a\in SU(2)~|~ ag\bar a=g~\text{for all }~g\in G \},$$can be one of the three types: $\{\pm 1\},  \{e^{tQ}\}$ for some $Q\in C(\bbi)$, or $SU(2)$ according to whether 
$G$ is non-abelian, $G$ is abelian but non-central, or $G$ is contained in the center $\{\pm 1\}$.
\item Given a pure unit quaternion $Q\in C(\bbi)$ and $t\in \RR$, the conjugation action of $e^{tQ}$ on the 2-sphere $C(\bbi)$ of pure unit quaternions is rotation about the axis through $Q$  of angle $2t$.
\item If $q$ is a unit quaternion and $Q \in C(\bbi)$  is a  pure unit quaternion which together satisfy  $-Q=  qQ\bar q$,  then $q$ is itself a pure unit quaternion,  and $q$ and $Q$ are perpendicular; i.e., $\Real(qQ)=0$.
\item If $Q_1,Q_2\in C(\bbi)$ and $t_1,t_2\in \RR$, then $$\Real(e^{t_1Q_1}e^{t_2Q_2})=\cos t_1\cos t_2-\sin t_1\sin t_2 \cos\nu$$
where $\nu$ denotes  angle (in $su(2)=\RR^3$) between $Q_1$ and $Q_2$.
\item If $e^{tQ_1}$ and $e^{sQ_2}$ commute, with $Q_i\in C(\bbi)$ and $e^{tQ_1}, e^{sQ_2}\ne \pm 1$, then $Q_1=\pm Q_2$.

\end{enumerate}
\qed\end{prop}

\section{The pillowcase}

 In this section we introduce the pillowcase as a quotient space of $\RR^2$, and describe how it arises as an $SU(2)$ character variety in two different ways.
\subsection{The pillowcase as quotient of $\RR^2$}
Let $G$ denote the split extension of $\ZZ^2$  by $\ZZ/2$ where the generator $\tau\in\ZZ/2$ acts on $(m,n)\in \ZZ^2$ by $\tau\cdot(m,n)=-(m,n)$. Then $G$ acts  affinely on the plane $\RR^2$ by
$$(m,n)\cdot(x,y)=(x+2\pi m,y+2\pi n),\ \ \tau\cdot(x,y)=(-x,-y).$$
The quotient $\RR^2/G$ is called {\em the pillowcase}. 
The quotient map 
\begin{equation}
\label{pillowcasemap}
\RR^{2}\to \RR^{2}/G
\end{equation}
is a branched cover, branched over four points with preimage  the lattice $(\ZZ\pi)^2\subset\RR^2$ of points with non-trivial isotropy.  The pillowcase is homeomorphic to a 2-sphere.  

  A fundamental domain for the action is the rectangle $[0,\pi]\times [0,2\pi]$ (see Figure \ref{pillow}), and the identifications along its boundary are
 \begin{equation}\label{pillowcaserel}
(x,0)\sim (x,2\pi),~(0,y)\sim(0,2\pi-y),\text{~ and~}(\pi,y)\sim(\pi,2\pi-y).
\end{equation}
Hence the moniker ``pillowcase.''  Taking the quotient in two steps, first by $\ZZ^2$ and then by $\ZZ/2$, exhibits the pillowcase as the quotient of the torus by the hyperelliptic involution.

   \begin{figure}
   \centering
\def\svgwidth{3.5in}
 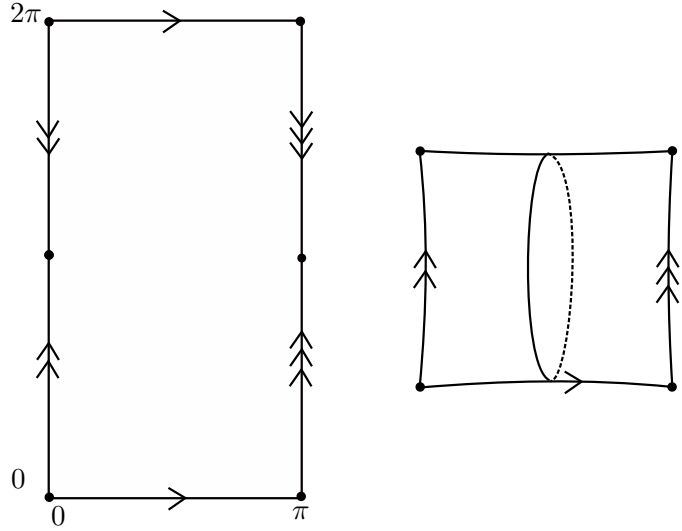
 \caption{The pillowcase.  On the left is a fundamental domain for the action of $\ZZ^2\rtimes \ZZ/2$ on $\RR^2$.  On the right is the ``pillowcase" obtained by performing the identifications on the left.  Topologically, the pillowcase is a 2-sphere. \label{pillow}}
\end{figure}

We will often describe a curve in the pillowcase in terms of a lift to $\RR^2$. A curve may have different lifts,  for example the straight line segments $t\mapsto (2t,3t)$ and $t\mapsto (-2t, -3t+4\pi), ~t\in [0,\frac{\pi}{2}]$ define the same curve in the pillowcase.

\subsection{The pillowcase as the $SU(2)$ character variety of the torus}
The $SU(2)$ character variety of the 2-dimensional torus $T$ is the space of conjugacy classes of representations:
$$\chi(T)=\{\rho:\pi_1(T)\to SU(2)\}/_\text{\small conjugation}.$$
If $\mu,\lambda\in \pi_1(T)$ denote generators, then to any pair $(x,y)\in \RR^2$ of real numbers one can assign the conjugacy class of the representation
$$\mu\mapsto e^{x \bbi   },\lambda\mapsto e^{y \bbi  }$$ in $\chi(T)$. The resulting map $\RR^2\to \chi(T)$ factors through the branched cover of Equation (\ref{pillowcasemap}) and induces a homeomorphism of  the pillowcase with $\chi(T)$.  Note that the identification depends on the choice of generators $\mu,\lambda$.

The representations which send $\mu$ to a traceless matrix, i.e. to $C(\bbi)$, correspond exactly to the line $x=\frac{\pi}{2}$, since $e^{\pi \bbi  /2}=\bbi$.  The line $\{x=\frac{\pi}{2}\}\subset \RR^2$ is sent to the circle in the pillowcase:
\begin{equation}
\label{tlcircle}
S(\bbi)=\{ \rho:\pi_1(T)\to SU(2)~|~\rho(\mu)=\bbi, \rho(\lambda)= e^{ y\bbi }\}
\end{equation}

(For more details, see e.g.~\cite{Kirk-Klassen-CS1}.)

 \subsection{The  pillowcase as  traceless representations  of $S^2$ with four marked points}   Consider  a $2$-sphere with four marked points, labelled $a,b,c,$ and $d$, respectively, so that the fundamental group $\pi_1(S^2\setminus \{a,b,c,d\})$ is presented (by abuse of notation) as
 $$\pi_1(S^2\setminus\{a,b,c,d\})=\langle a,b,c,d~|~ ba=cd\rangle$$
as indicated in Figure \ref{fig1.1}.

We denote by $R(S^2,\{a,b,c,d\})$ the space of conjugacy classes of homomorphisms  which take the loops at each puncture to a traceless quaternion: 
 $$R(S^2,\{a,b,c,d\})=\{\rho:\langle a,b,c,d~|~ ba=cd\rangle\to SU(2)~|~
 \rho(a), \rho(b), \rho(c),\rho(d)\in C(\bbi)\}/_\text{\small conjugation}.$$

\begin{prop} \label{s2prop}  There is a surjective quotient map
$$\psi:\RR^2\to R(S^2,\{a,b,c,d\})$$ 
given by 
$$\psi(\gamma, \theta): a\mapsto \bbi,~ b\mapsto e^{\gamma\bbk}\bbi,~ c\mapsto e^{\theta\bbk}\bbi, ~d\mapsto e^{(\theta-\gamma)\bbk}\bbi.$$ 
The map $\psi$   factors through the branched cover of Equation (\ref{pillowcasemap}) and induces a homeomorphism of the pillowcase with $R(S^2,\{a,b,c,d\})$.
The  four corner points  correspond  to reducible non-central representations, and all other points correspond  to  irreducible representations.
\end{prop}
 
   \begin{figure}   \centering
\def\svgwidth{1.8in}
 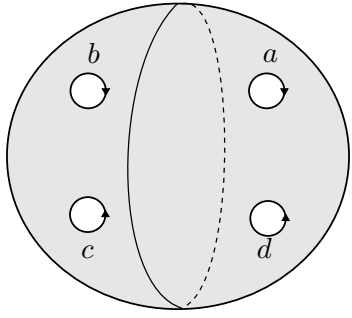

 \caption{\label{fig1.1}A four-punctured sphere with boundary curves added that generate its fundamental group (after choosing a basepoint near the barycenter of the curves  and arcs to the boundary).  }
\end{figure}
 
\begin{proof}  Given any $(\gamma,\theta)\in \RR^{2}$, the assignment $a\mapsto \bbi,~ b\mapsto e^{\gamma\bbk}\bbi,~ c\mapsto e^{\theta\bbk}\bbi, ~d\mapsto e^{(\theta-\gamma)\bbk}\bbi$   satisfies $ba=cd$, and hence $\psi$ maps into $R(S^{2}, \{a,b,c,d\})$.

 Any traceless representation $\rho:\pi_1(S^2\setminus\{a,b,c,d\})\to SU(2)$ can be conjugated so that $\rho(a)=\bbi$.  
Since $e^{-\frac{t}{2} \bbi} (e^{t\bbi} \bbj) e^{\frac{t}{2}\bbi }=\bbj,$ $\rho$ can be further conjugated (fixing $\rho(a)$) so that $\rho(b)=\cos(\gamma)\bbi+\sin (\gamma)\bbj=e^{\gamma\bbk}\bbi$ for some $\gamma\in [0,\pi]$.  

Suppose that $\rho(c)=c_1\bbi + c_2\bbj + c_3\bbk$.  Since $q\in SU(2)$ lies in $C(\bbi)$ if and only if $\Real(q)=0$, the equation 
$$0=\Real(\rho(d))=\Real(\rho(c^{-1}ba))=\Real(\rho(c^{-1})(e^{\gamma\bbk}\bbi)\bbi)
=-c_3\sin\gamma$$
implies that either $\sin\gamma=0$ or else $c_3=0$. 
If $\sin\gamma=0$, then $\rho(a)=\bbi$ and $\rho(b)=\pm \bbi$, and hence $\rho$ may be further conjugated by $ e^{-\frac{t}{2} \bbi}$, fixing $\rho(a)$ and $\rho(b)$, so that $c_3=0$.  So  we may assume by conjugating  that $c_{3}=0$ in either case. Hence $\rho(c)=c_1\bbi+ c_2\bbj=e^{\theta\bbk}\bbi$ for some $\theta\in[0,2\pi]$.

Thus the map $\psi:\RR^{2}\to R(S^{2}, \{a,b,c,d\})$ is onto. It is easy to check that $\psi(\gamma_1, \theta_1 )=\psi(\gamma_2, \theta_2 )$ if and only if $(\gamma_{2},\theta_{2})=\pm (\gamma_{1},\theta_{1})+(2\pi m,2\pi n)$, 
so that $\psi$ passes to a homeomorphism $\RR^{2}/G\to 
R(S^{2}, \{a,b,c,d\})$.
\end{proof}

\section{Representation spaces of knots}

 To a knot $K$ in a 3-manifold $Y$ with tubular neighborhood $N(K)$ we will assign several spaces of conjugacy classes of $SU(2)$ representations. The notation ``$\chi$'' will be used when we consider all conjugacy classes of representations of $\pi_1(Y\setminus N(K))$, and ``$R$'' will be used when we restrict to those representations which send the meridians of $K$ into the conjugacy class of pure unit quaternions $C(\bbi)$. We call such representations   {\em traceless} since $C(\bbi)$ corresponds to matrices in $SU(2)$ of trace zero.

\subsection{All representations}
 First, let $\chi(Y, K)$ denote the space of conjugacy classes of all representations of $\pi_1(Y\setminus N(K))$ into $ SU(2)$:
 \begin{equation}\label{chi}
\chi(Y, K)=\{\rho:\pi_1(Y\setminus N(K))\to SU(2)\}/_\text{\small conjugation}.
\end{equation}

\noindent  Restricting to the boundary  of $N(K)$ yields a  map$$
\chi(Y,K)\to \chi(\partial N(K)).$$
Choosing generators $\mu_K,\lambda_K$ for $\pi_1(\partial N(K))$ gives an identification of $\chi(\partial N(K))$ with the pillowcase. In particular, if $Y$ is a homology sphere we take the generators
 $\mu_K$ to be the canonical (isotopy class of) meridian and $\lambda_K$ the longitude arising as the boundary of an oriented Seifert surface.

If $Z$ is any 3-manifold containing a torus $T=S^1\times S^1\subset Z$,  restricting representations    yields a  map from the character variety $\chi(Z)$ of (all) $SU(2)$ representations of $\pi_1(Z)$  to the   pillowcase:
   \begin{equation}
\label{regres}
\chi(Z)\to \chi(T).
\end{equation}

 \subsection{Traceless representations} 
Next, denote by 
$R(Y, K)\subset\chi(Y, K)$
the subset  of  conjugacy classes of representations sending every meridian of $K$ into $C(\bbi)$.  Note that all meridians of $K$ are conjugate since we are assuming that $K$ is a knot rather than a link,  and so any representation sending a particular meridian $\mu_K\in \pi_1(Y\setminus N(K))$ into $C(\bbi)$ takes all meridians into $C(\bbi)$. Hence 
 \begin{equation}
\label{R}
R(Y,K)=\{\rho:\pi_1(Y\setminus N(K))\to SU(2)~|~ \rho(\mu_K)\in C(\bbi)\}/_\text{\small conjugation}.
\end{equation}
Given a pair $\mu_K,\lambda_K$ of generators of $\pi_1(\partial N(K))$ with $\mu_K$ a meridian, then   $R(Y,K)$ can be described as the preimage under the restriction $\chi(Y,K)\to \chi(\partial N(K))$ of the circle  $S(\bbi)$ of Equation (\ref{tlcircle}).

\medskip

If $S^2\subset Y$ is an embedded 2-sphere intersecting $K$ transversely in four points $a,b,c,d$,  we can restrict representations to $S^2$ to obtain a map:
   \begin{equation}
\label{tlres}
R(Y,K)\to R(S^2,\{a,b,c,d\})
\end{equation}
which we consider as the traceless analogue of (\ref{regres}).
\bigskip

\subsection{Adding an earring to avoid reducibles}

The spaces $\chi(Y,K)$ and $R(Y,K)$ are stratified according to the three possible stabilizers   of the constituent $SU(2)$  representations:  $\{\pm 1\}, ~S^1,$ or $SU(2)$.  To avoid singularities in moduli spaces arising from reducible connections, Kronheimer-Mrowka \cite{KM-khovanov}  introduce an auxiliary construction which ensures that only the center $\{\pm 1\}$ appears as a stabilizer (actually their construction requires the use of connections in a non-trivial $SO(3)$  bundle, with $w_2$ dual to the arc $W$ described below).   Let us recall their construction.

Fix a base point on $K$ and identify $N(K)$ with the unit normal (disk) bundle of $K$.  Let $H$ be the circle of radius $\tfrac{1}{2}$ in the normal disk fiber containing the base point. Denote the link $K\cup H$ by $K^\nat$.  Let  $W\subset Y$ be a radial arc in this normal disk  connecting $K$ and $H$.    
 See Figure \ref{natural}.

The boundary of a small tubular neighborhood of $H$ is a torus which is punctured once  by the arc $W$. Thus for the appropriate choices of basings and orientation of loops, the meridian $\mu_{H}$ and longitude $\lambda_{H}$ of this torus satisfy $[\lambda_{H}, \mu_{H}]=\mu_{W}$. Moreover, since $H$ is a small circle linking $K$, the longitude of $H$ equals the meridian of $K$, i.e. $\lambda_{K}=\mu_{H}$.  Thus  $[\mu_{K},\mu_{H}]=\mu_{W}$ in $\pi_1(Y\setminus (K^\nat\cup W))$.

Denote by $R^\nat(Y,K)$ the space of conjugacy classes of representations $\rho:\pi_1(Y\setminus (K^\nat\cup W))\to SU(2)$
which send the meridians  $\mu_K, \mu_H$ of $K$ and $H$ to $C(\bbi)$ and the meridian $\mu_W$ of the arc $W$ to $-1$:
\begin{equation}
\label{Rnat}
R^\nat(Y,K)=\{\rho:\pi_1(Y\setminus (K^\nat\cup W))\to SU(2)~|~
\rho(\mu_K),\rho(\mu_H)\in C(\bbi), ~\rho(\mu_W)=-1\}/_\text{\small conj}.
\end{equation}

\medskip

When $Y=S^3$ (or any homology sphere), there is a distinguished conjugacy class  $\alpha\in R^\nat(Y,K)$ which is characterized completely by the requirement that its restriction to   the complement of a (large) tubular neighborhood $N(K)$ of $K$  containing $H\cup W$ is abelian. The  representation $\alpha $ can be uniquely conjugated to satisfy:
\begin{equation}\label{alphadef}
\alpha(\mu_K)=\bbi, ~\alpha(\mu_H)=\bbj,~ \alpha(\mu_W)=-1, ~\text{and ~}\alpha(\lambda_K)=1.
\end{equation}
Since the set of all   conjugates of $\mu_K$ by loops in $Y\setminus N(K)$ generate $\pi_1(Y\setminus N(K))$, $\alpha$ sends each of these conjugates to $\bbi$.  The representation $\alpha$ and its restriction to $\pi_1(Y\setminus N(K))$, which will be denoted  $\alpha_{\pi/2}$,   appear frequently below.

\medskip

If $Z\subset Y$ is a codimension zero submanifold which contains $K^\nat\cup W$, then there is a restriction map
$R^\nat(Y,K)\to R^\nat(Z,K)$. For example, one can take $Z=N(K)$, a tubular neighborhood of $K$ large enough to contain $H \cup W$   to get
 $R^\nat(N(K), K)$.
 
\begin{prop}\label{nbd} 
 The space $R^\nat(N(K), K)$ is homeomorphic to a circle, and every representation in $R^\nat(N(K), K)$ is non-abelian. Moreover, the restriction map $$R^\nat(N(K), K)\to \chi(\partial N(K))$$ is injective, with image the vertical circle $S(\bbi) $ of Equation (\ref{tlcircle}).
\end{prop}
\begin{proof}
Any representation in $R^\nat(N(K),K)$ may be conjugated so that $\rho(\mu_K)=\bbi$. 
Since $\rho(\lambda_K)$ commutes with $\rho(\mu_K)$ it follows that $\rho(\lambda_K)=e^{\ell\bbi}$ for some $\ell$.

The relation $[\rho(\mu_K),\rho(\mu_H)]=\rho(\mu_W)=-1$ implies that $\rho$ is non-abelian and that $\rho(\mu_H)$ is perpendicular to $\bbi$ (Proposition \ref{basic}).
Thus $\rho(\mu_H)=\cos\theta\bbj +\sin\theta\bbk=e^{\theta\bbi}\bbj$ for some  $\theta$. 
Further conjugation by $e^{-\theta\bbi/2}$ fixes $\rho(\mu_K)=\bbi$ and $\rho(\lambda_K)=e^{\ell\bbi}$ but rotates so that 
$\rho(\mu_H)=\bbj$.

Conversely, any choice of $e^{\ell\bbi}\in S^1$ defines a unique conjugacy class in $R^\nat(N(K),K)$ by sending $\mu_K$ to $\bbi$, $\lambda_K$ to $e^{\ell\bbi}$, $\mu_H$ to $\bbj$, and $\mu_W$ to $-1$.
 \end{proof}

  Note that we can think of $K^\nat$ as obtained from $K$ by forming the connected sum of $K$ with one component of a Hopf link whose components are spanned by the arc $W$. Kronheimer-Mrowka also introduce a variant of this construction, which they denote by $K^\sharp$, obtained by taking the disjoint union of $K$ with a Hopf link rather than the connected sum. This leads to a different representation space $R^\sharp(Y,K)$; see Section \ref{unreduced case} below.

\subsection{Gluing representations and the relation between $R^\nat(Y,K)$ and $R(Y,K)$}\label{gluing}

The spaces $R(Y,K)$ are better known than $R^\nat(Y,K)$.  The relationship between the two is that $R^\nat(Y,K)$ maps to  $R(Y,K)$ with fibers either a circle or a point depending on the reducibility of the image. We make this precise in Proposition \ref{circle or point} below.

\medskip
We first present a  folklore description of the character variety of a free product with amalgamation in terms of the character varieties of its pieces. 
To describe it, fix a compact Lie group $G$ and (for this section only) let 
$\ccH(-)$ denote the functor which takes a discrete group  $H$ to  the space $\Hom(H,G)$ and let $\chi(-)$ denote the functor that takes $H$ to $\Hom(H,G)/_\text{\small conjugation}$. 

Suppose that the finitely presented group $H$ is  given a decomposition  as a free product with amalgamation
$H=A*_S B$. Then clearly
$$\ccH(H)\cong \ccH(A)\times_{\ccH(S)}\ccH(B):=\{(\rho_A,\rho_B)\in \ccH(A)\times \ccH(B)~|~ \rho_A|_S=\rho_B|_S\}.$$

\noindent The Lie group $G$ acts diagonally by conjugation on $\ccH(H)$ with quotient $\chi(H)$, and similarly for the subgroups $A,B$, and $S$. There is a surjective map 
$$\ccH(A)\times_{\ccH(S)}\ccH(B)\to
\chi(A)\times_{\chi(S)}\chi(B):=\{(c_A,c_B)\in \chi(A)\times \chi(B)~|~ c_A|_S=c_B|_S\}$$
which factors through $\chi(H)$.
These form a diagram: 
\[
 \begin{diagram}
\node{\ccH(A)\times_{\ccH(S)}\ccH(B)}\arrow[2]{e,t}{p_1}\arrow{se,l}{p_3}\node[2]{\chi(H)}\arrow{sw,r}{p_2}\\
\node[2]{\chi(A)\times_{\chi(S)}\chi(B)}
\end{diagram}
\]

Let Stab$(\rho)\subset G$ denote the centralizer of the image of a representation $\rho$,
\begin{equation}\label{stabilizer}
\text{Stab}(\rho)=\{g\in G~|~ \rho(x)=g\rho(x)g^{-1}\text{~for all~}x\}.
\end{equation}

\begin{lem}\label{double}
 Let $(c_A,c_B)\in \chi(A)\times_{\chi(S)}\chi(B)$ and choose $(\rho_A, \rho_B)\in p_3^{-1}(c_A,c_B)$ (so that $\rho_A|_S=\rho_B|_S$).  Then the fiber $p_2^{-1}(c_A,c_B)$ is homeomorphic to the double coset space
 $$ \text{Stab}(\rho_A)\backslash \text{Stab}(\rho_A|_S)/\text{Stab}(\rho_B).$$

\end{lem}

\noindent{\bf Remark.} Lemma \ref{double}  can be applied in  situations when we consider representation spaces and character varieties that place conjugation-invariant conditions on the value that the  representations take on specified elements in the groups $H$, $A$, $B$, and $S$.  For our purposes we will be interested in the traceless representations which arise in the definition of  $R$ and $R^\nat$.  

\medskip
\noindent{\em Proof of Lemma \ref{double}.}
Denote by $\rho_S$ the restriction of $\rho_A$ to $S$. We first identify the fiber   $p_3^{-1}(c_A, c_B)$ with  a quotient of $G\times \text{Stab}(\rho_S)$.  If $(\rho_A',\rho_B')\in p_3^{-1}(c_A, c_B)$  then there are $g_1,g_2\in G$ so that $(\rho_A',\rho_B')=(g_1\rho_Ag_1^{-1}, g_2 \rho_B g_2^{-1}).$  Since $\rho_A'$ and $\rho_B'$ agree on $S$, (as do $\rho_A$ and $\rho_B$) it follows that $g_1^{-1}g_2\in \text{Stab}(\rho_S)$.  Thus the map $G\times \text{Stab}(\rho_S)\to 
p_3^{-1}(c_A, c_B)$ taking $(g,t)$ to $(g\rho_Ag^{-1}, gt\rho_B (gt)^{-1})$ is surjective.

Since  $p_1$ is surjective, it maps $p_3^{-1}(c_A, c_B)$ surjectively to $p_2^{-1}(c_A, c_B)$. Suppose that $(g,t), (g',t')\in G\times \text{Stab}(\rho_S)$. Then $$p_1(g'\rho_Ag'^{-1}, g't'\rho_B(g't')^{-1})=p_1(g\rho_Ag^{-1}, gt\rho_B (gt)^{-1})$$ if and only if there is an $h\in G$ so that 
$$(g'\rho_Ag'^{-1}, g't'\rho_B(g't')^{-1})=(hg\rho_A(hg)^{-1}, hgt\rho_B (hgt)^{-1}).$$  This is equivalent to  $g'^{-1}hg\in  \text{Stab}(\rho_A)$ and $(g't')^{-1}hgt\in  \text{Stab}(\rho_B)$. Writing $a=g'^{-1}hg$ and $b=(g't')^{-1}hgt$ we see that 
$h=g'ag^{-1}$ and $ t'^{-1}at=b$.  It follows that 
$(g,t)$ and $ (g',t')$ correspond to the same element of $p_2^{-1}(c_A,c_B)$ if and only if there exist $a\in \text{Stab}(\rho_A)$ and $b\in \text{Stab}(\rho_B)$ so that $t'=atb^{-1}$.
\qed

 \bigskip

We  now use Lemma \ref{double}   to compare $R(Y,K)$ and $R^\nat(Y,K)$.  To do this,  
write 
\begin{equation}\label{decomptorus}
(Y,K^\nat\cup W)=(Y\setminus N(K),\emptyset)\cup_{\partial N(K)} (N(K), K^\nat\cup W).
\end{equation}

Suppose that $\rho_A\in \chi(Y,K)$ and $\rho_B\in R^\nat(N(K),K)$ map to the same point $\rho_S\in \chi(\partial N(K))$. Conjugate $\rho_B$ so that the restrictions of $\rho_A$ and $\rho_B$ to $\partial N(K)$ agree. Proposition \ref{nbd} shows that $\rho_B\in R^\nat(N(K),K)$ has non-abelian image.  Moreover, since $\rho_B(\mu_K)\in C(\bbi)$, the restriction $\rho_A$ to $Y\setminus N(K)$ of any representation $\rho\in R^\nat(Y,K)$ lies in $R(Y,K)\subset \chi(Y,K)$. Furthermore, the restriction $\rho_S$ to the separating torus $\partial N(K)$ is abelian but non-central, since it sends the meridian $\mu_K$ into $C(\bbi)$.  Thus $\text{Stab}(\rho_S)\cong S^1$.

 From 
Lemma \ref{double} we conclude that fiber of  the restriction map 
\begin{equation}
\label{reseq1}
R^\nat(Y,K)\to R(Y,K)\times_{\chi(\partial N(K))}R^\nat(N(K),K)
\end{equation}
over $\rho_A*\rho_B$ is $$\text{Stab}(\rho_A)\backslash S^1/ \{\pm1\}.$$
This is a single point if the restriction $\rho_A$ of $\rho$ to $Y\setminus N(K)$ has abelian image, and a circle if the restriction has non-abelian image.

Proposition \ref{nbd} asserts that the restriction map $R^\nat(N(K),K)\to\chi(\partial N(K))$ is injective. This allows us to identify the restriction map  of Equation (\ref{reseq1}) with the surjective map
\begin{equation}
\label{reseq2}
R^\nat(Y,K)\to R(Y,K).
\end{equation}
In summary, we have the following.

\begin{prop} \label{circle or point} Every representation in $R^\nat(Y,K)$ has non-abelian image. The forgetful map
$R^\nat(Y,K)\to R(Y,K)$ is a surjection. The fiber 
over a conjugacy class $\rho$ is either a circle or a point, depending on whether 
$\rho$ has non-abelian or abelian image.  \qed
 \end{prop}

Lemma \ref{double} is also useful when studying $R^\nat(Y,K)$ via a  decomposition of the pair $(Y,K)$ along a four-punctured 2-sphere. Suppose that 
$Y=Y_1\cup_{S^2}Y_2$ is a decomposition of $Y$ along a 2-sphere $S^2\subset Y$ which intersects $K$ in four points $a,b,c,d$.  We assume that $H\cup W$ lies in the interior of $Y_2$.  Then Lemma \ref{double} allows us to identify $R^\nat(Y,K)$ as a fiber product of $R(Y_1, K_1)$ and $R^\nat(Y_2,K_2)$ over the second  pillowcase $R(S^2,\{a,b,c,d\})$ (where each $K_i=K\cap Y_i$ is a union of two arcs). Precisely, there is a restriction map 
\begin{equation}
\label{decompS2}
R^\nat(Y,K)\to R(Y_1,K_1)\times_{R(S^2,\{a,b,c,d\})}R^\nat(Y_2,K_2)
\end{equation}
whose fiber over $\rho_1*\rho_2$ is 
$$\Stab(\rho_1)\backslash \Stab(\rho_{S^2})/\Stab(\rho_2).$$
Now $\Stab(\rho_2)=\pm 1$, since $H\cup W$ is contained in $Y_2$. The stabilizer of any non-singular point in the pillowcase $ R(S^2,\{a,b,c,d\})$ is just the center $\pm 1$, so the fiber  of the map in Equation  (\ref{decompS2}) over a pair $(\rho_1,\rho_2) $ is a single point  if the restriction to $S^2$ is not one of the four 
abelian conjugacy classes in $ R(S^2,\{a,b,c,d\})$.

 However, at one of the four abelian conjugacy classes, the stabilizer  is $S^1$, and hence the fiber of the map in Equation  (\ref{decompS2}) is 
$$\Stab(\rho_1)\backslash S^1,$$
which is a circle if $\rho_1$ is non-abelian and a single point if $\rho_1$ is abelian.
In summary:

\begin{prop} \label{natcase} Suppose that $S^2\subset Y$ is a separating 2-sphere meeting the knot $K$ transversely in four points $a,b,c,d$  and that $H\cup W$ lies in $Y_2$. Restricting to the two pieces in the decomposition $Y=Y_1\cup_{S^2} Y_2$ gives a surjection $$R^\nat(Y,K)\to R(Y_1,K_1)\times_{R(S^2,\{a,b,c,d\})}R^\nat(Y_2,K_2)$$ whose fibers are single points, with the exception of the fibers over the non-abelian  representations in $R(Y_1,K_1)$ which  restrict to one of the four abelian (corner) points in  $R(S^2,\{a,b,c,d\})$.  The fiber above these latter points  is a circle.
   \qed
 \end{prop}
In the examples below we will take $Y_2$ to be a 3-ball intersecting $K$ in a pair of unknotted arcs. We will explain, for torus and 2-bridge knots,  how to choose the 3-ball judiciously so that the restriction of every non-abelian representation in $R(Y_1,K_1)$  to $R(S^2,\{a,b,c,d\})$ avoids the corner points. The map of Proposition \ref{natcase} then has point fibers. 
Then a perturbation will be applied to make $R^\nat(Y_2,K_2)$ generic, in the sense that the fiber product in Proposition \ref{natcase} has finitely many points, corresponding to a finite intersection of $R(Y_1,K_1)$ and $R^\nat(Y_2,K_2)$ in $R(S^2\{a,b,c,d\})$.  These points will provide a finite generating set for the reduced singular instanton chain complex.

 \section{Knots with simple representation varieties}  \label{simple}

Much is known about the spaces $\chi(Y,K)$ for various $(Y,K)$, starting with Klassen's influential article \cite{Klassen}.   
The image of $\chi(S^3,K)$   in the pillowcase $\chi(\partial N(K))$ is also a well studied space: for example, it appears as part of the real locus of the $A$-polynomial of \cite{CCGLS}. The identification  of the algebraic count of the intersections of this image with the circle $S(\bbi) $ of Equation (\ref{tlcircle}), corresponding to the line $\{x=\frac{\pi}{2}\}$ (and more generally circles corresponding to the line $\{x=\theta\}$ for $\theta\in [0,\pi]$) with   Levine-Tristram knot signatures is explained in \cite{Lin, Herald2, Heusener}.

\medskip
\begin{df} \label{SRV}
We say  a knot $K$ in $S^3$ has a {\em simple representation variety}, or   {\em $\chi(S^3,K)$ is generic}  if  $\chi(S^3,K)$ is homeomorphic to a 1-complex made out of 
\begin{enumerate}
\item an arc of abelian representations, parameterized by 
$$t\in [0,\pi]\mapsto \alpha_t:\pi_1(S^3\setminus N(K))\to H_1(S^3\setminus N(K))\to SU(2),~ \alpha_t(\mu_K)=e^{t\bbi}.$$
\item a finite number of smooth arcs of  representations 
$$n_i:[0,1]\to \Hom(\pi_1(S^3\setminus N(K)),SU(2)), ~i=1,\cdots , k$$ whose interior points are non-abelian and whose endpoints $n_i(0), n_i(1)$  equal $\alpha_{s_{i,0}}$ and $\alpha_{s_{i,1}}$ for a pairwise distinct 
set of points $s_{1,0}, s_{1,1}, s_{2,0}, s_{2,1}, \cdots, s_{k,0}, s_{k,1}\in [0,\pi]$. 
\item a finite number of disjoint smooth circles 
$$c_i:S^1\subset \chi(S^3,K),~i=1,\cdots, m$$
of non-abelian representations disjoint from the arcs $\alpha_t$ and $n_i$. 
\item The restriction map $\chi(S^3,K)\to\chi(\partial N(K))$ restricts to an immersion on each arc   $\alpha_t, n_i$ and each circle $c_i$. This immersion is transverse to the circle $S(\bbi) $ of  representations which are traceless on the meridian (Equation (\ref{tlcircle})).
\end{enumerate}
 
\end{df}

 The points $s_{1,0}, s_{1,1},   \cdots, s_{k,0}, s_{k,1}$ are called {\em bifurcation points}, since they correspond to places where the irreducible representations bifurcate from the abelian representations.  This implies that the roots of the Alexander polynomial of $K$ which lie on the unit circle are those of the form $e^{2s_{i,j}\bbi}$, and are simple and distinct \cite{Klassen,Herald2}.

Not all knots have simple representation varieties, but the results of \cite{herald,Herald2} imply that  perturbation data (as explained below)
$\pi$ can be found so that the perturbed moduli space $\chi_{\pi}(S^3,K)$ has  the properties listed above. For small perturbations, the distinct points $s_{1,0}, s_{1,1},\cdots, s_{k,0}, s_{k,1}$ are close, but not necessarily equal to, the  set of roots of the Alexander polynomial.

\subsection{Torus knots }\label{toruschi} All torus knots have simple representation varieties. In fact $\chi(S^3,T_{p,q})$ is a union of an arc $\alpha_t$ of abelian representations as described above, and $\frac{(p-1)(q-1)}{2}$   arcs of non-abelian representations that limit to abelian representations at roots of the Alexander polynomial  \cite{Klassen}.  We review some of the details.

The $(p,q)$ torus knot  group has presentation $\langle x,y~| x^p=y^q\rangle$. In this presentation $\mu_K=x^sy^r$ and $\lambda_K=x^p(x^sy^r)^{-pq}$ where $pr+qs=1$ (see e.g.~\cite[Proposition 3.28]{burde}). The arc of conjugacy classes of abelian representations is parameterized by $\alpha_t(\mu_K)=e^{t\bbi}, ~t\in[0,\pi]$.  In terms of the generators $x$ and $y$ we have $\alpha_t(x)=e^{qt\bbi}$ and $\alpha_t(y)=e^{pt\bbi}$.

Since $x^p=y^q$ is central,  any non-abelian representation $\rho$ will send $x^p=y^q$ to $\pm1$. Hence $x$ is sent to a $p$th root of $\pm1$ in $SU(2)$ and $y$ is sent to a $q$th root of $\pm1$.
If $\rho$ is conjugated so that $\rho(\mu_{K})=e^{m\bbi}$ for some $m\in \RR$, then $\rho(\lambda_K)=\rho(x^{p}\mu_{K}^{-pq})=\pm e^{-pqm\bbi}$. Hence the image of the non-abelian part of $\chi(S^3,K)$ in the pillowcase lies in the straight lines $\ell=-pqm +k\pi$ (via the quotient $\RR^2\to \chi(T)$, taking $(m,\ell)$ to the conjugacy class of $\mu_K\mapsto e^{m\bbi}, \lambda_K\mapsto e^{\ell\bbi}$).  In fact the images in the pillowcase are embedded arcs that start and end at $\alpha_{s_i}$ where $\Delta_K(e^{2s_i\bbi})=0$ (see Theorem 19 of \cite{Klassen} and the discussion which follows its proof). 

 Each arc  of non-abelian representations is completely determined by the choice of $p$th and $q$th root of $\pm1$.   Explicitly, given a pair $a,b$ of integers of the same parity, then the assignment
\begin{equation}
\label{torusarcs}
x\mapsto e^{a\pi\bbi/p},~ y\mapsto \cos(b\pi/q) + \sin(b\pi/q)(\cos(u)\bbi+\sin(u)\bbj),~u\in[0,\pi]
\end{equation}
defines an arc of representations whose projection to the space of conjugacy classes is 1-1. Conversely, given any representation $\rho$, 
 $\rho(x)$ can be conjugated to $e^{a\pi\bbi/p}$ and $\rho(y)$ can be conjugated to $e^{b\pi\bbi/q}$ for some choice of integers of the same parity, and some choice of conjugating elements.  It follows that  any representation can be  conjugated to a unique point  on the corresponding path of Equation (\ref{torusarcs}): we first conjugate $\rho$ so that $\rho(x)$ has the desired image, and then further conjugate by an element in the circle through  $\rho(x)$ so that $y$ has the form of \eqref{torusarcs}. This sets up a bijection between arcs of non-abelian representations and pairs $(a,b)$ of the same parity satisfying $0<a<p$ and $0<b<q$.  For more details see \cite[Theorem 1]{Klassen}.

The endpoints of this arc are abelian, and hence lie on the arc $\alpha_{t}$.  These endpoints are determined by computing where they send $\mu_K$. When $u=0$
the pair $(\mu_K,\lambda_K)$ is sent to $(e^{(as/p + br/q)\pi\bbi}, e^{pr(a-b)\pi\bbi})$.
Hence the endpoint with $u=0$ equals $\alpha_{s_0}$ where $(s_0,0)$ and $((as/p + br/q)\pi, pr(a-b)\pi)$ in $\RR^2$ map to the same point in the pillowcase. Similarly one determines where the endpoint with $u=\pi$ is sent. 

These considerations suffice to show that all torus knots have simple representation varieties and to completely determine $\chi(S^3,K)$ and its image in the pillowcase.  We illustrate two examples. 

The simplest example is the trefoil knot, i.e. the $(2,3)$ torus knot. It has one arc of non-abelian representations, attached to the arc of abelian representations at the points $s_0=\tfrac{\pi}{6}$ and $s_1=\tfrac{5\pi}{6}$, corresponding to the fact that $\Delta_K(e^{2s_i})=0$, $i=0,1$. This arc of non-abelian representations is sent to an arc of slope $-6$ in the pillowcase.  Notice that the map $\chi(S^3,K)\to \chi(\partial N(K))$ fails to be injective at precisely the two points   in the pillowcase sent to the circle $S(\bbi) $ of Equation (\ref{tlcircle}).  In particular, $R(S^3,K)$ consists of exactly two points,  the abelian representation $\alpha_{\pi/2}$, and the non-abelian representation corresponding to $u=\pi/2$.

  \begin{figure}
\begin{center}
\def\svgwidth{2.5in}
 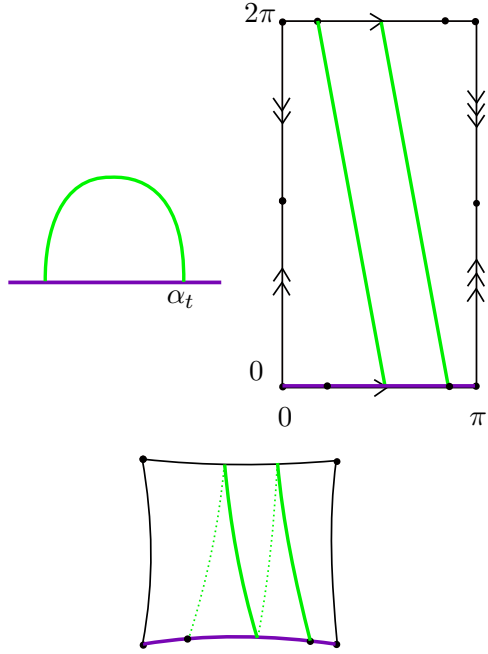
 \caption{$ \chi(S^3, T_{2,3})$\label{figtrefoil}}
\end{center}
\end{figure}

Next consider the $(3,5)$ torus knot. Its Alexander polynomial is
$\Delta_K(t)={t}^{8}-{t}^{7}+{t}^{5}-{t}^{4}+{t}^{3}-t+1$.  Thus $\Delta_K(e^{2s_i})=0$   
when $s_i=\alpha_{\pi/15}, \alpha_{2\pi/15}, \alpha_{4\pi/15}, \alpha_{7\pi/15}, \alpha_{8\pi/15}, \alpha_{11\pi/15},$
$ \alpha_{13\pi/15} ,$ and $ \alpha_{14\pi/15}$. These form the endpoints of the   four arcs of non-abelian representations, determined by the four possible choices $(a,b)=(1,1), (1,3), (2,2)$ and $(2, 4)$.  One calculates that the first arc has endpoints $\{\alpha_{\pi/15},\alpha_{11\pi/15}\}$, the second $\{\alpha_{7\pi/15},\alpha_{13\pi/15}\}$, the third $\{\alpha_{2\pi/15},\alpha_{8\pi/15}\}$, and the fourth $\{\alpha_{4\pi/15},\alpha_{14\pi/15}\}$. 
The restriction $\chi(S^3,K)\to \chi(\partial N(K))$ is far from injective, although it is injective when restricted to each arc of non-abelian representations. In this example, the space $R(S^3,K)$ consists of five points: the abelian $\alpha_{\pi/2}$, and the four midpoints of the non-abelian arcs, two sent to the point with ($\RR^2$) coordinates $(\tfrac{\pi}{2},-\frac{\pi}{2})$, and two sent to the point with coordinates  $(\tfrac{\pi}{2},\frac{\pi}{2})$. In particular, this implies that these traceless representations are not binary dihedral, since the longitude is not sent to $\pm 1$. We illustrate this in Figure \ref{figchi(3,5)}, where we have only drawn the image of the first arc in the pillowcase for clarity.

  \begin{figure}
\begin{center}
 \def\svgwidth{3.2in}
 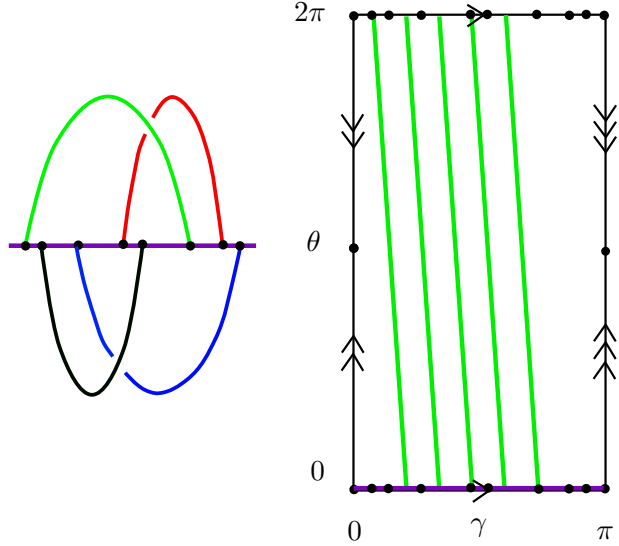
 \caption{$\chi(S^3, T_{3,5}) $\label{figchi(3,5)}}
\end{center}
\end{figure}

The representation variety $\chi(S^3,K)$ for $K$ a general $(p,q)$ torus knot consists of an arc $\alpha_t$ of abelian representations and  $\frac{(p-1)(q-1)}{2}$ arcs  of non-abelian representations with distinct endpoints $\alpha_{s_i}$ on $\alpha_t$, parameterized by pairs of integers $(a_{i},b_{i})$ in Equation (\ref{torusarcs}) above. 

For the purposes of describing the image of each arc in the pillowcase, it is more convenient to parameterize  the paths by their endpoints, rather than the pairs $(a_i
, b_{i})$. Thus  $\chi(S^3,K)$  for $K$ any $(p,q)$ torus knot and its image in the pillowcase $\chi(\partial N(K))$ is completely described by the data 
consisting of the pair $(p,q)$ and an additional $\frac{(p-1)(q-1)}{2}$ pairs
$(c_i,d_i), ~i=1,\cdots ,\frac{(p-1)(q-1)}{2},~ (c_i<d_i)$ of integers determined by the requirement that the endpoints of the $i$th arc of non-abelian representations are $\frac{c_i\pi}{pq}$ and $\frac{d_i\pi}{pq}$; these satisfy $\Delta_K(e^{ \frac{2c_i\pi}{pq}})=0= \Delta_K(e^{ \frac{2d_i\pi}{pq}})$. The pairs $(c_{i},d_{i})$ determine and are determined by the $(a_{i},b_{i})$, but the relationship is awkward to describe explicitly, and so we will use $(c_{i},d_{i})$ to prescribe each arc of non-abelians.

The subvariety $R(S^3,K)$ of $\chi(S^3,K)$ contains $|\sigma(K)|/2+1$ points where $\sigma(K)$ denotes the signature of $K$ \cite{Lin, Herald2, Heusener}. One of these points is $\alpha_{\pi/2}$, the rest correspond to intersections of some of the arcs of non-abelian representations with the circle $S(\bbi) $  of Equation (\ref{tlcircle}).

 The image of the arc of abelians $\alpha_t$ in the pillowcase has slope $0$, and each    arc of non-abelian representations maps to a line of slope $-pq$.  Hence, once we know their endpoints, we know how they map to the pillowcase. In particular, there is always one abelian traceless representation $\alpha_{\pi/2}$, and one non-abelian traceless representation for each pair 
$(c_i, d_i)$ as above satisfying $c_i<\frac{pq}{2}<d_i$.

For example, for the trefoil $T_{2,3}$, the data $(c_i, d_i)$ is just
 $  (1,5)$ and so $R(S^3, T_{2,3})$ has one abelian and one non-abelian traceless representations.  
For $K=T_{3,5}$, the data is $$(1,11), (7,13), (2,15), (4,14)$$ and so $R(S^3, T_{3,5})$ has one abelian and four non-abelian traceless representations, since each corresponding interval contains $\frac{15}{2}$. The data clearly determines the information in Figure \ref{figchi(3,5)}.

 In general, not every arc of non-abelian representations contains a point of $R(S^3,K)$: the inequality $|\sigma(K)|/2\leq (p-1)(q-1)/2$ may be strict.  For example, the data for the $(3,7)$-torus knot is $$
 (1,13), (11,17), (5,19), (2,16), (4,10), (8,20)$$
 and so the arcs determined by the pairs $(11,17)$ and $(4,10)$ do not contain traceless representations. This corresponds to the fact that 
$\sigma/2=4$.
Similarly, the data  for the $(4,9)$-torus knot is
$$(1,17), (15,33), (23,31), (7,25) , (2,34) , (14,22) , (6,30) , (10,26) , (19,35), (3,21) , (5,13), (11,29) 
$$
and 
$\sigma/2=8$.
It is known that the signature of a non-trivial  torus knot is always non-zero (see. e.g.~\cite{Kirk-Livingston-mutation}) from which it follows that $R(S^3,K)$ always contains at least one   point on a non-abelian arc. A much deeper result is the theorem of Kronheimer-Mrowka \cite{KM0} that every non-trivial knot in $S^3$ admits a non-abelian traceless representation.

Applying Proposition \ref{circle or point} we conclude that for $K$ the $(p,q)$ torus knot, $R^\nat(S^3,K)$ consists of one isolated point $\alpha$, corresponding to $\alpha_{\pi/2}$, and $|\sigma(K)|/2$ circles, one for each irreducible traceless representation of $\pi_1(S^3\setminus K)$.

\subsection{2-bridge knots}
We recall some facts about $\chi(S^3,K)$ when $K$ is a  2-bridge knot. Klassen \cite{Klassen} identified the spaces $\chi(S^3,K)$ for the twists knots.  Building on work of Riley \cite{Riley} (who considered  $SL(2,\CC)$ representations), Burde \cite{burde} determined $\chi(S^3,K)$ for all 2-bridge knots.  In contrast to torus knots, the image of $\chi(S^3,K)$ in the pillowcase is not given by linear equations, but rather by more complicated polynomial equations.   However, we will give   explicit (and linear)  equations which determine the {\em traceless} representation varieties $ R(S^3,K)$ and $R^\nat(S^3,K)$ for any 2-bridge knot in Section \ref{twobridgesection} below.

For the $m$-twist knot  (this is the 2-bridge knot corresponding to the continued fraction expansion $[1,1,m]$), $\chi(S^3,K)$ is a union of the arc $\alpha_t$ of abelian representations, $[\tfrac{m}{2}]$ circles of non-abelian representations, and, if $m$ is odd, one arc of non-abelian representations with endpoints on the arc $\alpha_t$ corresponding to the two roots of the Alexander polynomial on the unit circle.   Each circle contributes two points to $R(S^3,K)$, and the arc of non-abelian representations (when $m$ is odd) meets $R(S^3,K)$ in one point. As for all knots, the abelian representation $\alpha_{\pi/2}$ lies in $R(S^3,K)$.   

Thus,  for $K$ the $m$-twist knot, $R(S^3,K)$ consists of one abelian representation, and $2[\frac{m}{2}]=m$ non-abelian representations if $m$ is even, and 
$2[\frac{m}{2}]+1=m$ non-abelian  representations if $m$ is odd.  Using Proposition
\ref{circle or point} it follows that $R^\nat(S^3,K)$ consists of one isolated point $\alpha$ and $m$ circles.

For example, when $K$ is the  Figure 8 knot, corresponding to $m=2$, $\chi(S^3,K)$ consists of the arc $\alpha_t$ of abelian representations and a disjoint circle of non-abelian representations. A construction of this circle and arc, together with an explicit description of the image of $\chi(S^3,K)\to \chi(\partial N)$ can be found in Proposition 5.4 of \cite{Kirk-Klassen-CS1}.  From this one concludes that $R(S^3,K)$ consists of exactly three points: the arc $\alpha_t$ contributes the point $\alpha_{\pi/2}$ to $R(S^3,K)$, and the circle contributes two points to $R(S^3,K)$. This identifies $R^\nat(S^3,K)$ as the union of an isolated point $\alpha$ and two circles.

We illustrate the more complicated example of the $m$ twisted double of the unknot when $m=5$, this is  the knot $7_2$ in the knot tables. The space  $\chi(S^3,K)$ consists of two circles and one arc of non-abelians, and the arc $\alpha_t$ of abelians. 
The Alexander polynomial is $3-5t+3t^3$, with roots $(5\pm \sqrt{-11})/6\approx e^{(0.1)2\pi \bbi}$.  
The arc of non-abelians is embedded, but the two circles of non-abelians are immersed in the pillowcase with one transverse double point at $(m,\ell)=(\tfrac{\pi}{2}, 0)$.   The proof of these facts can be found in Burde's article \cite{burde}.

In particular, the space $R(S^3,K)$ contains six points: the abelian point $\alpha_{\pi/2}$, the midpoint of the arc of non-abelians, and two points on each of the two circles of non-abelians. Each of these six points is mapped to the point $(\tfrac{\pi}{2}, 0)$ in the pillowcase. The space $R^\nat(S^3,K)$ thus consists of one isolated point $\alpha$ corresponding to $\alpha_{\pi/2}$ via Proposition \ref{circle or point} and five circles corresponding to the five non-abelian representations in $R(S^3,K)$.

  \begin{figure}
\begin{center}
 \includegraphics[width=350pt]{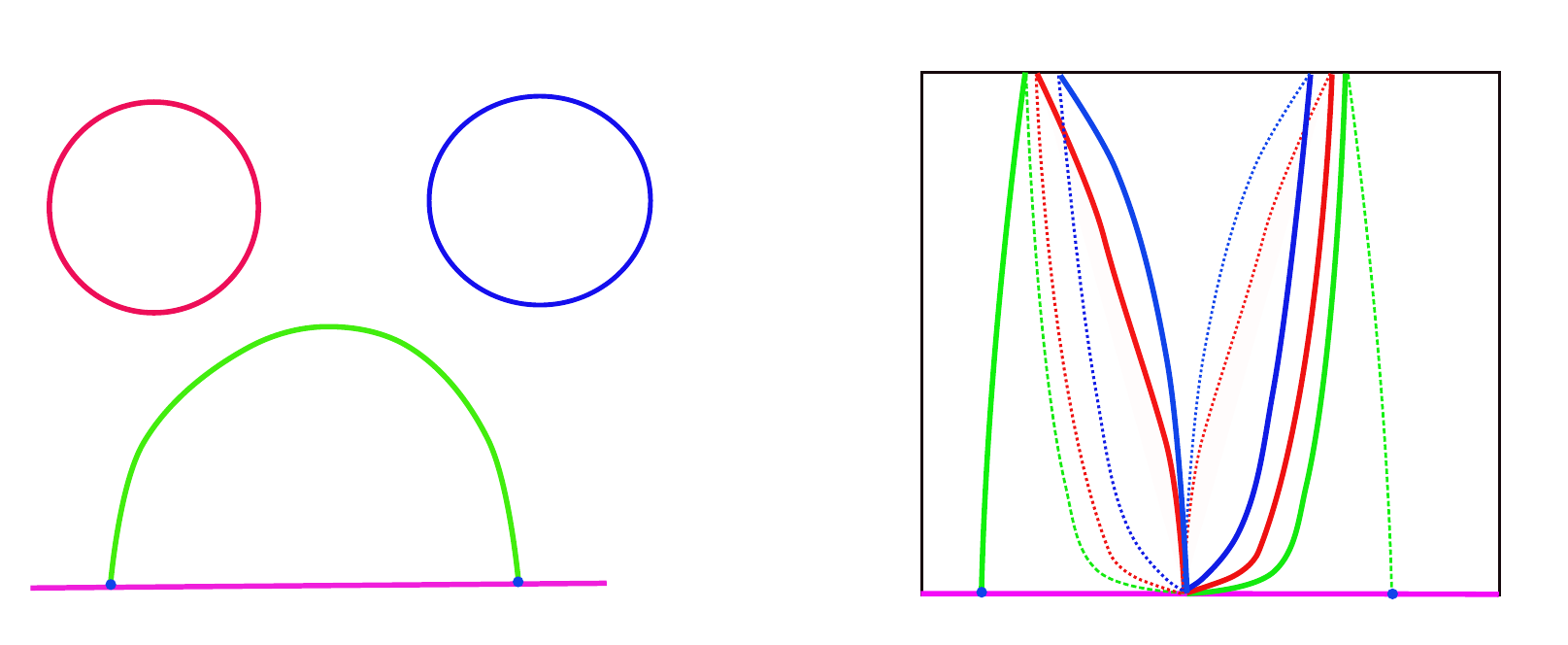}
 \caption{$\chi(S^3,K)$ and its image in $\chi(\partial N(K))$ for $K=7_2$.\label{figchi(7-2)}}
\end{center}
\end{figure}

 \section{Traceless representations of certain tangles}\label{2arcsinaball}
 
Having established the  presence of circles in $R^\nat(Y,K)$ in Proposition \ref{circle or point}, we are faced with the problem that this set is never finite (except for the unknot) and yet should give a generating set for the reduced instanton homology of $(Y,K)$, since it appears as the critical set of the relevant Chern-Simons functional.   The circles arise via the mechanism of Lemma \ref{double} because of  the presence of the   torus (i.e~ the boundary of the tubular neighborhood of $K$ which contains $H\cup W$) along which we can bend a representation. Holonomy perturbations  of the Chern-Simons functional (described below) are used to correct this problem.   We will show below that the circles can be eliminated by using a holonomy perturbation that lies in a 3-ball   intersecting $K$ in two unknotted arcs and   containing $H\cup W$. 
Thus we turn our attention to the representation spaces corresponding to  a pair of arcs in a ball.

\subsection{The space $R(B^3,A_1\cup A_2)$}  Consider the space of conjugacy classes of $SU(2)$ representations $R(B^3, A_1\cup A_2)$ of the complement of a pair of unknotted arcs in a 3-ball which send their meridians to elements in $C(\bbi)$, as illustrated in Figure \ref{2arcsfig}.  This corresponds exactly to the subspace  of $R(S^2,\{a,b,c,d\})$ consisting of those representations  which satisfy $\rho(a)=\rho(d)$ and $\rho(b)=\rho(c)$.  

 \begin{figure}  \centering
\def\svgwidth{2.2in}
 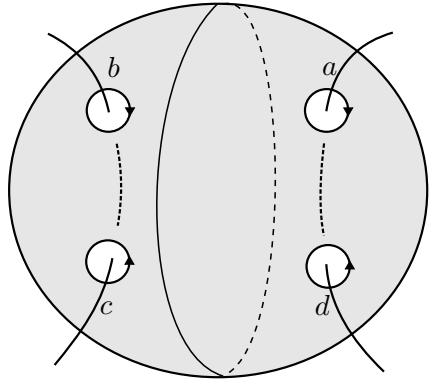
 \caption{ Two arcs in a $3$--ball,\label{2arcsfig} $(B^3,A_1\cup A_2)$. }
\end{figure}

We have the simple observation

\begin{prop}\label{unpert0}
 The space $ R(B^3, A_1\cup A_2)$ can be identified with the  arc $[0,\pi]$ via 
 $$a\mapsto \bbi, ~b\mapsto e^{\gamma\bbk}\bbi, ~c\mapsto e^{\gamma\bbk}\bbi, ~d\mapsto \bbi$$ for $\gamma\in [0,\pi]$. The restriction map
$$ R(B^3, A_1\cup A_2)\to R(S^2,\{a,b,c,d\})$$ is injective, with image   the diagonal arc  $\psi(\gamma,\gamma),~\gamma\in [0,\pi]$.  \qed
\end{prop}
\noindent Proposition \ref{unpert0}  is illustrated in Figure \ref{fig4}, where $R(S^2,\{a,b,c,d\})$ is represented as an identification space obtained from the rectangle $[0,\pi]\times [0,2\pi]$.

 \begin{figure}
\begin{center}
\def\svgwidth{2in}
 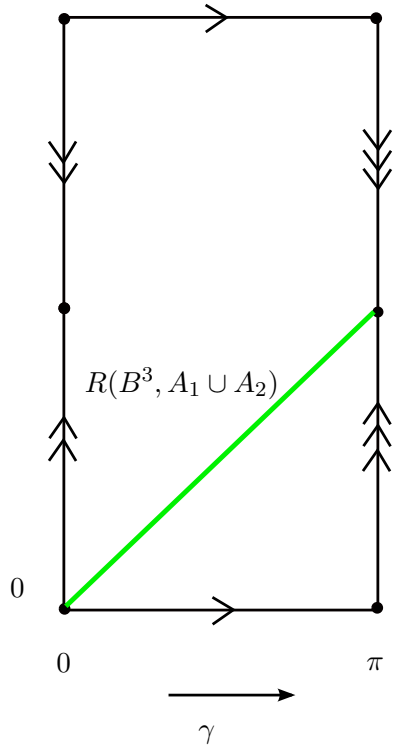
 \caption{The image of the restriction map $R(B^3, A_1\cup A_2)\to R(S^2,\{a,b,c,d\})$ is the arc $\theta=\gamma$ \label{fig4}}
\end{center}
\end{figure}

We now consider the space $R^\nat(B^3, A_1\cup A_2)$ consisting of conjugacy classes of $SU(2)$ representations which send $a,b,c,d$ and $h$ to $C(\bbi)$ and $w$ to $-1$, where $h$ denotes the meridian of $H$ and $w$ the meridian to the arc $W$.  As before, we may conjugate so that (abusing notation)
  $$a=d=\bbi,~ b=c=e^{\gamma\bbk}\bbi,~ w=-1$$
  with $\gamma\in [0,\pi]$.  The relation $[\bar a,\bar h]=w=-1$ implies that 
   $$h= e^{\tau\bbi}\bbj$$ for some $\tau$.
   
   Hence the image  of the restriction  
  $ R^\nat(B^3, A_1\cup A_2)\to R(S^2,\{a,b,c,d\})$ is equal to the image  $ R(B^3, A_1\cup A_2)\to R(S^2,\{a,b,c,d\})$. Both are given by the     arc  $\psi(\gamma,\gamma),~\gamma\in [0,\pi]$,  as illustrated above.  In the case of $R^\nat(B^3, A_1\cup A_2)$, however, the restriction is not injective. In fact, there is a circle action on $ R^\nat(B^3, A_1\cup A_2)$  defined by the parameter $\tau$. This action is free away from the two fixed points, which occur when $\sin\gamma=0$. The restriction map $ R^\nat(B^3, A_1\cup A_2)\to R(S^2,\{a,b,c,d\})$ factors through the orbit map, which is just the   map $ R^\nat(B^3, A_1\cup A_2)\to R(B^3, A_1\cup A_2)$ that forgets $h$ and $w$.  Summarizing, we have:
\begin{prop}\label{unpert}
  $ R^\nat(B^3, A_1\cup A_2)$ is a   2-sphere, corresponding to 
  $$a\mapsto \bbi, ~b\mapsto e^{\gamma\bbk}\bbi,~ c\mapsto e^{\gamma\bbk}\bbi, ~d\mapsto \bbi, 
  ~h\mapsto  e^{\tau\bbi}\bbj,~w\mapsto -1$$
  for $(\gamma,\tau)\in [0,\pi]\times [0,2\pi]$, cylindrical coordinates on $S^2$. The restriction to the pillowcase
  $$ R^\nat(B^3, A_1\cup A_2)\to R(S^2,\{a,b,c,d\})$$ has image the diagonal arc $\psi(\gamma,\gamma),~\gamma\in [0,\pi]$, and the fibers   are circles parameterized by $\tau$ over the interior of the arc and a single point over each endpoint.  \qed
\end{prop}

\section{Perturbations }

\subsection{Holonomy perturbations of the Chern-Simons functional  and its critical points}

We recall some aspects of the  definition  and basic properties of instanton homology and the related holonomy perturbations that we need. We provide the minimal description needed for our purposes, and refer to the series of articles \cite{KM1,KM-khovanov,KM-filtrations} for details.

\medskip

To the triple $(Y, K^\nat, W)$,  Kronheimer-Mrowka assign  {\em singular bundle data} ${\bf P}$ which  consists of an $SO(3)$ bundle over $Y\setminus K^\nat$ whose second Stiefel-Whitney class is Poincar\'e dual to $W$, as well as a certain kind of $O(2)$ reduction near $K\cup H$.  
The singular bundle data in turn gives rise to an affine space of $SO(3)$ connections 
 $\cC(Y,K;{\bf P})$ which have a prescribed singularity near $K^\nat$, with tangent spaces identified with spaces of  bundle-valued 1-forms with appropriate behavior near $K$.
 The  determinant one  gauge group $\cG(Y,K;{\bf P})$ acts on $\cC(Y,K;{\bf P})$.

 The Chern-Simons functional $\cs: \cC(Y,K;{\bf P})\to \RR$ has the property that its gradient vector field with respect to the $L^2$ inner product on $1$-forms is 
 $$({\text {\rm grad}} \cs)_B=*F_B$$
The critical points of the Chern-Simons functional therefore consist of {\em flat} singular $SO(3)$ connections.  The gauge group $\cG(Y,K;{\bf P})$ preserves 
the set of flat connections.  As explained in \cite{KM-khovanov}, the set of gauge equivalence classes of critical points, $\frak{C}(Y,K;{\bf P})$,  is identified, via the holonomy, with $R^\nat(Y,K)$.

  Any critical point $A$ of $\cs$ restricts to an honest flat connection on $Y-N(K)$. The bundle is trivial over $Y\setminus N(K)$ since $W\subset N(K)$ and hence the holonomy of $A$ lifts uniquely to $SU(2)$, giving a representation $\rho_A:\pi_1(Y\setminus N(K))\to SU(2)$ which sends the meridians to pure unit quaternions, i.e. $\rho_A\in R(Y,K)$.

 \medskip
  
   The reduced instanton homology $I^\nat(Y,K)$   is  the Morse-Floer  homology of a $\ZZ/4$ graded-chain complex $CI^\nat(Y,K)$ associated to $\cs$.  The chain complex  $CI^\nat(Y,K)$ should therefore be generated by the set of gauge equivalence classes  $\frak{C}(Y,K;{\bf P})$ of critical points  of $\cs$, which are identified with $R^\nat(Y,K)$. For any non-trivial knot  $K$ in $S^3$, however, Kronheimer-Mrowka have proved \cite{KM0}  that  there exist non-abelian traceless $SU(2)$ representations.  Thus Proposition \ref{circle or point} shows that $R^\nat(S^3,K)$ always contains circles.  Since the Morse complex of $\cs$ must be finitely generated in order to define a sensible theory, we see that we must perturb $\cs$ to ensure that its critical points are isolated and finite in number.   
  
  This is achieved by adding to  $\cs$ a {\em holonomy perturbation} $h_\pi:\cC(Y,K;{\bf P})\to \RR$. The function $h_\pi$ is constructed from  data   consisting of a collection of embeddings $e_i:S^1\times D^2\to Y\setminus(K^\nat\cup W)$ and some conjugation invariant functions $g_i:SU(2) \to \RR$. Following \cite{KM1, KM-khovanov}, $\pi=\{e_i, g_i\}$ is called the {\em perturbation data}, and determines the function  $h_\pi:\cC(Y,K;{\bf P})\to \RR$ by the formula
  $$h_\pi (A)=\sum _i \int_{ D^2} g_i(\text{hol}_{S^1 \times  \{x\}}(A))\eta(x) d^2x,$$
  where $\eta(x)$ is a (fixed) radially symmetric smooth cutoff function on $D^2$, and,  
  given $x\in D^2$, $\text{hol}_{S^1 \times  \{x\}}(A)$ denotes the holonomy of $A$ around the loop 
  $t\mapsto e_i(e^{t\bbi},x),~ t\in [0,2\pi]$. 
  
   Kronheimer-Mrowka consider more general perturbation functions, but these suffice for our purposes. In fact, we will only require one embedding $e:S^1\times D^2\to Y\setminus(K^\nat\cup W)$, and the function $g$ will be taken to be $g(q)=\epsilon \Real(q)$ for some small $\epsilon>0$.

 Denote by $P_j$ the image of $e_j$, a solid torus.  We will abuse terminology and call $P_J$ a {\em perturbation curve}. Denote by $p_j$ the meridian of $P_j$, i.e. the (suitably based) loop $e_j(*\times \partial D^2)$, and by $\ell_j$ a choice of longitude, e.g. $e_j(S^1\times *)$. Let $P=\sqcup_j P_j$.   Denote by $\mu_K,\mu_H, \mu_W$ (suitably based) meridians of $K, H, W$ respectively.

  Any critical point of $\cs + h_\pi$  is the gauge equivalence class of a connection which is flat outside the image $P$ of the 
embeddings $e_i:S^1\times D^2\to Y\setminus(K^\nat\cup W)$.   
Therefore,  critical points of   the perturbed functional $\cs + h_\pi$ are identified with conjugacy classes of representations 
$\pi_1(Y\setminus(K\cup H\cup W\cup P))\to SU(2)$  which send the meridians $\mu_K, \mu_H$ to $C(\bbi)$, and $\mu_W$ to $-1$, and which also satisfy a certain constraint determined by $g_j$  when restricted to the meridian and longitude of  the $j$th boundary torus of $P$. 

 Lemma 61 of \cite{herald} identifies the constraint and shows that given any   list of smooth functions $f_j:\RR\to \RR$ satisfying $f_j(-x)=-f_j(x)$ and  $f_j$ is $2\pi$ periodic (for example, $f_j(x)=\sin(x)$), there exist $g_j$ as above so that the  constraint on the $j$th boundary torus is given by the {\em perturbation condition}:
 \begin{center} {\em If the representation
 $\rho:\pi_1(Y\setminus(K\cup H\cup W\cup P))\to SU(2)$ takes the  meridian and longitude $p_j,\ell_j$ of the $j$th component of $P$ to $e^{\nu_j Q_j}$ and $e^{\beta_j Q_j}$ respectively, for some $Q_j\in C(\bbi)$, then $\nu_j=f_j(\beta_j)$.}
 \end{center}

  We denote the  space of conjugacy classes of such representations by $R_\pi^\nat(Y,K)$. Explicitly, $R_\pi^\nat(Y,K)$  is the space of $SU(2)$-conjugacy classes  of representations
  $$
  \rho:\pi_1(Y\setminus(K\cup H\cup W\cup P))\to SU(2)$$
  satisfying
\begin{equation}\label{pertcond}\rho(\mu_K),\rho(\mu_H)\in C(\bbi), \rho(\mu_W)=-1,
  \nu_j=f_j(\beta_j) \text{~when~} \rho(p_j)=e^{\nu_j Q_j} \text{~and ~}\rho(\ell_j)=e^{\beta_j Q_j}.
\end{equation}

Similar constructions apply to $\chi(Y,K)$,  the space of conjugacy classes of all representations $\pi_1(Y\setminus N(K))\to SU(2)$. One can perturb   using embeddings $e_i:S^1\times D^2\to Y\setminus N(K)$; the resulting critical set is denoted by  $\chi_\pi(Y,K)$, and is identified with those representations $\pi_1(Y\setminus(K\cup P))\to SU(2)$  satisfying the constraints 
 \begin{equation}\label{pertconst}
 \nu_j=f_j(\beta_j) \text{~when~} \rho(p_j)=e^{\nu_j Q_j} \text{~and ~}\rho(\ell_j)=e^{\beta_j Q_j}
\end{equation}
on each $\partial P_j$.
\medskip

 When   $P$  lies outside the tubular neighborhood $N(K)$ of $K$ containing $K^\nat\cup W$, restriction defines a map 
$R_\pi^\nat(Y,K)\to \chi_\pi(Y,K)$
with image which we denote by $R_\pi(Y,K)$. Just as in the unperturbed case, the fibers of the restriction map 
$$R_\pi^\nat(Y,K)\to R_\pi(Y,K)$$
are circles over every representation with non-abelian image, and a point over every representation with abelian image. 

In particular,  any choice of perturbation data $\pi=\{e_i, g_i\}$ which makes $R_\pi^\nat(Y,K)$  a finite union of isolated points  must include at least one perturbation curve $P\subset  Y\setminus (K^\nat\cup W)$ which intersects the separating torus $\partial N(K)$ essentially and hence must link $H\cup W$ in some way.  This observation motivates using the perturbation curve $P$ in Figure \ref{fig8}.  However, before analyzing the effect of perturbing along $P$ to turn circles into pairs of isolated points, a discussion  concerning the {\em non-degeneracy} of the space $R^\nat_\pi(Y,K)$ for appropriate perturbations is in order.  

\subsection{Non-degeneracy}

In Floer-type theories, the chain complexes are generated by critical points of a functional defined on a configuration space, and the boundary operator is defined by counting integral curves for the gradient of the functional that connect the critical points (so-called ``gradient flowlines").  The critical points of the functional  form a moduli space, and to have a finitely generated complex, then, one must ensure that the number of points in these spaces are finite.  This can be achieved by the analogue of a slight perturbation of a real-valued function on a  finite dimensional manifold to ensure that its critical points are non-degenerate, and hence isolated (and finite in number if the manifold is compact).    In the infinite dimensional setting, we likewise must perturb the functional defining the Floer theory to achieve non-degeneracy of the Hessian.  This can be viewed as a first step in constructing a Floer theory. 
 
To ensure that differentials are defined, and that the resulting homology is well defined, however, requires more.   Namely, we must  have some form of transversality for the moduli spaces of gradient flow lines connecting critical points of the functional. Precise conditions depend on the context, and are usually subsumed under the terms  ``regularity'' or ``transversality".     For example, in finite dimensions a function $f:M\to \RR$ is non-degenerate if it is a Morse function, and regularity adds the requirement that the stable and unstable manifolds intersect transversally (which is usually referred to as the Morse-Smale condition).  In Floer theory for Lagrangian intersections, non-degeneracy is achieved by a Hamiltonian isotopy of  the Lagrangian submanifolds so that they  intersect transversally.  Regularity is achieved by perturbing the almost complex structure which defines the $J$-holomorphic curve equation satisfied by gradient flowlines in the space of paths connecting the two Lagrangians.  In the context of instanton homology, non-degeneracy is expressed by the condition that the perturbed Chern-Simons functional is Morse in a suitable infinite-dimensional sense.

A Lagrangian-Floer theory  in the context of traceless character varieties of knot complements in 3-manifolds has not yet been worked out and identified with singular  instanton knot homology. Providing such an identification is an instance of the {\em Atiyah-Floer Conjecture}, which has been established in several other  settings.   In  all known cases critical points of the  Chern-Simons functional are non-degenerate (i.e.~Morse) exactly when the corresponding Lagrangian intersections are transverse.  This is proved by identifying the kernel of the Hessian of the Chern-Simons functional at  a critical point with some form of de Rham cohomology, which is then related to transversality of the Lagrangian submanifolds via the de Rham theorem and the Mayer-Vietoris sequence.

To carry this out carefully below in our context would take us too far afield from the intent of this article.  We offer instead the following as a notion of non-degeneracy in this article, which is adapted from Weil's observation \cite{weil} that $H^1(\pi;\mathfrak{g})$ is identified with the tangent space at a smooth point of the character variety $\chi(\pi,G)$.

Given perturbation data $\pi$ with the $f_j$ real analytic, the space $R^\nat_\pi(Y,K)$ is the orbit space of the free $SU(2)/\pm 1=SO(3)$ conjugation action on a real analytic variety.  This variety is constructed as follows.

Consider  a presentation $\langle G \mid R\rangle$ of $\pi_1(Y\setminus (K\cup W \cup P))$.  
 The presentation defines in the usual way a polynomial map
$F_1:SU(2)^G \to SU(2)^R$ with preimage $F_1^{-1}(1,1,\cdots,1)$ parameterizing  all homomorphisms $\pi_1(Y\setminus(K\cup H\cup W\cup P))\to SU(2).$

Fixing three words in the generators representing the meridians $\mu_K, \mu_H$ and $\mu_W$, a  map 
$F_2: SU(2)^G\to SU(2)$ is defined by sending a $G$-tuple to the image of $\mu_W$.  Then $F^{-1}_2(-1)$ meets $F_1^{-1}(1,1,\cdots,1)$ in those homomorphisms sending $\mu_K,\mu_H$ to perpendicular elements of $C(\bbi)$ and $\mu_W$ to $-1$.   (We are assuming $K$ is a knot, not a link. In the case when $K$ has $n$ components and $H$ links $K_1$,  define $F_2:SU(2)^G\to SU(2)\times \RR^{n-1}$
by taking the last factors to be the value $\Real(\mu_{K_i}),\ i=2,\cdots,n.$)

  Finally, for each perturbation curve $P_j$, let $G_j:SU(2)\to SU(2)$ be the map $G_j(e^{\nu Q})=e^{f_{j}(\nu) Q}$. The properties of $f_j$ imply that $G_j$ is well defined. The meridian and longitude $p_j,\ell_j$ of $P_j$, expressed as words in the generators $G$, give rise to a map $F_{3,j}:SU(2)^G\to SU(2)$ by $F_{3,j}= p_j^{-1} G_j(\ell_j).$  The preimage of $1$ precisely captures the perturbation constraints (Equation (\ref{pertconst})).

The product
$$F=F_2\times F_1\times \prod_j F_{3,j}:SU(2)^G\to SU(2)\times SU(2)^R \times \prod_j SU(2)$$
defines $R^\nat_\pi(Y,K)$ in the sense that 
$$R^\nat_\pi(Y,K)= F^{-1}(1,1,\cdots,1)/SU(2).$$
Thus $R^\nat_\pi(Y,K)$ is finite if and only if $ F^{-1}(1,1,\cdots,1)$ is a finite union of copies of $SO(3)$ with $SU(2)/\pm1$ acting freely.  

A  representation  $\rho$ whose conjugacy class lies in $R^\nat_\pi(Y,K)$ determines a point (which we denote again  by $\rho$) in $F^{-1}(1,1,\cdots,1)$, by evaluating $\rho$ on the elements of $G=\pi_1(Y\setminus(K\cup W\cup P))$.   
We call $\rho$  {\em non-degenerate} if is isolated in $R^\nat_\pi(Y,K)$ and the kernel of the differential $dF_\rho$ is 3-dimensional, i.e. it is the tangent space to the orbit under the diagonal conjugation action of $SU(2)$ on $SU(2)^G$.  (The fact that, for any $G$-tuple in the preimage, $\mu_K,\mu_H$ are sent to noncommuting elements of $SU(2)$ implies that the stabilizer of the $G$-tuple is $\{\pm 1\}$ and the orbit is 3-dimensional.) 
We then call $R^\nat_\pi(Y,K)$   non-degenerate  if it is finite and all its points are non-degenerate.

Standard arguments show that the notion of non-degeneracy is unchanged by changing the presentation. In fact one can take the presentation to be generated by all elements of $\pi_1(Y\setminus(K\cup H\cup W\cup P)$ and relations to be the entire multiplication table to see that the notion depends only on the group and the choice of perturbation.   

 It will be clear from our constructions below  that as a space of representations, $R^\nat_\pi(Y,K)$  is a transverse intersection of two smooth 1-dimensional manifolds in a 2-manifold, in two different ways corresponding to the two different pillowcase pictures. One can show that transversality in these two contexts is equivalent, and equivalent to non-degeneracy as defined above, and we will explore this issue in subsequent work. In  this sense we produce  in Theorem \ref{pertvar} below (in concert with the results of \cite{herald}) a perturbation which makes $R^\nat_\pi(Y,K)$ non-degenerate.

\medskip

A more sophisticated approach is to recast this in the context of group cohomology, which identifies the kernel of $dF_\rho$ with certain 1-cocycles and the tangents to the conjugation orbits with coboundaries, at least when the perturbation is trivial. Then non-degeneracy is equivalent to the vanishing of $H^1(Y\setminus(K\cup H);so(3)_\rho)$ (or, if $K$ has $n$ components, the vanishing of the kernel of 
$H^1(Y\setminus(K\cup H);so(3)_\rho)\to \oplus_{i=1}^n H^1(\mu_{K_i};so(3)_\rho)$).  See \cite[Lemma 3.13]{KM1} and also \cite[Proposition 2.10]{gerard}.  For non-trivial perturbations similar conditions apply.  The references  loc.cit. show that non-degeneracy in this cohomological sense is equivalent to the perturbed Chern-Simons functional $CS+h_\pi$ having a non-degenerate Hessian at its critical points, and therefore its critical points serve as generators for the instanton-Floer complex defining  $I^\nat(Y,K)$.

 \medskip
As an illustration of these ideas, we show that the distinguished representation $\alpha\in R^\nat(S^3,K)$, defined in Equation (\ref{alphadef}), is non-degenerate.

Decompose $S^3\setminus (K\cup H\cup W)$ into $S^3\setminus N(K)$ and $N(K)\setminus(K\cup H\cup W)$ along a torus $T=\partial N(K)$ as in Equation (\ref{decomptorus}). 
Let $\mu_K,\mu_H,\mu_W$ denote the meridians of $K,H,L$ and $\lambda_K$ the longitude of $K$.  Then $$\pi_1(N(K)\setminus(K\cup H\cup W))=\langle \mu_K,\mu_H,\mu_W,\lambda_K~|~ \mu_W=[\mu_K,\mu_H], 1=[\mu_K,\lambda_K]\rangle.$$
 The representation $\alpha$ restricts to the unique abelian    representation on $S^3\setminus N(K)$ sending the meridian $\mu_K$  (which generates $H_1(S^3\setminus N(K))$)  to $\bbi$.  The longitude $\lambda_K$ is sent to $1$ since it maps to zero in $H_1(S^3\setminus N(K))$.   Finally, $\alpha$ restricts to the non-abelian representation 
 $$\alpha(\mu_K)=\bbi, ~\alpha(\mu_H)=\bbj,~ \alpha(\mu_W)=-1, ~\alpha(\lambda_K)=1$$
 on $N(K)\setminus(K\cup H\cup W)$.

 Hence (taking coefficients in $su(2)$ twisted by $\alpha$),  $H^0(S^3\setminus N(K))=\RR=H^0(T), \ H^0(N(K)\setminus(K\cup H\cup W))=0$. Since $\alpha(\mu_K)=e^{\bbi\pi/2}$ satisfies $\Delta_K(\alpha(\mu_K)^2)\ne 0$, $H^1(S^3\setminus N(K))=\RR$ (see \cite{Klassen}) and $H^1(T)=\RR^2$.

A straightforward calculation yields $H^1(N(K)\setminus(K\cup H\cup W))=\RR^4$ and that the restriction to $H^1(T)=\RR^2$ is surjective.  Indeed, given any pair of unit quaternions $(q_1, q_2)$ near $(\bbi,\bbj)$ and a third unit quaternion $q_3$ in the unique circle subgroup through $q_1$, the assignment
 \begin{equation}\label{family}
\mu_K\mapsto q_1,~\mu_H\mapsto q_2,~ \mu_W\mapsto [q_1,q_2],~\lambda_K\mapsto q_3
\end{equation}
gives a smooth 7-dimensional family of irreducible representations near $\alpha$
on which conjugation acts freely modulo $\pm1$, so that $H^1(N(K)\setminus(K\cup H\cup W))$ is 4-dimensional and maps onto $H^1(T)$.

The Mayer-Vietoris sequence then shows that $H^1(S^3\setminus (K\cup H\cup W))=\RR^3$.   Since $H^1(S^3\setminus (K\cup H\cup W))\cong \ker dF_1/B_1$, where $B_1$ denotes the tangent space to the 3-dimensional orbit through $\alpha$  of the conjugation action (see e.g. \cite{weil})  it follows that $\ker dF_1\cong\RR^6$.  The map $SU(2)\times SU(2)\to SU(2)$ taking a pair to their commutator has $-1$ as a regular value, and hence the map $F_2$ is a submersion near $\alpha$ because the quaternions $q_1,q_2$ of Equation (\ref{family}) can be chosen arbitrarily near $\bbi,\bbj$. This implies that the kernel of $d(F_1\times F_2)$ at $\alpha$ is $6-3=3$ dimensional, so that $\alpha$ is non-degenerate. In particular, $\alpha$ remains non-degenerate under small perturbations.

  \subsection{Perturbation in a 3-ball}\label{pertinball}

Place a loop $P$ inside $B^3$, linking $A_2$ and $H$ as illustrated in Figure \ref{fig8}.  Use the standard meridian-longitude framing of  $P$ to think of $P$ as the image of an embedding  of a solid torus $e:S^1\times D^2\to B^3$.  Label the generators of $\pi_1(B^3\setminus(A_1\cup A_2\cup H\cup W\cup P))$ by $a,b,c,d,h,w,$ and $p$
as indicated in the figure.

  \begin{figure}[H]\centering
\def\svgwidth{3.5in}
 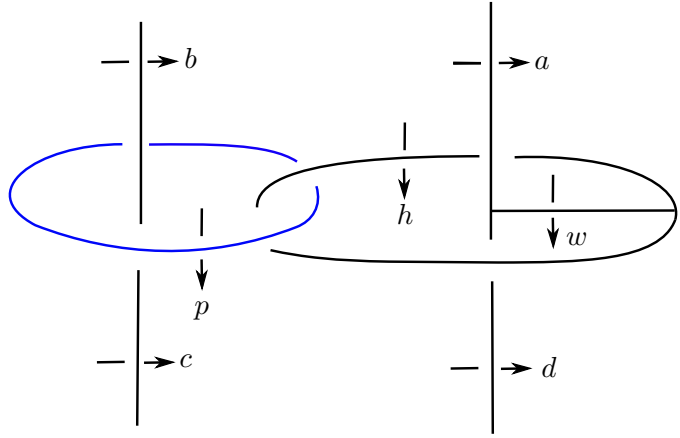

 \caption{\label{fig8}The local picture in the 3-ball where the perturbation occurs.  The holonomy perturbation occurs in the neighborhood of the blue circle.}
\end{figure}

 Fix a smooth function $f:\RR\to \RR$ satisfying $f(0)=0, |f(x)|\leq 1, ~f(-x)=-f(x)$,  $f$ is $2\pi$ periodic, and $f(x)\ne 0$ when $x$ is not a multiple of $\pi$  (it suffices for our purposes to take  $f(x)=\sin(x)$).  Note that $f(n\pi)=0$ for integers $n$.

Fix an $\epsilon\ge 0$ and denote the data $(P,f,\epsilon)$ by $\pi$. 
Define the {\em $\pi$-perturbed moduli space} $R^\nat_{\pi}(B^3,A_1\cup A_2)$ to be the space of conjugacy classes of $SU(2)$ representations which send $a,b,c,d$ and $h$ to $C(\bbi)$, $w$ to $-1$, and which satisfy the perturbation condition. Thus   if the longitude $\lambda_P=bh$ of the  component $P$ is sent to $e^{\beta Q}$ and the meridian $p$ to $e^{\nu Q}$ for some $Q\in C(\bbi)$, then 
 $$\nu=\epsilon f(\beta).$$
In particular, when $\epsilon =0$,   $R^\nat_{\pi}(B^3,A_1\cup A_2)= R^\nat (B^3,A_1\cup A_2)$.
 
 \medskip

 The following is our main result. Its statement is perhaps best understood by examining Figure \ref{fig9}.

\begin{thm}\label{pertvar}  For  all $\epsilon>0$ small enough, the space   $R^\nat_{\pi}(B^3,A_1\cup A_2)$ is homeomorphic to a   circle, 
 parameterized  by $\beta\in \RR/2\pi\ZZ$ by
 the assignment
 \begin{equation*}
\begin{split}
\rho(\beta): a\mapsto \bbi, ~b\mapsto e^{\left(\tfrac \pi 2 + \beta + \epsilon f(\beta) \right) \bbk}\bbi, ~c\mapsto e^{\left(\tfrac \pi 2 + \beta -\epsilon f(\beta) \right)\bbk}\bbi, \\~d\mapsto e^{ -2\epsilon f(\beta)\bbk}\bbi, ~h\mapsto -\bbj e^{-\epsilon f(\beta)\bbk},
 ~p\mapsto e^{\epsilon f(\beta)\bbk}, ~ w\mapsto -1
. \end{split}
\end{equation*}
 \end{thm}

 \begin{proof}
A straightforward calculation using the Seifert-Van Kampen  theorem shows that $$\pi_1(B^3\setminus(A_1\cup A_2\cup P\cup H\cup W))$$ is presented with generators $a,b,c,d,w,h, p$, and   relations
\begin{equation}\label {eqstep0}  c=\bar p b p,\ d= \bar c b a, \  [bh,p] =1, \text{~and~} [a\bar p,h]=(ha)w (\bar a\bar h). 
\end{equation}
(The two commutator relations come from the tori which form the boundaries of the tubular neighborhoods of $P$ and $H$.)

Let  $R^\nat_{\pi}(B^3 \setminus P,A_1\cup A_2)$ denote the space of conjugacy classes of $SU(2)$ representations which send $a,b,c,d,h$ to $C(\bbi)$ and $w$ to $-1$,
with no restrictions on where $p$ is sent.

Since $a,b,c,d$ satisfy the same relation $d=\bar c b a$ in 
$\pi_1(B^3\setminus(A_1\cup A_2\cup P\cup H\cup W))$ that is the defining relation in $\pi_1(S^2 \setminus \{a,b,c,d\})$, 
 Proposition \ref{s2prop} implies that we may assume up to conjugation  that any  representation satisfies (abusing notation slightly to identify generators with their image in $SU(2)$):
 \begin{equation*}\label{eqstep1}
a=\bbi, ~b=e^{\gamma\bbk}\bbi, ~c=e^{\theta\bbk}\bbi, ~d=e^{(\theta-\gamma)\bbk}\bbi
\end{equation*}
for $(\gamma,\theta)\in [0,\pi]\times S^1$. {\em  It will be convenient to relax this condition and assume that $$(\gamma, \theta)\in \RR/(2\pi\ZZ)\times S^1=S^1\times S^1$$ for most of this proof, and then determine which pairs of representations are equivalent at the end of the argument.}

 Denote by $Q_H$ the image of $h$, and let
$Q_P\in C(\bbi)$ be a pure unit quaternion so that 
 $\lambda_P=bh$ is sent to $e^{\beta Q_P}$ for some $\beta\in [0,2\pi)$. 
 Since $p$ commutes with $\lambda_P$, it follows that the representation sends $p$ to $e^{\nu Q_P}$ for some angle $\nu$.   The perturbation condition implies that 
  $\nu=\epsilon f(\beta)$, but we'll impose this condition last.  For the moment, we assume  $\epsilon$ is small, so that $|\nu|<\frac \pi 2$. 
Note that
 $Q_P$ is determined up to $\pm 1$, unless $\lambda_P$ is sent to $\pm 1$, in which case we may take $Q_P$ to be any  element of $C(\bbi)$.

  Summarizing what we have accomplished so far:   any   representation  whose conjugacy class lies in  $R^\nat_{\pi}(B^3,A_1\cup A_2)$ can be conjugated so that  \begin{equation}\label{eqstep2}
a=\bbi, ~b=e^{\gamma\bbk}\bbi, ~c=e^{\theta\bbk}\bbi, ~d=e^{(\theta-\gamma)\bbk}\bbi, h=Q_H, p=e^{\nu Q_P}, ~ w=-1, \lambda_P=e^{\beta Q_P}
\end{equation} 
for some 6-tuple
   \begin{equation}\label{eqstep3}
   (\gamma, \theta, Q_H, Q_P, \beta, \nu)\in
 [0,2\pi)\times [0,2\pi)\times  C(\bbi)\times C(\bbi)\times [0,2\pi)\times \left(-\tfrac \pi 2 , \tfrac \pi 2 \right) \end{equation} 
 One obvious redundancy in this description is that $(\nu , \beta, Q_P)$ is equivalent to $(-\nu, \beta+\pi, -Q_P)$.

\medskip

The relation $\bar p bp=c$ implies that
 \begin{equation}\label{eqstep6}
e^{-\nu Q_P}e^{\gamma\bbk}\bbi e^{\nu Q_P}=e^{\theta\bbk}\bbi . 
\end{equation} 
Recall from Proposition \ref{basic} that the conjugation action of $e^{tQ}$ on the 2-sphere $C(\bbi)$ is rotation about the axis through $\pm Q$ of angle $2t$.  
 
 If $e^{\gamma\bbk}\ne e^{\theta\bbk}$, then Equation (\ref{eqstep6}) implies that  $Q_P$ lies on the great circle in $C(\bbi)$ through $\bbk$ and 
 $e^{\frac{\gamma+\theta}{2}\bbk}\bbi$. This is because the only axes $\pm Q$ for which the orbits of rotation pass through both $e^{\gamma\bbk}\bbi$ and $
 e^{\theta\bbk}\bbi $ lie on this great circle.  

 If $e^{\gamma\bbk}= e^{\theta\bbk}$, then $e^{\nu Q_P}$ stabilizes $e^{\gamma\bbk}\bbi$. When $\nu\ne 0$, this implies that  $Q_P=\pm e^{\gamma\bbk}\bbi,$ so either $\nu=0$ (and $Q_P$ is not constrained by this relation) or $Q_P=\pm e^{\gamma \bbk}\bbi$.

 Hence, in any case,
  \begin{equation}\label{eqstep7}
  \nu = 0 \mbox{ and } \gamma=\theta, \mbox{ or } Q_P= \cos t \bbk + \sin t e^{\frac{\gamma+\theta}{2}\bbk}\bbi 
 \end{equation}
 for some $t\in[0,2\pi)$.

The relation $h=\bar b\lambda_P$ implies that $Q_H$  is determined by $\gamma,$ $\beta$ and $Q_P$   by the equation
 \begin{equation*}
 Q_H=-\bbi e^{-\gamma\bbk}e^{\beta Q_P}.\end{equation*}
Since $\Real(Q_H)=0$, using Equation (\ref{eqstep7}) we see
\begin{eqnarray*} 
 0=\Real(Q_H)&=&\Real(-\bbi e^{-\gamma\bbk}e^{\beta Q_P})\\
 &=&\Real\big(-\bbi e^{-\gamma\bbk}(\cos\beta +\sin \beta(\cos t \bbk + \sin t e^{\frac{\gamma+\theta}{2}\bbk}\bbi ))\big)\\
 &=&\sin\beta\Real\big(-\bbi e^{-\gamma\bbk}  (\cos t \bbk + \sin t e^{\frac{\gamma+\theta}{2}\bbk}\bbi  )\big)\\
 &=&\sin\beta\sin t \Real\big(-\bbi e^{-\gamma\bbk}    e^{\frac{\gamma+\theta}{2}\bbk}\bbi   \big)\\
 &=&\sin\beta\sin t\cos(\tfrac{\theta-\gamma}{2})
 \end{eqnarray*}
 Hence 
\begin{equation}\label{eqstep9}
0=\sin\beta\sin t\cos(\tfrac{\theta-\gamma}{2})
 \end{equation}
 
 Next, consider the relation  $[a\bar p,h]=haw\bar a\bar h=-1$, or, in terms of our chosen coordinates:
$$[\bbi e^{-\nu Q_P}, Q_H]=-1.$$
 This equation can be rewritten, using the fact that if $Q\in C(\bbi)$, then $ Q^{-1}=-Q$,  as
$$
 \bbi e^{-\nu Q_P}  Q_H (\bbi e^{-\nu  Q_P})^{-1}=-Q_H. 
$$
Part (3) of Proposition \ref{basic} shows that $ \bbi e^{-\nu Q_P}$ is itself a pure unit quaternion, and is perpendicular to $Q_H$.  Hence 

\begin{eqnarray} 
 0&=& \Real(\bbi e^{-\nu Q_P})\nonumber \\
&=&-\sin\nu \Real(\bbi Q_P)\nonumber \\
&=&-\sin \nu \Real\big(\bbi (\cos t \bbk + \sin t e^{\frac{\gamma+\theta}{2}\bbk}\bbi )\big)\nonumber\\
&=& \sin \nu \sin t\cos(\tfrac{\gamma+\theta}{2}) \label{eqstep10prime}
 \end{eqnarray}
\medskip

Thus far we have shown that all $SU(2)$ representations of $\pi_1(B^3 \setminus (A_1 \cup A_2\cup P \cup H\cup W))$ sending $a,b,c,d,h$ into $C(\bbi)$, and $w\mapsto -1$ can be conjugated to have the form 
in \eqref{eqstep2}, for a 6-tuple in \eqref{eqstep3};  furthermore, either $\nu=0$ and $\gamma=\theta$, or we can express $Q_P$ as $$Q_P=\cos t \bbk + \sin t  e^{\tfrac{\gamma+\theta}2 \bbk} \bbi.$$  In the latter case,   equations \eqref{eqstep9} and \eqref{eqstep10prime}  must also hold.  

Now assume that such a representation corresponds to the holonomy of a perturbed flat connection on $B^3 \setminus (A_1 \cup A_2 \cup H \cup W)$.  Then there is the additional condition that $\nu = \epsilon f(\beta)$, where  $f$ is an odd, $2\pi$--periodic, function, whose zeroes occur only at  multiples of $\pi$.  We will  examine what this additional restriction implies about the representations in $R^\nat_{\pi}(B^3,A_1\cup A_2)$.

To begin, note that $\nu = 0 $ if and only if $\sin \beta = 0$.  If $\nu =\sin \beta = 0$, then the representation is independent of $Q_P \in C(\bbi)$;  in particular, in this case, we can assume $Q_P=\pm\bbk$. We next examine the case that $\nu \neq 0$ and $\sin \beta\neq 0$.   In this case, either
$$\sin t=0,$$or $$ \cos(\tfrac{\theta+\gamma}{2})=0\  \ \text{and}\ \cos(\tfrac{\theta-\gamma}{2})=0.$$

\noindent  (note we are using the fact that  $\nu=0$ if and only if $\sin\nu=0$, since $|\nu|<\frac{\pi}{2}$.)

 Suppose $\cos(\tfrac{\theta+\gamma}{2})=0$ and $\cos(\tfrac{\theta-\gamma}{2})=0$. Since $\gamma\in[0,2\pi)$ and $\theta\in [0,2\pi)$, there are two solutions:
$$(\gamma,\theta) = (0,\pi) \text{~or~} (\pi,0).$$

\noindent If $(\gamma,\theta)$ is equal to $(\pi, 0)$, then Equation (\ref{eqstep2}) implies that 
$a=\bbi, b=-\bbi,  c= \bbi$.  If   $(\gamma,\theta)=(0,\pi)$, then $a=\bbi, b=\bbi, c=-\bbi$.
The relation $\bar p b  p=c$  implies in either case that 
$e^{-\nu Q_P}\bbi e^{\nu Q_P}=-\bbi$. Proposition  \ref{basic} then implies that $e^{-\nu Q_P}$ is a pure unit quaternion, which is impossible  since $|\nu|<\tfrac{\pi}{2}$. 
Thus one of 
$\cos(\tfrac{\theta+\gamma}{2})$ and $\cos(\tfrac{\theta-\gamma}{2})$ must be non-zero.  It follows that if $\sin \beta \neq 0$, then $\sin t=0$.  In particular, either $\sin \beta \neq 0$ and $\sin t$ {\em must} equal zero, or $\sin \beta = 0$ and we can assume $\sin t=0$ without changing the representation.  Hence we may assume that  $Q_P=\sigma\bbk$ for some $\sigma\in \{\pm 1\}$.

%
\bigskip

We've now seen that 
any   representation  whose conjugacy class lies in  $R^\nat_{\pi}(B^3,A_1\cup A_2)$ can be conjugated so that  \begin{equation}\label{eqstep2prime}
a=\bbi, ~b=e^{\gamma\bbk}\bbi, ~c=e^{\theta\bbk}\bbi, ~d=e^{(\theta-\gamma)\bbk}\bbi, ~h=-\bbi e^{(\beta\sigma-\gamma)\bbk}, ~
p=e^{\epsilon f(\beta) \sigma \bbk}, ~ w=-1,~ \lambda_P=e^{\beta \sigma \bbk}
\end{equation} 
for some 4-tuple
   \begin{equation}\label{eqstep3prime}
   (\gamma, \theta, \beta, \sigma)\in
 [0,2\pi)\times [0,2\pi)\times   [0,2\pi)\times \{ \pm 1 \}
  \end{equation} 
 
Since $f(\beta+\pi)=-f(\beta)$, $(\gamma, \theta, \beta, \sigma)$ gives the same representation as $(\gamma, \theta,2\pi-  \beta, -\sigma)$ when $\beta>0$. When $\beta=0$, then $\beta=-\beta$. Hence we may assume that $\sigma =1$, or more precisely, any   representation  whose conjugacy class lies in  $R^\nat_{\pi}(B^3,A_1\cup A_2)$ can be conjugated so that
 
    \begin{equation}\label{eqstep2subprime}
a=\bbi, ~b=e^{\gamma\bbk}\bbi, ~c=e^{\theta\bbk}\bbi, ~d=e^{(\theta-\gamma)\bbk}\bbi,~ h=-\bbi e^{(\beta-\gamma)\bbk}, 
~p=e^{\epsilon f(\beta)  \bbk}, ~ w=-1, ~\lambda_P=e^{\beta  \bbk} \end{equation} 
where so far, the angles $\gamma, \theta, \beta$  could lie anywhere in $[0,2\pi)$.  

\medskip

Finally, we determine what relations between the angles are necessary to satisfy the perturbed flat equation.

   The relation $\bar p b p=c$ implies that $e^{-\epsilon f(\beta)\bbk}e^{\gamma\bbk}\bbi e^{\epsilon f(\beta) \bbk}=e^{\theta\bbk}\bbi$,
which implies $e^{(\gamma-2\epsilon f(\beta))\bbk}=e^{\theta\bbk}$, so
 \begin{equation}\label{eqstep10.5}
\theta \equiv \gamma-2\epsilon f(\beta)\pmod{2\pi}.
 \end{equation}

The relation $[a\bar p, h]=-1$ gives
$$[\bbi e^{-\epsilon f(\beta)\bbk},-\bbi e^{(\beta-\gamma)\bbk}]=-1$$
and so
$$-1=\bbi e^{-\epsilon f(\beta)\bbk}(-\bbi )e^{(\beta-\gamma)\bbk}
e^{\epsilon f(\beta)\bbk}(-\bbi)e^{-(\beta-\gamma)\bbk}\bbi
=e^{2(\epsilon f(\beta)+\beta-\gamma)\bbk}.$$
Hence
  \begin{equation}\label{eqstep12}   \gamma \equiv \beta+\epsilon f(\beta) +\tfrac \pi 2  \mod{\pi}.
 \end{equation}
 In other words, we have 
  \begin{equation}\label{firstcircle}\begin{split}
 \gamma ~&\equiv ~\beta+\epsilon f(\beta) +\tfrac \pi 2  \pmod{2\pi}\\
\theta ~&\equiv ~\beta -\epsilon f(\beta)  +\tfrac \pi 2\pmod{2\pi}
\end{split}\end{equation}
or else 
 \begin{equation}\label{secondcircle}\begin{split} 
 \gamma ~&\equiv \beta+\epsilon f(\beta) +\tfrac {3\pi} 2  \pmod{2\pi}\\
\theta ~&\equiv~ \beta -\epsilon f(\beta)  +\tfrac {3\pi} 2\pmod{2\pi}
\end{split}\end{equation}
These formulas give us two families of representations indexed by $\beta\in [0, 2\pi )$ whose union maps subjectively to $R^\nat_\pi(B^3, A_1\cup A_2)$. These define parameterizations of two smooth circles.
\begin{equation}\label{twocircles}
\begin{array}{rcccrcc}
 \rho(\beta):~~~~~~a&\mapsto& \bbi&\hspace{.5in}& \rho '(\beta):~~~~~~ a&\mapsto& \bbi \\
  b&\mapsto &
e^{(\beta+\epsilon f(\beta)) \bbk} \bbj&&
b&\mapsto &
-e^{(\beta+\epsilon f(\beta)) \bbk} \bbj\\
c&\mapsto & 
e^{(\beta-\epsilon f(\beta)) \bbk} \bbj&&c&\mapsto & 
-e^{(\beta-\epsilon f(\beta)) \bbk} \bbj
\\
d&\mapsto & e^{-2\epsilon f(\beta) \bbk} \bbi&& d&\mapsto & e^{-2\epsilon f(\beta) \bbk} \bbi
\\
h&\mapsto & 
-\bbj e^{(-\epsilon f(\beta) ) \bbk}&&
h&\mapsto & 
\bbj e^{(-\epsilon f(\beta) ) \bbk}\\
p&\mapsto & e^{\epsilon f(\beta) \bbk}&&p&\mapsto & e^{\epsilon f(\beta) \bbk}\\
w&\mapsto & -1&&w&\mapsto & -1
\end{array}
\end{equation}

 We next observe that the  second circle of representations, $\rho'$, is simply a conjugate of a reparametrization of the first one, $\rho$.  
Indeed, from the odd symmetry of the function $f$, and the fact that conjugation by $\bbi$  sends $\bbj\mapsto -\bbj$ and $\bbk\mapsto -\bbk$, a straightforward calculation using (\ref{twocircles}) shows that  $\bbi$ conjugates $\rho'({2\pi-\beta})$ to $ \rho(\beta)$.

Finally, we show that in the circle $\rho $,  distinct values $\beta_1, \beta_2\in [0,2\pi)$ never give conjugate representations (which includes {\em equal representations}).  Suppose that conjugation by some $g\in SU(2)$ sends $\rho({\beta_1}) $ to $\rho({\beta_2})$.  Then conjugation by $g$ fixes $\bbi=\rho({\beta_i})(a)$, so $g=e^{\tau \bbi}$ for some $\tau$.   

Considering the real part of the condition that $g \left( \rho({\beta_1})(p)\right)  g^{-1} = \rho({\beta_2})(p)$, we see that $f(\beta_1)=\pm f(\beta_2)$.  The equation $f(\beta_1)= f(\beta_2)=0 $ only occurs when 
 $\{\beta_1, \beta_2\} = \{ 0, \pi\}$, by our assumptions about the function $f$.  Then consideration of the image of $b$ shows $g\bbj g^{-1}=-\bbj$ (so $g=\pm \bbi$), but  the image of $h$ gives the contradictory condition that $g\bbj g^{-1}=\bbj$.  
 
 For any other pair of $\beta$ values, 
$g \left( \rho({\beta_1})(p)\right)  g^{-1} = \rho({\beta_2})(p)$ is impossible unless 
$g \in \{\pm 1, \pm \bbi\}$.
The cases $g=\pm 1$ imply $\rho({\beta_1})=\rho({\beta_2})$, which is easy to rule out by considering the images of $p$ and $b$.  

If $g=\pm \bbi$, then 
$\rho({\beta_2})(p)=g\left(\rho({\beta_1})(p)\right)g^{-1}$ implies that $f(\beta_2)=-f(\beta_1)$.  But then $\rho({\beta_2})(h)=g\left(\rho({\beta_1})(h)\right)g^{-1}$
implies that $\epsilon f(\beta_2)=-\epsilon f(\beta_1)+ \pi \pmod{2\pi}$, which is impossible given that $\epsilon |f(x)| < \tfrac \pi 2$.  This shows that $\rho(\beta)$, $\beta\in [0,2\pi)$, represent distinct conjugacy classes of representations.

 \end{proof}

    \begin{figure}[H]
\begin{center}

\def\svgwidth{3.2in}
 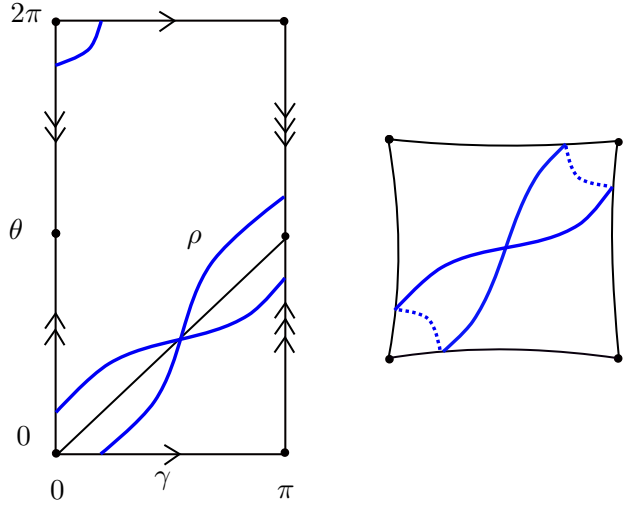

 \caption{\label{fig9} The restriction to the pillowcase of the circle  of perturbed traceless representations of the 3-ball containing a trivial tangle.  The black arc in the figure on the left is the image of the unperturbed traceless representations, and the blue circle is the image of the perturbed traceless representations.}
\end{center}
\end{figure}

Theorem \ref{pertvar} allows us to determine the image $R^\natural_\pi(B^3,A_1\cup A_2)\to R(S^2,\{a,b,c,d\})$, as well as to determine the properties of the limit of  $R^\natural_{\pi}(B^3,A_1\cup A_2)\to R(S^2,\{a,b,c,d\})$ as $\epsilon\to 0$. 
 To emphasize  the dependence on the parameter $\epsilon$, we write 
 $$R^\nat_{(P,f,\epsilon)}(B^3,A_1\cup A_2)\cong\{\rho^\epsilon(\beta)~|~\beta\in \RR/2\pi\ZZ\}$$
 where $\rho^\epsilon $ denotes the   circle of representations of Equation (\ref{twocircles}).

Theorem  \ref{pertvar}  shows that the circle fibers of the restriction $R^\natural(B^3,A_1\cup A_2)\to R(S^2,\{a,b,c,d\})$ have been perturbed away to pairs of points in the restriction 
 $ R^\nat_\pi(B^3, A_1\cup A_2)\to R(S^2,\{a,b,c,d\})$.  We make this more precise in the following Corollary.

 \begin{cor}\label{corimage} Given perturbation data $\pi=(P,f, \epsilon)$ with $P$ as in Figure \ref{fig8}, the restriction map to the pillowcase
  $$ R^\nat_{(P,f,\epsilon)}(B^3, A_1\cup A_2)\to R(S^2,\{a,b,c,d\})$$   is an immersion of 
  the circle $\{\rho(\beta)~|~\beta\in \RR/2\pi\ZZ\}$ for $\epsilon>0$ small.
  Its image is the image of the curve in $\RR^2$,
  $\beta\mapsto
 (\beta+ \epsilon f(\beta)+ \tfrac\pi 2 ,  \beta- \epsilon  f(\beta)+ \tfrac\pi 2 )$
 under the canonical projection of $\RR^2$ to the pillowcase.

The limit   
$$\lim_{\epsilon\to 0}R^\natural_{P, f,\epsilon}(B^3,A_1\cup A_2)\subset R^\natural(B^3,A_1\cup A_2)$$ is a circle consisting of those representations which send $h$ to 
 $\pm \bbj$, i.e.~ those points with $\tau=\pm \tfrac{\pi}{2}$ in Proposition \ref{unpert}. 
 Its projection to the pillowcase is a map from the circle onto the the diagonal  $\beta\mapsto (\beta+\tfrac\pi 2, \beta+\tfrac \pi 2)$, which is a 2-1  immersion except at the corners when $\beta=\tfrac \pi 2$ and $\tfrac {3\pi} 2$. 
 
  \end{cor}
 \begin{proof}
The statements all follow immediately from the formulas of  Theorem \ref{pertvar} except possibly the fact that the circle $ \{\rho ^\epsilon(\beta)\}$ is immersed in the pillowcase when $\epsilon>0$ is small enough. 

This follows from the fact that the smooth embedding
 $$\RR\to \RR^2, ~ \beta\mapsto
 (\beta+ \epsilon f(\beta)+ \tfrac\pi 2 ,  \beta- \epsilon  f(\beta)+ \tfrac\pi 2 )$$
 avoids the branch points $(\pi\ZZ)\times(\pi\ZZ)$ in the branched cover $\RR^2\to R(S^2,\{a,b,c,d\})$ of the pillowcase. 
 
 To see why this is true, recall that the branch points are the points $\{(\pi k,\pi\ell)~|~k,\ell\in \ZZ\}$. Suppose  $(\beta+ \epsilon f(\beta)+ \tfrac\pi 2 ,  \beta- \epsilon  f(\beta)+ \tfrac\pi 2 )=(\pi k,\pi\ell)$, then  $\epsilon f(\beta)=\tfrac \pi 2(k-\ell)$, and hence $k=\ell$ and $f(\beta)=0$.  But we also have $\beta+\tfrac\pi 2 =\tfrac \pi 2(k+\ell)$, so that $\beta=
 \tfrac \pi 2 (2\ell-1)$ and hence $f(\beta)\ne 0$, since $f$ only vanishes for multiples of $\pi$.
  \end{proof}

Figure \ref{fig9} illustrates the situation when $f(x)=\sin(x)$ and $\epsilon=0.2$.  The immersed blue circle (which by abuse of notation we also denote $\rho $) has one double point, corresponding to the  parameter values  $\beta=0,\pi$.  The black diagonal arc corresponds to the image of the unperturbed moduli space, each point in the interior of the black arc corresponding to a latitudinal circle in the 2-sphere $R^\nat(B^3,A_1\cup A_2)$ and the endpoints corresponding to the poles. As $\epsilon$ shrinks towards $ 0$, the circle approaches the black curve, with pairs of points collapsing to a single point along the interior of the arc. In the limit, the  pair of distinct representations on each circle fiber over the black curve are the two unperturbed representations of Proposition \ref{unpert} sending $h$ to $\bbj$ and $-\bbj$ and which project to the given point on the black arc in the pillowcase.

 Notice that  the restriction map $R^\natural_\pi(B^3,A_1\cup A_2)\to R(S^2,\{a,b,c,d\})$ is not an embedding and encloses zero {\em signed} area in the pillowcase.  This can   presumably   be explained by the fact that the image has to lift to a  Legendrian circle with respect to a  natural connection in a Chern-Simons $U(1)$ bundle over $R(S^2,\{a,b,c,d\})$ as in the situation for closed surfaces, c.f.~\cite{RSW, herald}.

\section{The  intersection picture} \label{lip} 

The following corollary summarizes the results of Theorem \ref{pertvar} and Corollary
\ref{corimage} in a statement that suggests an  intersection picture corresponding to   Diagram  (\ref{SVKD}).  Perturbing along the unknotted curve $P$ in a 3-ball  corrects, in a manner which is independent of the pair $(Y,K)$,  for the problem that the unperturbed Chern-Simons functional is never Morse.

\begin{cor}\label{Lagrangian} Suppose that $K\subset Y$ is a knot and $B^3\subset Y$ is a 3-ball intersecting $K$ transversally in two trivial arcs $A=A_1\cup A_2$.  Let $(Y_0,K_0):=(Y\setminus B^3, K\setminus A)$.

Assume that $R(Y_0,K_0)$ is a smooth 1-manifold away from a finite number of points and that the restriction to the pillowcase $$R(Y_0,K_0)\to R(S^2,\{a,b,c,d\})$$ is an immersion transverse to the arc $\theta=\gamma$ on the manifold points, and takes the non-manifold points outside a neighborhood of the arc $\theta=\gamma$.    Then for small enough choice of perturbation, the intersection in the pillowcase of
$R(Y_0,K_0)$ and the circle $\rho$ is transverse, and hence the set $R^\nat_\pi(Y,K)$ is non-degenerate (and finite).

\qed 
\end{cor}

In complete generality, achieving the conditions on $R(Y_0,K_0)$ will also involve suitable perturbations along curves in $Y_0\setminus K_0$.  One of our primary goals is   to ``combinatorialize'' the instanton homology  of a knot  in terms of the intersection picture given by Diagram (\ref{SVKD}) and Corollary \ref{Lagrangian}.    In later work we will explore the calculation of   gradings, and explore  differentials in terms of this picture.    The reader should look at Figures \ref{fig12} and \ref{fig14} below
 for an illustration of Corollary \ref{Lagrangian}.

 \bigskip

 \section{The unreduced case}\label{unreduced case}
 Kronheimer-Mrowka construct two versions of their singular instanton homology for $(Y,K)$. The first, {\em reduced} instanton homology $I^\nat(Y,K)$,  corresponds to 
 taking the connected sum $K^\nat$ of $K$ with a Hopf link and working with an $SO(3)$ bundle which is non-trivial on the torus which separates the  components of the Hopf link.  For this version of their theory,  the critical set of the Chern-Simons functional is identified with $R^\nat(Y,K)$. For knots in $S^3$ with simple representation varieties (or for all knots, after applying a perturbation outside a neighborhood of $K$ which contains $H$ and $W$),  $R^\nat(S^3,K)$ consists of one circle for each non-abelian point in the space $R(S^3,K)$ of traceless representations, and  one non-degenerate isolated point corresponding to the abelian representation in $R(S^3,K)$. Theorem \ref{pertvar}  then shows how to perturb along one curve $P$ to turn each circle into a non-degenerate pair of isolated points.
 
 The second version, {\em unreduced} instanton  homology $I^\sharp(Y,K)$  corresponds to 
 taking the disjoint union  $K^\sharp$ of $K$ with a Hopf link $H_1\cup H_2$ rather than the connected sum.  The critical set of the corresponding Chern-Simons functional is $R^\sharp(Y,K)=R^\nat(Y, K\cup H_1)$.  For knots in $S^3$ with simple representation varieties (or for all knots, after further perturbation),  $R^\sharp(S^3,K)$ consists of a copy of $SO(3)$   for each non-abelian point in the space $R(S^3,K)$ of traceless representations, and  one 2-sphere corresponding to the abelian representation in $R(S^3,K)$. This follows quickly from Lemma \ref{double} by decomposing along the 2-sphere separating the Hopf link from $K$, by similar but easier versions of Propositions \ref{circle or point} and \ref{natcase}.  In this section we  prove a counterpart to Theorem \ref{pertvar} for  $R^\sharp(Y,K)$ by using two perturbation curves. 
 
\medskip
 
Given a knot (or link) $K$ in a 3-manifold $Y$, consider a 3-ball $B^3\subset Y$ intersecting $K$ in two unknotted arcs $A_1\cup A_2$. Place the Hopf link $H$ inside $B^3$ and place an arc $W$ spanning the two components of the Hopf link, as shown by the black curves in Figure \ref{unreduced}. 

As Kronheimer-Mrowka observe in \cite{KM1,KM-khovanov}, every conjugacy class of representations which take the meridians of $K$ and $H$ to traceless matrices and the meridian of $W$ to $-1$ can be uniquely conjugated so that the  two meridians of $H$ are sent to $\bbi$ and $\bbj$ respectively. This shows that the corresponding unperturbed representation space $R^\sharp(Y,K)$ is homeomorphic to the space $\tilde R(Y,K)$ of all traceless representations of $\pi_1(Y\setminus K)$ (not conjugacy classes).  Since the orbits of the $SU(2)$ conjugation action on $\tilde R(Y,K)$ have dimension greater than 1, $R^\sharp(Y,K)$ is never finite.  For knots in $S^3$ with simple representation varieties, $R^\sharp(Y,K)$ is a disjoint union of a 2-sphere corresponding to the distinguished representation $\alpha$ and a number of copies of $SO(3)=SU(2)/\pm1$.

 Place two smaller 3-balls $B^3_1$ and $B^3_2$ in the interior of $B^3$ as illustrated in Figure \ref{unreduced}.  Denote their boundaries by $S^2_i=\partial B^3_i$.  Place a perturbation curve  $P_1$ inside $B^3_1$ and a second perturbation curve $P_2$ inside $B^3_2$, as indicated.  Label the various meridians $a,b,c,d,m,n,$ and $p_2$, as indicated.

     \begin{figure}[H]
\begin{center}

\def\svgwidth{3.7in}
 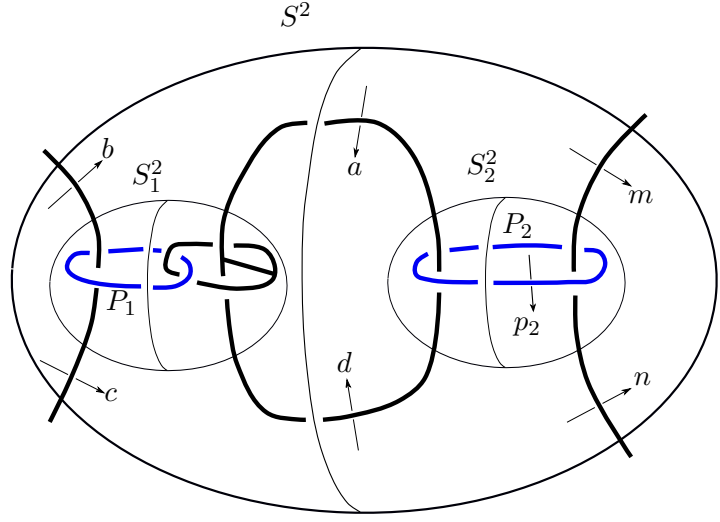
  \caption{The ``$\sharp$" construction and two perturbation curves. The black Hopf link in the center and the arc connecting its components represents the additional data defining the $\sharp$ construction relevant to unreduced singular instanton homology.  The holonomy perturbation takes place in the neighborhood of the blue circles $P_1$ and $P_2$. \label{unreduced}}
\end{center}
\end{figure} 
 
 Fix  $\epsilon>0$, and choose the perturbation function for the curve $P_1$ to be 
 $\epsilon\sin(x)$. For the curve $P_2$, take the perturbation function $2\epsilon \sin(x)$.   With this perturbation data $\pi$ in place, we obtain a space 
 $R^\sharp_\pi(B^3, A_1\cup A_2)$ and a restriction map to the pillowcase:
 $$R^\sharp_\pi(B^3, A_1\cup A_2)\to R(S^2,\{m,b,c,n\}).$$

Let $\arcsin:[-1,1]\to [-\frac{\pi}{2},\frac{\pi}{2}]$ denote the inverse sine function and define two smooth functions
\begin{equation}\label{taudefs}\begin{split}
&\tau_1:S^1\to [-\tfrac{\pi}{6},\tfrac{\pi}{6}],~\tau_1(\beta)=\arcsin(-\tfrac{1}{2}\sin(\beta)),
\text{~and}\\
& \tau_2 :S^1\to [\tfrac{5\pi}{6},\tfrac{7\pi}{6}], ~\tau_2(\beta)=\pi-\tau_1(\beta).
 \end{split}\end{equation}
  The set $\{(\beta,\tau)\in S^1\times S^1~|~ \sin\beta=-2\sin \tau\}$ is precisely the disjoint union of the graphs of $\tau_1(\beta)$ and $\tau_2(\beta)$.

 \begin{thm}\label{pertvar2} The space $R^\sharp_\pi(B^3, A_1\cup A_2)$ is the disjoint union of two   circles, parameterized by $\beta\in S^1$ and $i=1,2$. These satisfy 
 $$m\mapsto \bbi,~ b\mapsto e^{ -\tau_i(\beta)\bbk}e^{(\frac{\pi}{2} +\beta +\epsilon\sin\beta)\bbk}\bbi,~
 c\mapsto e^{- \tau_i(\beta)\bbk}e^{(\frac{\pi}{2} +\beta -\epsilon\sin\beta)\bbk}\bbi,~
 n\mapsto e^{-2\epsilon\sin(\beta)\bbk}\bbi$$
 and 
 $$a\mapsto e^{-\tau_i(\beta)\bbk}\bbi, ~d\mapsto e^{(-\tau_i(\beta)-2\epsilon\sin(\beta))\bbk}\bbi.$$
 In particular the restriction map 
 $R^\sharp_\pi(B^3, A_1\cup A_2)\to R(S^2,\{m,b,c,n\})$ is given by 
 $$\rho_i(\beta)= 
 \big(-\tau_i(\beta) + \frac{\pi}{2} +\beta +\epsilon\sin\beta,
 -\tau_i(\beta) + \frac{\pi}{2} +\beta -\epsilon\sin\beta\big), ~ i=1,2.$$
 \end{thm}
 \begin{proof} 
 
 Choose $\rho\in R^\sharp_\pi(B^3, A_1\cup A_2)$.
  Theorem \ref{pertvar} applied to the sphere $S^2_1$ 
implies that $\rho$ may be uniquely conjugated to to a representation satisfying 
 \begin{equation*}
  a\mapsto \bbi, ~b\mapsto e^{\left(\tfrac \pi 2 + \beta + \epsilon \sin(\beta) \right) \bbk}\bbi, ~c\mapsto e^{\left(\tfrac \pi 2 + \beta -\epsilon \sin(\beta) \right)\bbk}\bbi, ~d\mapsto e^{ -2\epsilon \sin(\beta)\bbk}\bbi  \end{equation*}
for some $\beta\in [0,2\pi)$. 
 Proposition \ref{s2prop}, applied to $S^2$ (or $S^2_2$), then implies that there exists a $\tau\in S^1$ so that  $$m\mapsto e^{\tau\bbk}\bbi\text{~and~} n\mapsto e^{(\tau-2\epsilon\sin\beta)\bbk}\bbi.$$
 
  The fundamental group $\pi_1(B^3_2\setminus(K\cup H\cup P_2))$ is generated by $a,d,p_2, m,n$ subject to the relations $d=\bar{p}_2ap_2, n=
  \bar{p}_2mp_2$, and $[\bar am, p_2]=1$. 
   The longitude of $P_2$ is equal to $\bar a m$, which is sent to 
 $$-\bbi e^{\tau\bbk}\bbi=e^{-\tau\bbk}.$$
 The perturbation condition for $P_2$ then says that the meridian $p_2$ is sent to $e^{-2\epsilon\sin(\tau)\bbk}$. 
  The relation $d=\bar{p}_2ap_2$ then implies
 $$e^{ -2\epsilon \sin(\beta)\bbk}\bbi= e^{2\epsilon\sin(\tau)\bbk}\bbi e^{-2\epsilon\sin(\tau)\bbk}=e^{4\epsilon\sin(\tau)\bbk}\bbi$$
 so that (since $\epsilon$ is small)
 $\sin\beta=-2\sin\tau$.
 The relation $ n=
  \bar{p}_2mp_2$  places the same restriction  $\sin\beta=-2\sin\tau$.
   Thus $\tau=\tau_1(\beta)$ or $\tau=\tau_2(\beta)$.

Conversely, given any $\tau$ satisfying  $\sin\beta=-2\sin\tau$, the assignment
\begin{equation}\label{conjugated}\begin{split}
 a\mapsto \bbi, ~b\mapsto e^{\left(\tfrac \pi 2 + \beta + \epsilon \sin(\beta) \right) \bbk}\bbi, ~c\mapsto e^{\left(\tfrac \pi 2 + \beta -\epsilon \sin(\beta) \right)\bbk}\bbi,  ~d\mapsto e^{ -2\epsilon \sin(\beta)\bbk}\bbi, \\
 m\mapsto e^{\tau\bbk}\bbi,~ n\mapsto e^{(\tau-2\epsilon\sin\beta)\bbk}\bbi,~
 p_2\mapsto  e^{-2\epsilon\sin(\tau)\bbk}\end{split}\end{equation}
 uniquely defines a representation whose conjugacy class lies in $R^\nat_\pi(B^3, A_1\cup A_2\cup H)$.
 
 Conjugating by $e^{-\frac{\tau}{2}\bbk}$ completes the proof.
  \end{proof}
 
  Figure \ref{unpertpic} illustrates the image of the two circles of Theorem \ref{pertvar2}, as well as the circle of Theorem \ref{pertvar}.  The red  and green circles correspond to $\rho_1$ and $\rho_2$, respectively, from  Theorem \ref{pertvar2}. The blue circle corresponds to the circle $\rho$ of Theorem \ref{pertvar} and Corollary \ref{corimage}. In this figure we used a moderately sized perturbation ($\epsilon=.4$) to highlight the fact that these three circles map to three distinct (but close) immersed circles in the pillowcase. It is straightforward to check that (just as in the case of $\rho$) for $i=1,2$, the map $\rho_i:S^1\to R(S^2, \{m,b,c,n\})$ is an immersion with a single double point corresponding to $\rho_i(0)=\rho_i(\pi)$.
  As $\epsilon\to 0$, each circle  limits to a generically 2-1 map onto the diagonal arc $\gamma=\theta$.
 
    \begin{figure}[H]\centering
\def\svgwidth{2.3in}
 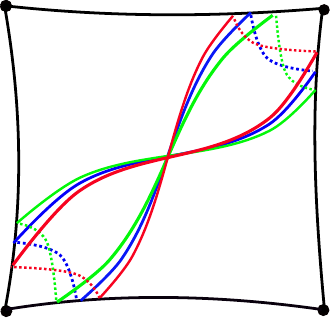

 \caption{Figure illustrating Theorem \ref{pertvar2}.  The theorem analyzes the perturbed traceless representations of the ball from Figure \ref{unreduced}.  There are two circles, $\rho_1$ and $\rho_2$, of such representations, and the figure shows their restriction to  the pillowcase.  These are the red and green circles.  Also shown (in blue) is the image of the circle of representations $\rho$, from  Theorem \ref{pertvar}. \label{unpertpic}}
\end{figure}
  
 In particular, Corollary  \ref{Lagrangian} has the following unreduced counterpart.
 
 \begin{cor}\label{Lagrangian2}\label{corpert2} Suppose that $K\subset Y$ is a knot and $B^3\subset Y$ is a 3-ball intersecting $K$ transversally in two trivial arcs $A=A_1\cup A_2$.  

Assume that $R(Y_0,K_0)=R(Y\setminus B^3, K\setminus A)$ is a smooth 1-manifold away from a finite number of points and that the restriction to the pillowcase $$R(Y_0,K_0)\to R(S^2,\{a,b,c,d\})$$ is an immersion transverse to the arc $\theta=\gamma$ on the manifold points, and takes the non-manifold points outside a neighborhood of the arc $\theta=\gamma$ (or that a suitable perturbation in $Y_0\setminus K_0$ has been applied to achieve these conditions).

Then for a small enough choice of perturbation, the intersection in the pillowcase of
$R(Y_0,K_0)$ and each circle  $\rho_1,\rho_2$ is transverse, and hence the set $R^\sharp_\pi(Y,K)$ is non-degenerate (and finite), and contains two points for each point of   $R^\nat_\pi(Y,K)$.
\qed 
\end{cor}

 In particular, for knots in the 3-sphere with simple representation varieties, the 2-sphere corresponding to the representation $\alpha_{\pi/2}$ is perturbed into a pair of isolated points and each $SO(3)$ component is perturbed into four isolated points. This is the unreduced counterpart to the fact that in the reduced case, the isolated representation $\alpha$ perturbs to a non-degenerate isolated point and each circle perturbs to a pair of isolated points.

 \medskip

For convenience, we will restrict our calculations in Section \ref{data}  below to  the reduced instanton chain complex $CI^\nat(Y,K)$.  By using Theorem \ref{pertvar2}, each calculation has its unreduced counterpart, and in each case $CI^\sharp(Y,K)$ has twice as many generators as $CI^\nat(Y,K)$.

\section{Examples: 2-bridge knots}\label{twobridgesection}

 Having worked out the perturbation picture inside the 3-ball, completing the analysis of  Diagram (\ref{SVKD}) is reduced to understanding the restriction map  from $R(Y_0,K_0)=R(Y\setminus B^3, K\setminus (A_1\cup A_2))$ to the pillowcase for various   $(Y,K)$. We explore this in detail for 2-bridge knots and torus knots in this and the following section.  In contrast to the analogous  question for knot complements and the image of their full representation varieties in the pillowcase as the character variety of the torus, the situation is much simpler (in fact linear) for 2-bridge knots and  complicated for torus knots. 
 
 \medskip
Consider the arcs $A_1(n)\cup A_2(n)\subset B^3$ indicated in Figure \ref{ntwist}, where $n$ refers to the number of half-twists (positive or negative according to the sign of $n$). Orient these arcs arbitrarily and let $a,b,c,d$ denote their oriented meridians as indicated.

  \begin{figure}[H]
\begin{center}
\def\svgwidth{2.5in}
 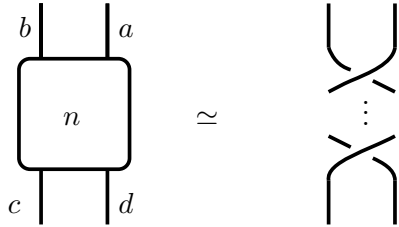
 \caption{\label{ntwist} A box labeled ``$n$" indicates $n$ positive or negative crossings, according to whether $n>0$ or $n<0$.}
\end{center}
\end{figure}

It is straightforward to check that if $\pi_1(B^3\setminus(A_1(n)\cup A_2(n))\to SU(2)$ is a representation given by 
$$b\mapsto e^{x\bbk}\bbi, ~a\mapsto e^{y\bbk}\bbi,~c\mapsto e^{z\bbk}\bbi,~d\mapsto e^{w\bbk}\bbi,$$
then 
\begin{equation}\label{eqn18}
\begin{pmatrix} x\\ y\end{pmatrix}= \begin{pmatrix} 1+n&-n\\ n&1-n\end{pmatrix} \begin{pmatrix} z \\ w\end{pmatrix}.
\end{equation}
Note that this formula holds with any choice of orientations of the arcs and choice of sign of $n$.

Let $a_1,a_2,\cdots ,a_m$ be integers, with $m$ odd, and define the rational number $p/q$ with $p$ and $q$ relatively prime by the continued fraction expansion:
\begin{equation}
\label{CFE}
\frac{p}{q}=a_1+ \frac{1}{a_2+\frac{1}{a_3+\cdots}}.
\end{equation}
Consider the diagram of the 2-bridge knot $K=K(p/q)$ associated to this sequence (see \cite{Burde-Zieschang}), with a 3-ball intersecting in  unknotted arcs as indicated in Figure  \ref{2-bridgeK}. 
We assume   $p$ is odd so that $K(p/q)$ is a knot, not a link. 
\medskip

Our notational convention is that  of Burde-Zieschang \cite{Burde-Zieschang} and is consistent with the convention that the 2-fold branched cover of $K(p/q)$ is $L(p,q)$, where $L(p,q)$ is oriented as the quotient of $S^3$, in other words $L(p,q)$ is $-p/q$ surgery on the unknot.  Some other authors use other conventions.

  \begin{figure}
\begin{center}
\def\svgwidth{2.4in}

 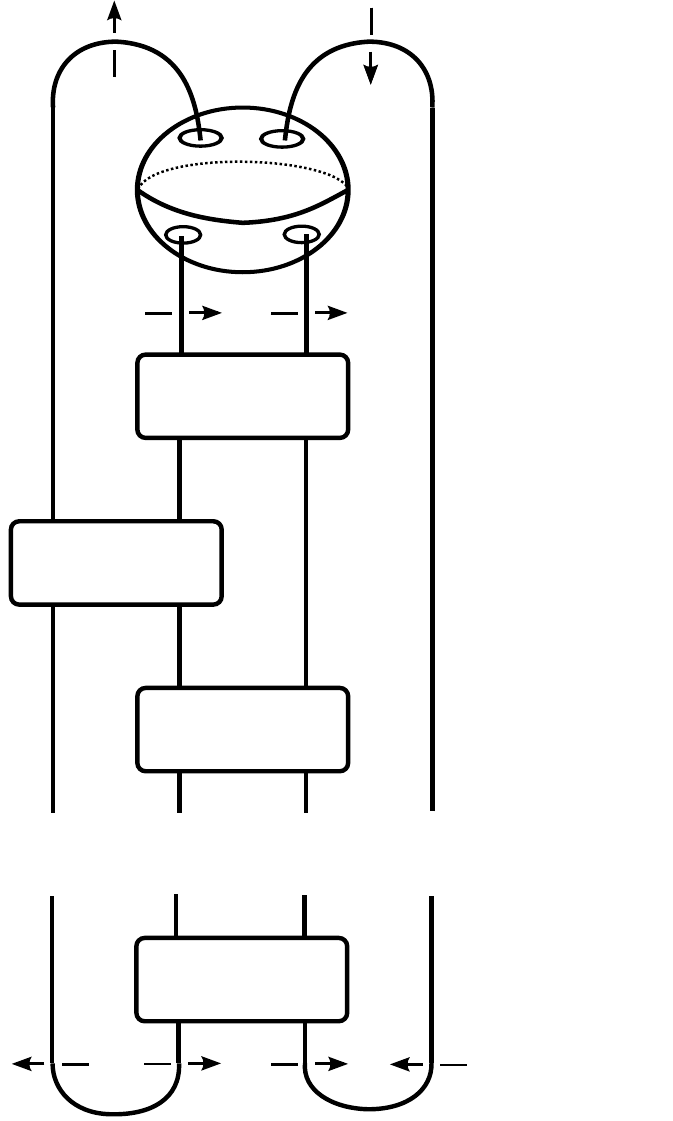
 \caption{\label{2-bridgeK}}
\end{center}
\end{figure}

The complement $S^3\setminus (K\cup B^3)$ is homeomorphic to a 3-ball with two unknotted arcs removed, and hence Proposition \ref{unpert0} implies that  $R(S^3\setminus B^3, K\setminus (A_1\cup A_2))$ is an arc 
 $$\tau_t:\pi_1(S^3\setminus (K\cup B^3))\to SU(2), t\in [0,\pi]$$  
 determined by $$ \tau_t(a')= \bbi,~ \tau_t(b')=\bbi, ~ \tau_t(c')=e^{t\bbk}\bbi,~ \tau_t(d')=e^{t\bbk}\bbi, $$
 where $a',b',c', d'$ are illustrated in Figure \ref{2-bridgeK}.

Then $ \tau_t(a)=\bbi$, and writing 
$$ \tau_t(b)=e^{\gamma\bbk}\bbi, ~ \tau_t(c)=e^{\theta\bbk}\bbi,~ \tau_t(d)=e^{(\theta-\gamma)\bbk}\bbi,$$
 Equation (\ref{eqn18}) implies that   
 \begin{equation}
\label{matrixeqn}
\begin{pmatrix} \gamma(t)\\ \theta(t) \\ \theta(t)-\gamma(t) \end{pmatrix}=
M(-a_1)N(a_{2})M(-a_{3})\cdots N(a_{m-1})M(-a_m)\begin{pmatrix} t\\ t \\ 0 \end{pmatrix}
\mod 2\pi,
\end{equation}
 
 where
$$
M(a)=\begin{pmatrix} 1&0&0\\ 0&1+a&-a\\ 0&a&1-a \end{pmatrix}\text{~and~}
N(a)=\begin{pmatrix} 1+a&-a&0\\ a&1-a&0\\ 0&0&1 \end{pmatrix}.
$$

Thus the path $\tau_t$ will restrict to a linear path in the pillowcase of the form $$(\gamma(t), \theta(t))=(mt, nt), t\in [0,\pi]$$ for some integers $m,n$.  More precisely, this linear path in $\RR^2$  projects to a path in the pillowcase via the branched cover $\RR^2\to R(S^2,\{a,b,c,d\})$.  We next show that $(m,n)=(q, q-p)$.

\begin{lem}\label{pandq}   Let  $p,q$  be relatively prime non-zero integers   and $a_1,a_2,\cdots ,a_m$ a sequence giving the continued fraction expansion for $p/q$.
Then the solution to Equation (\ref{matrixeqn}) is 
 $$( \gamma(t),\theta(t))= \pm (qt,  (q-p)t).$$
\end{lem}

\begin{proof} Note that 
$$M(-a_m)\begin{pmatrix} 1\\ 1 \\ 0 \end{pmatrix}
=\begin{pmatrix} 1\\ 1-a_m \\ -a_m \end{pmatrix}$$

  Suppose by induction that 
 $$M(-a_3)N(a_4)\cdots N(a_{m-1})M(-a_m)\begin{pmatrix} 1\\ 1 \\ 0 \end{pmatrix}
=\begin{pmatrix} s\\ s-r \\ -r \end{pmatrix}$$
for some relatively prime pair of integers $r,s$ so that 
$\frac{s}{r} $
has continued fraction expansion given by $ a_3,a_4,\cdots, a_m $.
Then 
$$M(-a_1)N(a_2)\cdots M(-a_m)\begin{pmatrix} 1\\ 1 \\ 0 \end{pmatrix}
=M(-a_1)N(a_2)\begin{pmatrix} s\\ s-r \\ -r \end{pmatrix}=
 \left( \begin {array}{c} s+{ a_2}\,r\\ \noalign{\medskip}s-r+{ 
a_2}\,r-{ a_1}\,s-{ a_1}\,{ a_2}\,r\\ \noalign{\medskip}-{ a_1}
\,s-{ a_1}\,{ a_2}\,r-r\end {array} \right). $$
Setting $p_0=a_1s+a_1a_2r+r$ and $q_0=s+a_2r$ it is easy to check that $p_0$ and $q_0$ are relatively prime and 
$$\frac{p_0}{q_0}=a_1+ \frac{1}{a_2+\frac{s}{r}}.$$
In particular, $\frac{p_0}{q_0}$ has continued fraction expansion given by $ a_1,\cdots, a_m $ so that $(p_0,q_0)=\pm (p,q)$, and 
$$M(-a_1)N(a_2)\cdots M(-a_m)\begin{pmatrix} t\\ t \\ 0 \end{pmatrix}
= \pm t \begin{pmatrix} q\\ q-p \\ -p \end{pmatrix}.$$ 
\end{proof}

The two paths
$$t\mapsto (qt, (q-p)t), ~t\mapsto (-qt, -(q-p)t),~ t\in [0,\pi]$$
are identical as maps to the pillowcase $R(S^2,\{a,b,c,d\})=\RR^2/\sim$. Hence the sign ambiguity in Lemma \ref{pandq} does not affect the image in the pillowcase.

\medskip
 The intersection of the curve $(\gamma(t),\theta(t))=(qt, (q-p)t),~t\in [0,\pi]$ with the curve $\gamma=\theta$ in $R(S^2,\{a,b,c,d\})$ occurs at the $\frac{p+1}{2}$  points
 \begin{equation}
\label{points}x_\ell=(q \tfrac{2\pi \ell}{p}, (q-p)\tfrac{2\pi \ell}{p}),~ \ell =0,1,\cdots, \tfrac{p-1}{2}.
\end{equation}
 Note that the intersection point $x_0$ corresponds to the distinguished  representation $\alpha$ of Equation (\ref{alphadef}).

Combining  this observation with   Proposition \ref{unpert} and  Lemma \ref{double}  one immediately concludes

\begin{thm}\label{2bridgethm} For  the 2-bridge knot  $K=K(p/q)$, the space $R^\nat(S^3, K)$ is a union of circles and one isolated point $\alpha$, 
one circle for each  intersection point $x_\ell, \ell=1,2,\cdots,\frac{p-1}{2} 
$ with $\alpha$ corresponding to $x_0$.
 
 For perturbation data $\pi$ as above, the space $R^\nat_\pi(S^3, K)$ is a union of pairs of isolated non-degenerate points $x_{\ell,1},x_{\ell,2}$, $\ell=1,2,\cdots,\frac{p-1}{2}$, and one additional non-degenerate  point $\alpha'$,
corresponding to the  intersections in $R(S^2,\{a,b,c,d\})=[0,\pi]\times [0,2\pi]/{\sim}$ of   the curve 
$(qt, (q-p)t), ~t\in [0,\pi] $ with the  circle $\{\rho^\epsilon(\beta)\}$.
\end{thm}

\begin{proof}
 The first assertion follows from Proposition \ref{unpert} and  Lemma \ref{double}. The second then follows  similarly  from Corollary \ref{corimage} and  Lemma \ref{double}.
\end{proof}

Note that if we consider the space $\tilde R (S^3, K)$ of traceless $SU(2)$ representations of $K$ ({\em not} modulo conjugation) then the same reasoning gives that for $K=K(p/q)$, $\tilde R (S^3, K)$ is homeomorphic to 
the union of a 2-sphere (the conjugacy class of $\alpha$) and $\frac{p-1}{2}$ copies of $SO(3) =SU(2)/\pm 1$ (the conjugacy classes of each $x_\ell$), giving a different argument for the result of Lewallen \cite[Theorem 2.4]{Lewallen}.  Since the spaces $\tilde R (S^3, K)$ and $ R^\sharp(S^3, K)$ are homeomorphic, Theorem \ref{pertvar2} and Corollary \ref{corpert2} implies that after perturbing along the two curves $P_1,P_2$, $R^\sharp_\pi(S^3,K)$ is the union of  $4\ell$ points $x_{\ell,1},x_{\ell,2},x_{\ell,3},x_{\ell,4}$, $\ell=1,2,\cdots,\frac{p-1}{2}$ and two additional points $\alpha'_1,\alpha'_2$.

\bigskip

We illustrate Theorem \ref{2bridgethm} in a few examples. Consider first  the $(2,n)$ torus knot  $T_{2,n}$, corresponding to $p/q=-n/1$.
Thus the restriction of  $(\gamma(t),\theta(t))$ to the pillowcase is parameterized by $$t\in[0,\pi]\mapsto (t,(n+1)t)\in [0,\pi]\times[0,2\pi].$$

Figure \ref{fig12} illustrates the case of the  right handed trefoil knot, $T_{2,3}=K(-3/1)$, with corresponding curve $t\mapsto (t,4t)$.  The unperturbed moduli space 
$R^\nat(S^3, T_{2,3})$ consists of a isolated point $\alpha$ and a circle $c$. The perturbed moduli space 
$R^\nat_\pi(S^3, T_{2,3})$  with $\pi=(P,\sin(x),\epsilon)$, consists of three isolated representations, $\alpha', x_1,x_2$. As $\epsilon\to 0$, $\alpha'\to \alpha$ and $x_1,x_2$ converge to a pair of antipodal points on the circle $c$. These three points
  are the generators of the instanton knot homology chain complex of the trefoil.

  \begin{figure}
\begin{center}
\def\svgwidth{2in}

 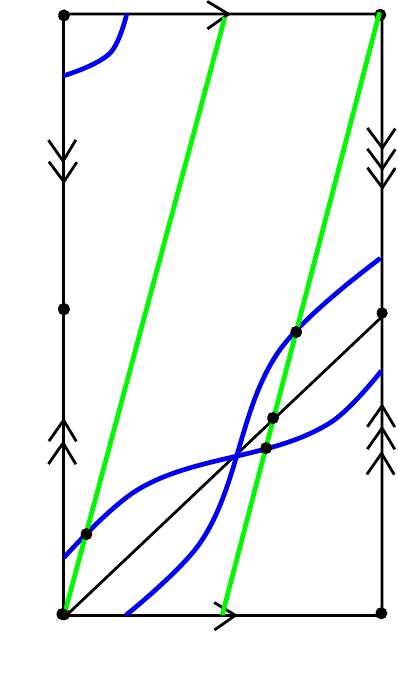

 \caption{$R^\nat$ and $R^\natural_\pi$ for the Trefoil \label{fig12}}
\end{center}
\end{figure}

As a further example Figure \ref{fig14} illustrates the situation for $K$ the Figure 8 knot,  which can be expressed as $K(-5/3)$.  This gives the curve $t\mapsto (3t,8t)$.  We conclude that the unperturbed space $R^\nat(S^3, K)$ consists of a pair of circles $c_x, c_y$ and the isolated point $\alpha$. The perturbed moduli space (generating the instanton chain complex) $R^\nat_\pi(S^3, K)$  contains 5 points, $\alpha',  x_1, x_2, y_1,y_2$.  

 \begin{figure}
\begin{center}

\def\svgwidth{2in}

 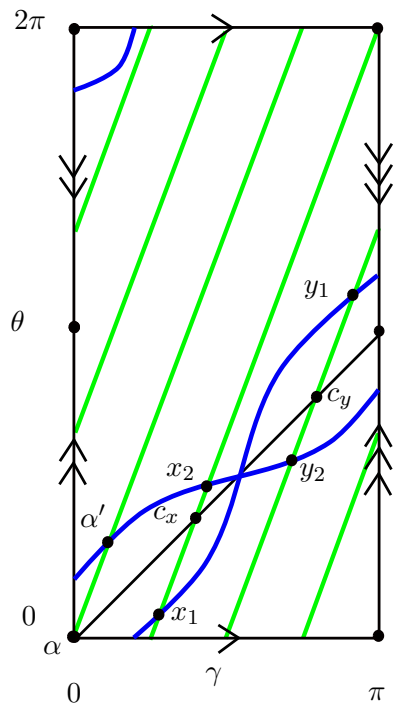

 \caption{$R^\nat$ and $R^\natural_\pi$ for the Figure 8 knot \label{fig14}}
\end{center}
\end{figure}

It is worth recalling that the 2-fold branched cover of $K(p/q)$ is $L(p,q)$, and hence the fraction $p/q$ does not uniquely determine $K=K(p/q)$.  For example the Figure 8 knot can also be described as $K(5/2)$, and hence its character variety $R(S^3,K)$ can also be described as the intersection of $(\gamma,\theta)= (2t, -3t)$ with $\{\gamma=\theta\}$.  This corresponds to the fact that one can choose different (up to isotopy) 3-balls intersecting $K$ in two unknotted arcs, for the same reason that 3-manifolds admit different Heegaard splittings.

\medskip 

As a last example,  Figure \ref{fig72} illustrates the case when $K=7_2=K(-11/5)$. The restriction $R(S^3\setminus B^3, K\setminus (A_1\cup A_2))\to R(S^2,\{a,b,c,d\})$ takes the arc to $(\gamma(t),\theta(t))=(5t,16t)$.
 The perturbed moduli space  $R^\nat_\pi(S^3, K)$  contains 11 points, 
 $$\alpha',  x_1, x_2, y_1,y_2, z_1,z_2,w_1,w_2,v_1,v_2.$$

  \begin{figure}
\begin{center}

\def\svgwidth{2in}

 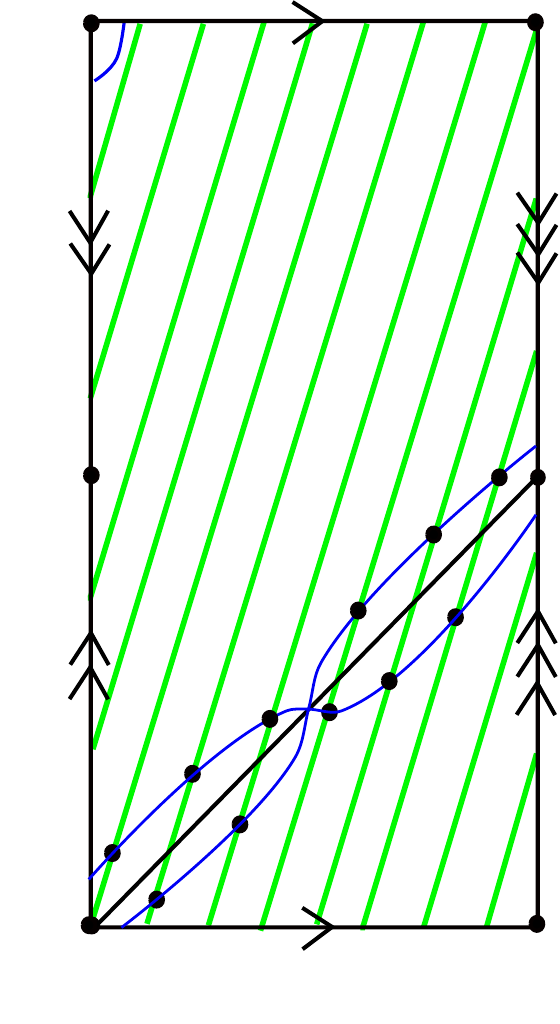

 \caption{$R^\natural_\pi$ for $7_2$ \label{fig72}}
\end{center}
\end{figure}

\section{Examples: Torus knots}\label{torusknotsect}

In contrast to the situation for the full character varieties $\chi(S^3,T_{p,q})$, the varieties of traceless representations $R(S^3,T_{p,q})$ for torus knots are more complicated than those for 2-bridge knots.  We explore the situation in enough detail to establish  that a 3-ball intersecting a torus knot in 2 unknotted arcs can be found so that the restriction to the pillowcase of any non-abelian representation in $R(S^3,T_{p,q})$ avoids the corners.  We then give a method  to describe a 2-variable polynomial which cuts out the traceless character variety of the complement of a 3-ball meeting a torus knot in two arcs.

Note that counting the points of $R(S^3,T_{p,q})$  can be done by looking at the intersections of $\chi(S^3,T_{p,q})$ with the circle $S(\bbi) $ of Equation (\ref{tlcircle}), or equivalently by computing the signature of a torus knot.   We establish that perturbing in a ball using Theorem \ref{pertvar} gives rise to a set of $\sigma(T_{p,q})+1$ generators for the instanton complex $CI^\nat(S^3,T_{p,q})$.

\bigskip
  
  Figure \ref{torusknot} illustrates a $(p,q)$ torus knot $T_{p,q}$ in $S^3$. We view $S^3$ as $\frac{q}{r}$ and $-\frac{s}{p}$ Dehn surgery on the two components of a Hopf link, where $pr+qs=1$.  The knot  $T_{p,q}$ is isotopic to a curve parallel to the first component, a fact which can be verified by identifying the parallel curve with a regular fiber in a Seifert fibering of $S^3$ with singular fibers of order $p$ and $q$. In the figure, $T_{p,q}$ has been isotoped so that it meets a 3-ball in a pair of trivial arcs $A_1\cup A_2$.

  \begin{figure}
\begin{center}

\def\svgwidth{3.5in}

 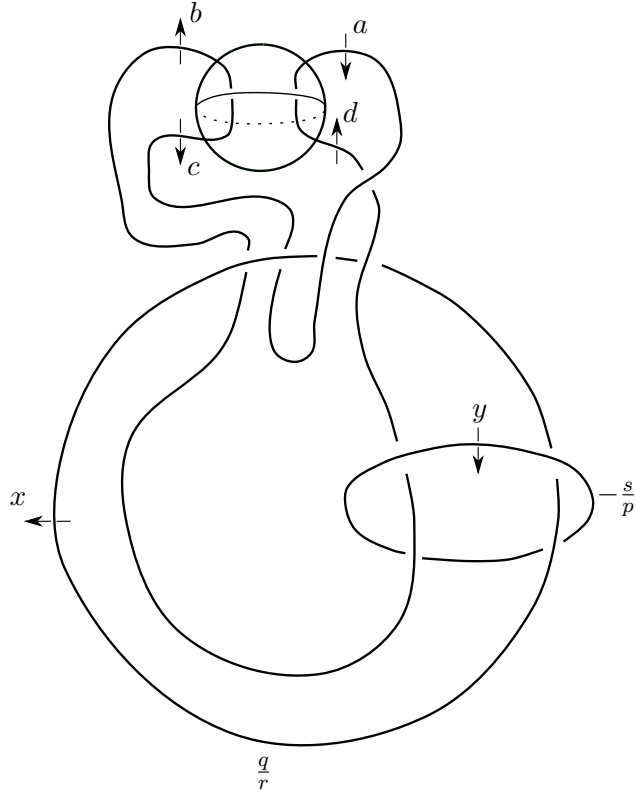
 \caption{The $(p,q)$ torus knot \label{torusknot}}
\end{center}
\end{figure} 

We wish to identify  the space $$R(Y_0,K_0):=R(S^3\setminus B^3, T_{p,q}\setminus(A_1\cup A_2))$$ and its image in the pillowcase.   Note that $$(S^3\setminus B^3)\setminus (T_{p,q}\setminus(A_1\cup A_2))=S^3\setminus (B^3\cup T_{p,q}).$$
A straightforward calculation using the Wirtinger presentation shows that
$\pi_1(S^3\setminus (B^3\cup T_{p,q}))$  has the presentation
$$\langle x,y,a,b,c,d~|~ c=\bar{x}ax, ad\bar{a}=yxb\bar{x}\bar{y}, 
1=[y, xb]= [x, d\bar{a} y]=y^s(xb)^{-p}=x^q(d\bar a y)^r \rangle
  $$
 with $x,y,a,b,c,d$ the generators illustrated in Figure \ref{torusknot}.
 
{\em A priori}, it is clear that $\pi_1((S^3\setminus (B^3\cup T_{p,q}))$ is a free group on two generators: up to homotopy $S^3\setminus (B^3\cup T_{p,q})$ is obtained by gluing two solid tori along a disk, just as a torus knot complement is obtained by gluing two solid tori along an annulus.   Word manipulation in the presentation above, however, provides an explicit identification with a free group on generators
 $$A=(xb)^qy^r \ \ \text{and}\  \ B=(d\bar a y)^{-s} x^p.$$
 Manipulating words shows that 
 $$A^s=xb, ~A^p=y, ~B^r=x, ~B^{-q}=d\bar a y.$$
 Solving for $a,b,c,d$ (using the relation $d\bar a=\bar c d$) yields
 \begin{equation}
\label{words}
a=A^{s+p}B^{q-r}, ~b=B^{-r}A^s, ~ c=B^{-r}A^{s+p}B^q=B^{-r}aB^r,~
d=B^{-q}A^sB^{q-r}=B^{-(q-r)}bB^{q-r}.
\end{equation}

 \medskip
 
 Since $\pi_1((S^3\setminus (B^3\cup T_{p,q}))$ is free on $A$ and $B$,
 the assignment $A\mapsto M, B\mapsto N$ gives a representation for any pair $M,N\in SU(2)$.   A general such assignment will not  yield a traceless representation. 
 However, writing the pair $(M,N)$ as $$M=e^{uQ}, N=e^{vR}, \text{~where~}
 Q,R\in C(\bbi), ~u,v\in [0,\pi],$$
then the corresponding representation  determines a point in $R(Y_0,K_0)$  if and only if the images of $a$ and $b$
\begin{equation}\label{teq0}
M^{s+p}N^{q-r}\text{~and ~} N^{-r}M^s
\end{equation}
are traceless, or, equivalently, if and only if
\begin{equation}
\label{teq}
0=\Real(e^{(s+p)uQ} e^{(q-r)vR})\text{~and ~} 0=\Real(e^{-rvR}e^{suQ} ).
\end{equation}

 Define 
$$V_{p,q,r,s}=\{ (M,N)\in SU(2)\times SU(2)~|~ M^{s+p}N^{q-r}=\bbi\text{~and~}
N^{-r}M^s=e^{\gamma\bbk}\bbi, \gamma\in [0,\pi]\}.$$ 
An analysis using Equation (\ref{teq0}) leads to the following result.

\begin{thm}\label{torusknotreps} 
The assignment 
 $A\mapsto M, B\mapsto N$
 induces a homeomorphism $$V_{p,q,r,s}\cong R(S^3\setminus B^3, T_{p,q}\setminus (A_1\cup A_2))=R(Y_0,K_0).$$
 
For $(M,N)\in V_{p,q,r,s}$, define $\gamma\in [0,\pi]$ and $\theta\in [0,2\pi)$ by 
$$N^{-r}M^s=e^{\gamma\bbk}\bbi\text{~and~}N^{-r}M^{s+p} N^q=e^{\theta\bbk}\bbi.$$
Then $ e^{( \gamma-\theta)\bbk}=M^pN^q$ 
and the space $R(S^3,T_{p,q})$ is homeomorphic to the subset   of $V_{p,q,r,s}$ consisting of those pairs $(M,N)$ satisfying $M^pN^q=1$. 

 The restriction to the pillowcase $R(S^3,T_{p,q})\to R(S^2,\{a,b,c,d\})$  takes every non-abelian representation to a non-abelian representation.  
 
Finally, for suitable small perturbations, $R^\nat_\pi(S^3, T_{p,q})$ contains  $|\sigma(T_{p,q})|+1$ non-degenerate points. 
\end{thm}
\begin{proof}
Let ${\bf I}$ denote the interval  of pure unit quaternions of the form $e^{\gamma\bbk}\bbi$ with $\gamma\in [0,\pi]$.  Define $F:SU(2)\times SU(2)\to SU(2)\times SU(2)$ by the formula  $F(M,N)= (M^{s+p}N^{q-r}, N^{-r}M^s)$ so that
$V_{p,q,r,s}=F^{-1}(\{\bbi\}\times  {\bf I}).$

The pair $(M,N)\in V_{p,q,r,s}$ determines a representation which sends $a$ to $\bbi$ and $b$ to a unit quaternion, say $e^{\gamma\bbk}\bbi$ in ${\bf I}$.  Then $c$ is sent to $N^{-r}aN^r$ and $d$ is sent to $N^{q-r}bN^{q-r}$, so that $c$ and $d$ are sent to the conjugacy class of pure unit quaternions. This shows $V_{p,q,r,s}$ maps into $R(Y_0,K_0)$. 

The map is surjective since any representation in $R(Y_0,K_0)$ can be conjugated to send $a$ to $\bbi$ and $b$ into ${\bf I}$. 
If $A$ and $B$ are sent to non-commuting quaternions then the resulting conjugacy class is uniquely determined.

 If $A, B$ are sent to commuting elements $M,N$, then 
$(M^{s+p}N^{q-r}, N^{-r}M^s)=(\bbi, \pm \bbi)$, and hence $M$ and $N$ lie in the circle $\{e^{\mu \bbi}\}$, say $M=e^{\mu_1 \bbi}, N=e^{\mu_1 \bbi}$.  Therefore 
$$(\bbi, \pm \bbi)=(e^{((s+p)\mu_1+(q+r)\mu_2)\bbi}, e^{(s\mu_1-r\mu_2)\bbi})$$
so that $((s+p)\mu_1+(q+r)\mu_2)=\frac{\pi}{2}$ and $s\mu_1-r\mu_2=\pm \frac{\pi}{2}$ modulo $2\pi$. Since $pr+qs=1$, these equations uniquely determine $\mu_1$ and $\mu_2$ modulo $2\pi$, and hence $M$ and $N$ are again uniquely determined, and so the map is a homeomorphism.
 
One computes $$e^{(\gamma-\theta)\bbk}=-e^{\gamma\bbk}\bbi e^{\theta\bbk}\bbi=(N^{-r}M^s)^{-1}N^{-r}M^{s+p}N^q=M^pN^q.$$
 The space $R(S^3,T_{p,q})$ corresponds to those representations in $R(Y_0,K_0)$ satisfying $a=d $ and $ b=c$.  From Equation (\ref{words}) we see that if $A,B$ are sent to $M,N$,  then $a=d $ and $ b=c$  precisely when $M^pN^q=1$.

 Suppose a representation in $R(S^3,T_{p,q})$ restricts to an abelian representation in the pillowcase $R(S^2,\{a,b,c,d\})$. The representation is given by a pair $(M,N)\in V_{p,q,r,s}$ satisfying $M^pN^q=1$, and, since its restriction is abelian, Equation (\ref{words}) gives
  \begin{equation}
\label{relation1}
\bbi=  M^{s+p}N^{q-r}, ~\sigma\bbi=N^{-r}M^s, ~ \sigma \bbi =N^{-r}\bbi N^r,~
\bbi =N^{-(q-r)}(\sigma \bbi) N^{q-r}
\end{equation}
for some choice of sign $\sigma=\pm1$.

 Hence 
$$N^{-1}\bbi N=N^{-rp-qs}\bbi N^{rp+qs}
=(N^{-r})^{p+s} (N^{-(q-r)})^s \bbi (N^{(q-r)})^s(N^{r})^{p+s}=\sigma^{2s+p}\bbi$$

When  $\sigma^p=1$, this implies that $N=e^{\mu\bbi}$ for some $\mu$. Then 
$M=M^{pr+qs}=N^{-qr}(N^r\sigma\bbi)^q$ and so $M$ and $N$ commute, so that the representation is abelian.

We will show that    $\sigma^p=-1$ is impossible. Suppose to the contrary that $\sigma^p=-1$, so that $\sigma=-1$ and $p$ is odd. The equation $N^{-1}\bbi N=-\bbi$ implies that $N$ is a pure unit quaternion, and hence $N^2=-1$.  Then $-\bbi=N^{-r}\bbi N^r=N^{-(q-r)}\bbi N^{q-r}$ imply that $r$ is odd  and  $q$ is even. Since   $M^p=N^{-q}$, it follows that $M^{2p}=1$.

From Equation (\ref{relation1}) we obtain
$$1=(-\bbi)\bbi=N^{-r}M^sM^{s+p}N^{q-r}=
N^{-r}M^{2s}M^{ p}N^{q}N^{-r}=N^{-r}M^{2s}N^{-r}=-N^{-r}M^{2s}N^{r}$$
so that $M^{2s}=-1$.   But then
 $M^2=M^{2(pr+qs)}=M^{2qs}=(-1)^q=1,
$ 
a contradiction. Hence $\sigma^p\ne-1$.

It follows from \cite{Herald2, Heusener} that $R(S^3, T_{p,q})$ contains $|\sigma(T_{p,q})|/2$ 
non-abelian representations and one abelian representation. Since the restriction of any non-abelians to the pillowcase avoids the corners, and hence correspond to intersection points of the image $V_{p,q,r,s}\to R(S^2,\{a,b,c,d\})$ with the {\em interior} of the arc $\gamma=\theta$.  Applying Propositions  
 \ref{circle or point}, \ref{natcase}, and Theorem \ref{pertvar} we see that each non-abelian representation in $R(S^3, T_{p,q})$ gives rise to two non-degenerate points of $R^\nat(S^3,T_{p,q})$ and the   representation $\alpha$ gives one more non-degenerate point.
  \end{proof}

Theorem \ref{torusknotreps}   does not give an explicit description of  $R(Y_0,K_0) \cong V_{p,q,r,s}$ or its image in the pillowcase in the sense which was carried out for 2-bridge knots above. The space $V_{p,q,r,s}$ is complicated; in fact the map $F:SU(2)\times SU(2)\to SU(2)\times SU(2)$ is not transverse to the interval $\{\bbi\}\times {\bf I}$; we will see an example below where $V_{p,q,r,s}$ is a singular variety.  

We give a more explicit description of $V_{p,q,r,s}$ in two different ways in Proposition \ref{crosssect} and Theorem \ref{toruspt2}  below.  In Proposition \ref{crosssect} we prove that for any $p,q,r,$ and $s$, $V_{p,q,r,s}$ contains an arc which maps to a straight line segment in the pillowcase. In Theorem \ref{toruspt2}  we show that $V_{p,q,r,s}$ is a semi-algebraic set in $\RR^2$ contained in the zero set of a polynomial determined explicitly by the integers $p,q,r,s$. The method of Theorem \ref{toruspt2}  lends itself easily to computer calculation and makes it easy describe $V_{p,q,r,s}$ in particular examples.

\bigskip

Write ${\bf I}$ for $\{\bbi\}\times {\bf I}$, so that $V_{p,q,r,s}=F^{-1}(\bf I)$.   The following  result shows how to construct  a splitting of  $V_{p,q,r,s}\to {\bf I}$, yielding a curve in $V_{p,q,r,s}$ whose image in the pillowcase is a straight line. 
 
\begin{prop}\label{crosssect}  There exists a cross section $s:  {\bf I}\to V_{p,q,r,s}$ of the map $F:V_{p,q,r,s}\to{\bf I}$ whose image consists of binary dihedral representations.  The image of  $s(\bf I)$ in the pillowcase $R(S^2,\{a,b,c,d\})$ is a straight line segment. 

If $p$ and $q$ are odd, or if $p$ is even and $q-2r=\pm 1$, then 
 the initial point $s(\bbi)= s(e^{0\bbk}\bbi)$ is the restriction to $(S^3\setminus B^3, T_{p,q}\setminus (A_1\cup A_2))$ of the unique abelian traceless representation of 
 $\pi_1(S^3\setminus T_{p,q})$.  
\end{prop}
\begin{proof} Assume that $q$ is odd by interchanging $p$ and $q$ if necessary. Fix $\gamma\in [0,\pi]$. We construct $s(e^{\gamma\bbk}\bbi)$.
We consider three cases: 
\begin{enumerate}
\item $p, r$ both odd,
\item $p $ odd, $ r$ even,
\item $p$ even 
\end{enumerate}
 For the first case, $p,q, r$, and $p+s$ are odd, and $s$ and $q-r$ are even. 
 Set
 $$s(e^{\gamma\bbk}\bbi)=((-1)^{(p+s+q-r-1)/2}\bbi, (-1)^{(s-r-1)/2}e^{\gamma\bbk}\bbi).\ 
 $$
 For the second case, $p,q,s$, and $q-r$ are odd, and $r, p+s$ are even.
  Set
 $$s(e^{\gamma\bbk}\bbi)=((-1)^{(r-s+1)/2}e^{\gamma\bbk}\bbi,(-1)^{(p+s+q-r-1)/2}\bbi).
$$
For the third case, $q,s,$ and $p+s$ are odd.
Set $$s(e^{\gamma\bbk}\bbi)=(e^{\tau\bbk}\bbi, e^{\psi\bbk}),$$ where 
$$\tau= \frac{q-r}{q-2r}  \gamma  +  \frac{\pi(rp-qs+2rs+q-2r)}{2(q-2r)} ,
\psi=\frac{1}{q-2r}\gamma + \frac{\pi  p }{2(q-2r)}.
 $$
A calculation shows that in each of the three cases, $F(s(e^{\gamma\bbk}\bbi))=(\bbi, e^{\gamma\bbk}\bbi)$.  Note that in each case the generators are sent to    the binary dihedral subgroup 
$\{e^{\theta\bbk}\}\cup \{e^{\theta\bbk}\bbi\}$ of $SU(2)$.

Denote $ s(e^{\gamma\bbk}\bbi)$ by $(M(\gamma),N(\gamma))$.  Then $M(0)$ and $N(0)$ commute so that the corresponding representation is abelian.  Moreover, when $p$ and $q$ are both odd, or if $p$ is even and $q-2r=\pm 1$,   $M(0)^pN(0)^q=1$, so that this abelian representation extends to a traceless abelian representation of $\pi_1(S^3\setminus T_{p,q})$.

Theorem \ref{torusknotreps}  implies that $e^{(\gamma-\theta)\bbk}=M^p(\gamma)N^q(\gamma)$.   For the three cases:
$$
e^{(\gamma-\theta)\bbk}=M^p(\gamma)N^q(\gamma)
=\begin{cases}e^{-\gamma\bbk}&\text{case (1)}, \\
e^{\gamma\bbk}&\text{case (2)},\\
e^{\frac{q}{q-2r}\gamma\bbk}e^{\frac{\pi p (2q-2r)}{2(q-2r)}\bbk}&\text{case (3)},
\end{cases}
$$
so that in each case $\theta$ is a linear function of $\gamma\in [0,\pi]$.
\end{proof}

\medskip

We turn now to a different description of $R(Y_0,K_0)\cong V_{p,q,r,s}$ in terms of Chebyshev polynomials. Theorem \ref{toruspt2} below roughly says that if $(M,N)=(e^{uQ}, e^{vR})\in V_{p,q,r,s},$ then there exists a polynomial $p(x,y)$ so that $p(\cos u,\cos v)=0$, and that conversely, the zero set of this polynomial, subject to some inequalities (essentially $|x|,|y|\leq 1$), parameterizes $V_{p,q,r,s}$.

\medskip
First, we return to Equation (\ref{teq}).    Using Lemma \ref{basic} we may rewrite this  as
\begin{equation}
\label{toruseqns}
\begin{split}
0&=\cos((s+p)u)\cos((q-r)v)-\sin((s+p)u)\sin((q-r)v) Q\cdot R,  \\
0&=\cos(su)\cos(-rv)-\sin(su)\sin(-rv) Q\cdot R
\end{split}
\end{equation}
where $Q\cdot R$ denotes the dot product of $Q$ and $R$.

Therefore, if we define $\tilde V_{p,q,r,s}$ to be the set of pairs $(M,N)=(e^{uQ},e^{vR})\in SU(2)\times SU(2)$ satisfying Equations (\ref{toruseqns}), then $\tilde V_{p,q,r,s} $ is in bijective correspondence with all representations (not conjugacy classes) of traceless representations.  Note that $ V_{p,q,r,s}\subset \tilde V_{p,q,r,s}$.

For each integer $n$, there exist (Chebyshev) polynomials $T_n(x)$ and $S_n(x)$ so that 
 \begin{equation}
\label{chubbychef}
\cos(nu)=T_n(\cos u) \text{~ and ~}\sin(nu)=\sin uS_n(\cos u).
\end{equation}
Hence Equation (\ref{toruseqns}) can be rewritten as
\begin{equation}
\label{toruseqns2}
\begin{split}
0&=T_{s+p}(\cos u)T_{q-r}(\cos v)-\sin(u)\sin(v) S_{s+p}(\cos u)S_{q-r}(\cos v)Q\cdot R,  \\
0&=T_{s}(\cos u)T_{-r}(\cos v)-\sin(u)\sin(v) S_{s}(\cos u)S_{-r}(\cos v)Q\cdot R.
\end{split}
\end{equation}
Substituting $x=\cos u, y=\cos v$, multiplying the first equation by $S_s(x)S_{-r}(y)$, the second by $S_{s+p}(x) S_{q-r}(y)$ and subtracting yields the polynomial equation 
\begin{equation}
\label{toruseqns4}
p_{p,q,r,s}(x,y):=T_{s+p}(x)T_{q-r}(y) S_{s}(x)S_{-r}(y)- S_{s+p}(x)S_{q-r}(y)T_{s}(x)T_{-r}(y)=0.\end{equation}
Thus the map
$$ \tilde V_{p,q,r,s}\to\RR^2,~(e^{uQ}, e^{vR})\mapsto (\cos u , \cos v)$$
takes its image in the zero set of the polynomial $p_{p,q,r,s}(x,y)$ of Equation (\ref{toruseqns4}).

  Denote by $Z$ the zero set of  $p_{p,q,r,s}$.  For $(x,y)\in Z$,    the two ratios 
  \begin{equation}
  \label{toruseqns3}
\frac{T_{s+p}(x)T_{q-r}(y)}{ \sqrt{(1-x^2)(1-y^2)}S_{s+p}(x)S_{q-r}(y)}  
, ~\frac{T_{s}(x)T_{-r}(y)}{ \sqrt{(1-x^2)(1-y^2)}S_{s}(x)S_{-r}(y)}
\end{equation}
 are equal if neither denominator vanishes.
 
 Denote by  $Z_0\subset Z$  the subset containing those points so that 
 at least one of the denominators in the ratios of (\ref{toruseqns3}) is non-zero, and if the other is zero, so is its numerator. 
Define the function  $\tau(x,y)$ on $Z_0$ to be one of these two ratios, so that 
\begin{equation*}
\tau(x,y)=\frac{T_{s+p}(x)T_{q-r}(y)}{ \sqrt{(1-x^2)(1-y^2)}S_{s+p}(x)S_{q-r}(y)}  
=\frac{T_{s}(x)T_{-r}(y)}{ \sqrt{(1-x^2)(1-y^2)}S_{s}(x)S_{-r}(y)}
\end{equation*}
 for $(x,y)\in Z_0$ (and at least one of these ratios is defined).

 Denote by $Z_1\subset Z$ the subset containing those points so that 
 both denominators  and both numerators in the ratios of (\ref{toruseqns3}) are zero.
 
Call two pairs of quaternions $(M,N)$ and $(M',N')$  conjugate if there exists $g\in SU(2)$ so that $(M',N')=(gMg^{-1}, gNg^{-1})$.

\begin{thm}\label{toruspt2}  Let $p_{p,q,r,s}(x,y)$  be the polynomial defined above in Equation (\ref{toruseqns4}), and let $Z$ denote its zero set in $\RR^2$, and $Z_0,Z_1\subset Z$ the subsets defined above.

   If $(M,N)=(e^{uQ},e^{vR})\in  V_{p,q,r,s}$ then $p_{p,q,r,s}(\cos u, \cos v)=0$, and  the fiber of the map $$V_{p,q,r,s}\to Z,~  (e^{uQ},e^{vR})\mapsto(\cos u,\cos v)$$ over a point $(x,y)\in Z$ is given as follows.
\begin{enumerate}
\item If $|x|< 1, |y|< 1$, $(x,y)\in Z_0$, and  $|\tau(x,y)|\leq 1$,    then the fiber over $(x,y)$ is a single point,   conjugate to $(e^{uQ}, e^{vR})$, where $u=\arccos x, v=\arccos y, $  and
$$ Q=\bbi, R=e^{t\bbk}\bbi,t=\arccos(\tau(x,y)).$$ 
The corresponding point $(\gamma,\theta)$ in the pillowcase $R(S^2,\{a,b,c,d\})$ satisfies
\begin{equation*}
\begin{split}
\cos\gamma&=-T_{2s+p}(x)T_{q-2r}(y)+\sqrt{(1-x^2)(1-y^2)} S_{2s+p}(x)S_{q-2r}(y)\tau(x,y),  \\
\cos(\gamma-\theta)&=T_{p}(x)T_{q}(y)-\sqrt{(1-x^2)(1-y^2)} S_{p}(x)S_{q}(y) \tau(x,y).\end{split}
\end{equation*}

\item If $|x|< 1, |y|< 1$,  and $(x,y)\in Z_1$, then   the fiber over $(x,y)$ is an arc conjugate to the arc $
t\mapsto  (e^{u\bbi}, e^{v e^{t\bbk}\bbi}), t\in [0,\pi]$, where $u=\arccos x, v=\arccos y$.

The image $(\gamma(t), \theta(t))$ of this path in the pillowcase $R(S^2,\{a,b,c,d\})$ satisfies
\begin{equation*}
\begin{split}
\cos\gamma&=-T_{2s+p}(x)T_{q-2r}(y)+\sqrt{(1-x^2)(1-y^2)} S_{2s+p}(x)S_{q-2r}(y)\cos t  \\
\cos(\gamma-\theta)&=T_{p}(x)T_{q}(y)-\sqrt{(1-x^2)(1-y^2)} S_{p}(x)S_{q}(y) \cos t.\end{split}
\end{equation*}

\item If $|x|\leq 1, |y|\leq 1$, one of $|x|,|y|$ equals 1,  and $(x,y)\in Z_1$, then   the fiber over $(x,y)$ is a single point. 

The corresponding point $(\gamma,\theta)$ in the pillowcase $R(S^2,\{a,b,c,d\})$ satisfies
\begin{equation*}
\begin{split}
&\cos\gamma=-T_{2s+p}(x)T_{q-2r}(y)  \\
&\cos(\gamma-\theta)=T_{p}(x)T_{q}(y).\end{split}
\end{equation*}

\item In all other cases, the fiber is empty, i.e.~ $(x,y)$ is not in  the image.
\end{enumerate}

\end{thm}

\begin{proof} If $(M,N)=(e^{uQ}, e^{vR})\in V_{p,q,r,s}$, then $(u,v,Q,R)$ satisfy the Equations (\ref{toruseqns2}), and hence satisfy $p_{p,q,r,s}(x,y)=0$ where $x=\cos v, y=\cos u$. In particular, $|x|\leq 1$, $|y|\leq 1$. 

If neither of the denominators in Equation (\ref{toruseqns3}) vanishes, then $ (x,y)\in Z_0$, $|x|<1, |y<1$,   and $\tau(x,y)=Q\cdot R$ and hence $|\tau(x,y)|\leq 1$. If exactly one of the denominators in Equation (\ref{toruseqns3}) vanishes, then 
Equations (\ref{toruseqns2}) show its numerator also vanishes, thus $(x,y)\in Z_0$  and   $|\tau(x,y)|=|Q\cdot R|\leq 1$.
 If $(x,y)\not\in Z_0$, then Equations (\ref{toruseqns2}) show that both numerators are zero, so that $(x,y)\in Z_1$.  Moreover, $Q\cdot R$ can be any number in $[-1,1]$ i.e.~ $Q\cdot R=\cos t $ for some $t\in [0,\pi]$.  By conjugating the pair $(Q,R)$ we may assume $Q=\bbi$ and $R=e^{t\bbk}\bbi$.
Thus    the image 
of $V_{p,q,r,s}\to Z$ lies in   $Z_0\cup Z_1$ and we have established (3).

Conversely, suppose that $(x,y)\in Z_0$, $|x|<1,|y|<1$, and $|\tau(x,y)|\leq1$. Let 
$$u=\arccos x, v=\arccos y, t=\arccos(\tau(x,y)), Q=\bbi, R=e^{t\bbk}\bbi.$$
Then $(u,v,Q,R)$ satisfy (\ref{toruseqns2}). Hence the pair $(M',N')=(e^{uQ}, e^{vR})$ satisfies Equation (\ref{teq}), thus defining a traceless representation, which is uniquely conjugate to a pair $(M,N)\in V_{p,q,r,s}$ by Theorem \ref{torusknotreps}.  Conjugation does not change $u,v$ nor $x,y$. 
Proposition  \ref{s2prop} implies that  there exists a unit quaternion  which conjugates the triple 
$$(e^{(p+s)uQ}e^{(q-r)vR},e^{-rvR}e^{suQ},e^{puQ}e^{qvR})$$
to 
$$(\bbi, e^{\gamma \bbk}\bbi, e^{(\gamma-\theta)\bbk}).$$
Therefore, 
$$-\cos\gamma= \Real(e^{-rvR}e^{suQ}e^{(p+s)uQ}e^{(q-r)vR})\text{~and~}
\cos(\gamma-\theta)=\Real(e^{puQ}e^{qvR}).
$$
Using conjugation invariance and part (4) of Proposition \ref{basic}, these equations can be rewritten as
\begin{equation}
\label{pilloweqns}
\begin{split}
-\cos\gamma&=\cos((2s+p)u)\cos((q-2r)v)-\sin((2s+p)u)\sin((q-2r)v) Q\cdot R,  \\
\cos(\gamma-\theta)&=\cos(pu)\cos(qv)-\sin(pu)\sin(qv) Q\cdot R
\end{split}
\end{equation}
Substituting $x$ and $y$  transforms these to
\begin{equation*}
\begin{split}
-\cos\gamma&=T_{2s+p}(x)T_{q-2r}(y)-\sqrt{(1-x^2)(1-y^2)} S_{2s+p}(x)S_{q-2r}(y)\tau(x,y),  \\
\cos(\gamma-\theta)&=T_{p}(x)T_{q}(y)-\sqrt{(1-x^2)(1-y^2)} S_{p}(x)S_{q}(y) \tau(x,y),\end{split}
\end{equation*}
establishing (1).

Suppose that $(x,y)\in Z_1$ and $|x|\leq 1,|y|\leq 1$.  Then, for any $t\in [0,\pi]$, define
$$u=\arccos x, v=\arccos y,  Q=\bbi, R=e^{t\bbk}\bbi.$$
Then $(u,v,Q,R)$ satisfy (\ref{toruseqns2}). Hence the pair $(M',N')=(e^{uQ}, e^{vR})$ satisfies Equation (\ref{teq}), defining a traceless representation  which is  conjugate to a pair $(M,N)\in V_{p,q,r,s}$ by Theorem \ref{torusknotreps}.  Conjugation does not change $u,v, Q\cdot R$, nor $x,y$. 

If $|x|=1$, then $u=0$ or $\pi$, so that $(M',N')=(\pm 1, e^{vR})$, which is conjugate to $(\pm 1, e^{v\bbi})\in V_{p,q,r,s}$, and in particular independent of $t$. Thus the fiber of $V_{p,q,r,s}\to Z_1\cap \{ |x|=1\}$ is a single point.  If $|y|=1$, then $(M',N')=(e^{u\bbi}, \pm 1)=(M,N)$ and so again the fiber of $V_{p,q,r,s}\to Z_1\cap \{ |y|=1\}$ is a single point. 

When $|x|<1$ and $|y|<1$, $(M,N)$ is uniquely determined by $(M',N')$.
The assertions about $\cos \gamma $ and $\cos(\gamma-\theta)$ follow from 
Equation (\ref{pilloweqns}).
Thus we have established (2) and (3).
\end{proof}

The statement of   Theorem \ref{toruspt2} is unfortunately somewhat technical, and does not easily reveal the structure of $R(Y_0,K_0)$ and its image in the pillowcase.  However,   the polynomials $p_{p,q,r,s}(x,y)$ can be computed and their zero sets graphed using  computer algebra software. 
We present a few examples.

\medskip
  
The  $(2,2n+1)$ torus knots are particularly simple to understand in this context. Take $p=2, q=2n+1, r=n+1, s=-1$. Then $p_{2, 2n+1, n+1, -1}(x,y)=x$, and $\tau(x,y)=0$ along the arc  $x=0, y\in [-1,1]$. Theorem \ref{toruspt2} then says that  $V_{2, 2n+1, n+1, -1}$ is an arc, and gives a  parameterization 
by $(e^{uQ}, e^{vR})$ where $$u=\arccos(0)=t=\tfrac{\pi}{2}, v=\arccos(y)\in [0,\pi], Q=\bbi, R=e^{t\bbk}\bbi=\bbj.$$ Conjugating the pair $(e^{uQ}, e^{vR})=(\bbi, e^{v\bbj})$ by $e^{\frac{\pi}{4}\bbi}e^{\frac{\pi n v}{2}\bbj}$ yields the pair 
$(M,N)=(e^{nv\bbk}\bbi, e^{v\bbk})$ which satisfies $F(M,N)=(\bbi, e^{(\pi-v)\bbk}\bbi)$, and hence parametrizes $V_{2,2n+1,n+1, -1}$.  Computing $M^2N^{2n+1}$ yields $e^{(\pi + (2n+1)v)\bbk}$ and from Theorem \ref{torusknotreps} we see that $\gamma=\pi-v$ and $\gamma-\theta=\pi+ (2n+1)v$, so that 
$\theta = (2n+2)\gamma
$ mod $2\pi$. This is the same arc identified in the examination of 2-bridge knots in Section \ref{twobridgesection}, and is also the same arc produced by the cross section of Proposition \ref{crosssect} (note that $q-2r=-1$).
  Theorem \ref{toruspt2} gives $\cos\gamma=-\cos v$ and $\cos(\gamma-\theta)=-\cos((2n+1)v)$. The first equation implies $\gamma=\pi-v$ and the second that $\gamma-\theta=\pi \pm (2n+1) v$, which implies $(1\pm(2n+1))\gamma=\theta$. So these equations are not quite sharp enough to give $\theta=(2n+2)\gamma$.

\bigskip

We turn to the $(3,4)$ torus knot. This knot is interesting because $R(Y_0,K_0)$ is  singular and, as we shall see below, the 
instanton complex $CI^\nat(S^3,T_{3,4})$  has a non-trivial differential.  

Take $p=3,q=4, r=3, s=-2$. The polynomial $p_{3,4,3,-2}(x,y)$ of Equation (\ref{toruseqns4}) is computed, using Equation (\ref{chubbychef}), as
$$p_{3,4,3,-2}(x,y)= y \left( 4\,{x}^{2}+4\,{y}^{2}-3 \right).$$
We show that $V_{3,4,3,-2}$
is the union of an arc and a circle that meet in two points.  

The 
zero set of $p_{3,4,3,-2}$ meets the set $Z_0\cup Z_1$ of Theorem \ref{toruspt2} in union of the arc $(x,0), x\in [-1,1]$ and the circle $x^2+y^2=\frac{3}{4}$.
The endpoints $(\pm 1,0)$ of the arc lie in $Z_1$ and fall under case (3) of Theorem \ref{toruspt2}, and the rest of the points lie in $Z_0$. Moreover, $\tau(x,y)=\frac{xy}{\sqrt{(1-x^2)(1-y^2)}}$, which is less than 1 on $Z_0$, so that the map $$V_{3,4,3,-2}\to Z\cap \{ |x|\leq 1, |y|\leq 1\}$$ is a homeomorphism.

Applying Theorem \ref{toruspt2} we see that the arc $(x,0), x\in[-1,1]$ has $y=0$ and $\tau(x,y)=0$, so that $v=t=\frac{\pi}{2}$ and hence  gives
 the arc $(e^{uQ}, e^{vR})=(e^{u\bbi}, \bbj), ~u\in [0,\pi]$ in $\tilde V_{3,4,3,-2}$. Conjugating by   $e^{-\frac{\pi}{4}\bbi}e^{-\frac{\pi}{4}\bbk}e^{-\frac{u}{2}\bbi} $   yields the arc 
 $(e^{u\bbk}, e^{-u\bbk}\bbi)$. This arc lies in $V_{3,4,3,-2}$,
since it is  sent by $F$ to the arc $(\bbi, e^{u\bbk}\bbi)$.  This   is the arc identified in Proposition \ref{crosssect} (after a change of notation, since we are taking $p$ odd and $q$ even here; the condition $q-2r=\pm 1$ transforms to $p+2s=\pm 1$).

Since $y=\cos v= 0$ along this arc and $\tau(x,0)=0$, Theorem \ref{toruspt2}
gives
$$\begin{pmatrix} \cos\gamma\\ \cos(\gamma-\theta)\end{pmatrix}
=\begin{pmatrix} -T_1(x)T_{-2}(y)\\ T_3(x)T_4(y)\end{pmatrix}
=\begin{pmatrix} \cos u\\ \cos (3u)\end{pmatrix}
$$
and so $\gamma=u$ and $\gamma-\theta=\pm 3u$.
Thus $\theta=-2\gamma$ or $\theta= 4\gamma$.  At the point $x=0$,
$u=\frac{\pi}{2}$ and so $(e^{u\bbk}, e^{-u\bbk}\bbi)=(\bbk,-\bbj)$. Hence
$N^{-r}M^s=\bbj=e^{\frac{\pi}{2}\bbk}\bbj$ so that $\gamma=\frac{\pi}{2}$, and
$M^pN^q=-\bbk=e^{-\frac{\pi}{2}}\bbk$ so that $\gamma-\theta=-\frac{\pi}{2}$.

Hence at this point (and by continuity along the entire arc) $\theta=-2\gamma$  rather than $4\gamma$. We denote this arc in $V_{3,4,3,-2}$ by ${\bf I}_0$, and consider it parameterized by $\gamma=u\in [0,\pi]$.

The circle $4x^2+4y^2-3=0$ lies entirely in $Z_0$, since $\tau(x,y)<1$ on this circle.   It intersects the arc ${\bf I}_0$ in the points $(x,y)=(\pm \frac{\sqrt{3}}{2},0)$
and since $x=\cos(u)$ and $u=\gamma$, the intersection points occur when $\gamma=\frac{\pi}{6}$ and $\frac{5\pi}{6}$.  In particular $V_{3,4,3,-2}$ is singular, and so the map $F:SU(2)\times SU(2)\to SU(2)\times SU(2)$ of Theorem \ref{torusknotreps}    is not transverse to $\bbi \times \bf I$, in contrast to the case for $(2,n)$ torus knots.
 
We use Theorem \ref{toruspt2}
 to find the image of the circle in the pillowcase.  When $4x^2+4y^2=3$, then 
 one  calculates
$\cos\gamma= x$ and $ \cos(\gamma-\theta)=-4x^3+3x.$
Since $\cos(3\gamma)=4\cos^2\gamma+3\cos\gamma$, this  yields
$$\cos(\gamma-\theta)=-\cos3\gamma,$$
so that $\theta$ equals $\pi +4\gamma$ or $\pi -2\gamma$.  
 Since $\cos\gamma=x=\cos u$ and $\gamma,u\in[\frac{\pi}{6},\frac{5\pi}{6}]$, it follows that $u=\gamma$. From this one computes that at the points $(x,y)=(0, \pm \frac{\sqrt{3}}{2})$, $\gamma=\frac{\pi}{2}$ and $\theta=\pi$, so that 
 $$\theta=\pi + 4\gamma.$$
 
The map of the circle to the pillowcase takes the pairs $(x,y)$ and $(x, -y)$ in this circle  to the same point.  Thus we have proved the following.

\begin{prop} For the $(3,4)$ torus knot, the space $V_{3,4,3,-2}\cong R(Y_0,K_0)$ is
homeomorphic to the  union of three arcs, 
$${\bf I}_0:[0,\pi]\to V_{3,4,3,-2}, ~{\bf I}_\pm:[\tfrac{\pi}{6},\tfrac{5\pi}{6}]\to V_{3,4,3,-2},$$
where $$ {\bf I}_\pm(\tfrac{\pi}{6})={\bf I}_0(\tfrac{\pi}{6}), ~
{\bf I}_\pm(\tfrac{5\pi}{6})={\bf I}
 _0(\tfrac{5\pi}{6}).$$
 The arc ${\bf I}_0$ maps to $\theta=-2\gamma$ and each arc ${\bf I}_\pm$ maps to the arc $\theta=4\gamma+\pi$ in the pillowcase.\qed
 \end{prop}

Figure \ref{fig34} illustrates the   space $V_{3,4,3,-1} $   and its image in the pillowcase $R(S^2,\{a,b,c,d\})$.
  \begin{figure}
\begin{center}
\def\svgwidth{5in}

 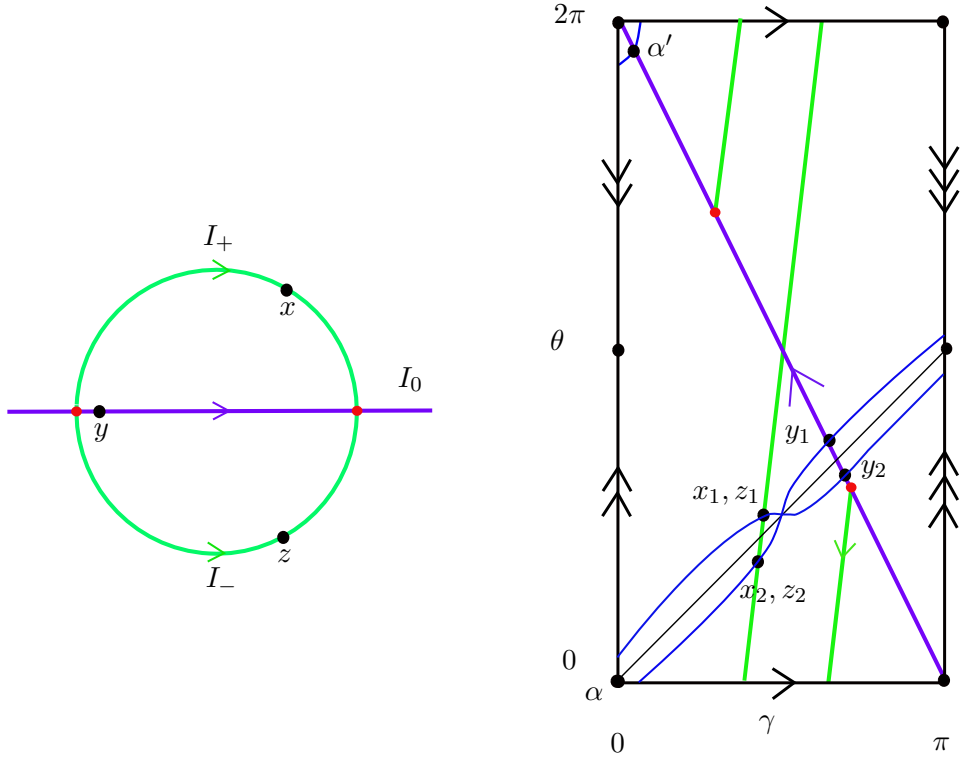

 \caption{The (3,4) torus knot \label{fig34}}
\end{center}
\end{figure} 
Notice that the images of    $\bf I_+$, and $\bf I_-$   in the pillowcase each intersect  the arc $\theta=\gamma$ (transversely) in one interior point,  and $\bf I_0$ intersects $\theta=\gamma$ in two points,  the distinguished representation $\alpha$ at one corner and one interior point. This gives the four traceless representations of the $(3,4)$ torus knot.  Perturbing as above we conclude that the instanton chain complex for $T_{3,4}$ is generated by seven generators, $\alpha', x_1,x_2, y_1,y_2,z_1,z_2$ of the reduced instanton complex.  Note that $x_i$ and $z_i$ are mapped to the same point in the pillowcase.   The representation labelled $y$ lies on ${\bf I}_0$, hence by Proposition \ref{crosssect} corresponds to a binary dihedral representation.

The perturbation illustrated by the blue circle in Figure \ref{fig34} corresponds to a small choice of perturbation parameter $\epsilon$.  Increasing $\epsilon$ eventually moves the intersection point labelled $y_2$ past the singular point where $\bf I_+$ and $\bf I_-$  bifurcate from $\bf I_0$.  This also creates two new intersection points, one each on $\bf I_+$ and $\bf I_-$.  This presumably corresponds to an elementary expansion of the reduced instanton complex.

The  reduced Khovanov homology of the $(3,4)$ torus knot has rank $5$, and since it forms the $E_2$ page of a spectral sequence converging to $I^\nat(S^3,K)$  (see Section \ref{data}),  one differential in $CI^\nat(S^3, T_{3,4})$ is non-trivial.  Arguments involving the gradings as in \cite{KM-khovanov} suggest that the differential involves the generator coming from the abelian representation.

 \medskip

We briefly list some other calculations.  For the $(4,5)$ torus knot, taking $p=4, q=5, r=4, s=-3$
$$p_{4,5,4,-3}(x,y)=x \left( 16\,{y}^{4}+16\,{x}^{2}{y}^{2}-20\,{y}^{2}-4\,{x}^{2}+3
 \right)
$$ and
$$\tau(x,y)={\frac {xy}{\sqrt {1-{x}^{2}}\sqrt {1-{y}^{2}}}}.$$
The intersection of the zero set of $p_{4,5,4,-3}(x,y)$ with $Z_0\cup Z_1$ is the union of an arc (corresponding to $x=0$) and  a circle which meets the arc in two points. The zero set of $p_{4,5,4,-3}(x,y)$ contains two other components which lie outside the square $|x|\leq 1, |y|\leq 1$ and hence do not contribute to $V_{4,5,4,-3}$, and so $V_{4,5,4,-3}$ is the union of a circle and an arc, and is homeomorphic to $V_{3,4,3,-2}$.

For the $(3,5)$ torus knot, $p_{3,5,2,-1}(x,y)=-8\,{y}^{4}+6\,{y}^{2}-2\,{x}^{2}$ and $V_{3,5,2,-1}$ is a figure 8 curve, i.e. a  circle with one double point.

\medskip

It is not clear from our description whether the image of $R(Y_0,K_0)$ in the pillowcase is always contained in the union of straight lines for all $(p,q)$ torus knots, as is the case for $(p,q)$ equal to  $(2,n)$ and $(3,4)$.  Further calculations reveal that the polynomials $p_{p,q,r,s}$ can be quite complicated and their zero sets highly singular, and it seems unlikely that the restriction to the pillowcase is linear.

  \section{Calculations of $CI^\nat(S^3,K)$}\label{data}

In this section, we use our results  to carry out some calculations of instanton homology. To set the stage, observe that for  a knot in a homology sphere with a simple representation variety, we have arranged that $R^\nat_\pi(Y,K)$ has $2k+1$ points, which we label as $\alpha'$ and $ \beta_{i,1},\beta_{i,2},~i=1,..., k$. Here $\alpha'$ corresponds to  the perturbation of  the distinguished isolated point $\alpha\in R^\nat(Y,K)$  that restricts to the abelian representation $\alpha_{\pi/2}$ on $Y\setminus N(K)$.  For each $i=1,...,k$, the two points $\beta_{i,j}\in R^\nat_\pi(Y,K)$ are those that result from perturbing the circle of non-abelian  traceless representations in $R^\nat(Y,K)$ coming from  $\beta_i\in R(Y,K)$. 

   \subsection{Summary of the results of Kronheimer-Mrowka}  The generators of $CI^\nat(Y,K)$ are the   points of $R^\nat_\pi(Y,K)$ for some generic perturbation data $\pi$.    The differential on $CI^\nat(Y,K)$ is defined  using moduli spaces of singular instantons on the cylinder $Y\times \RR$.  Though we suppress it from the notation, it is important to emphasize that the chain group $CI^\nat(Y,K)$ depends on the choice of perturbation data $\pi$, and that the differential depends on additional perturbation data on the cylinder. Generators of  $CI^\nat(Y,K)$ come equipped with a relative $\ZZ/4$-grading.   This grading is determined by  the spectral flow of the family of Hessians of  the Chern-Simons functional along a path of connections joining a  pair of generators.      A standard argument  shows that   the gradings of $ \beta_{i,1}$ and $\beta_{i,2}$ differ by 1 (after relabeling if necessary).    The relative grading is promoted to an absolute grading by defining  a grading difference associated to paths of connections on a cobordism of pairs from $(Y,K)$ to  $(S^3,U)$, where $U$ is the unknot, and normalizing the  grading of the unique point of $R^\nat(S^3,U)$  \cite[Proposition 4.4]{KM-khovanov}.   One can use a splitting theorem for spectral flow (e.g. \cite{APS, nicolaescu, Daniel-Kirk}) as in \cite{BHKK} to   compute the relative grading between a pair generators that lie on the same path component in $\chi(Y,K)$.  This is implicit in the discussion in \cite[Section 11]{KM-filtrations}.

   In \cite{KM-khovanov}, Kronheimer-Mrowka define a different $\ZZ/4$-graded chain complex for knots in $S^3$
whose homology  is $I^\nat(S^3,K)$.  To distinguish it from $CI^\nat(S^3,K)$ we denote it by $FCI^\nat(S^3,K)$.  This chain complex is filtered, and the associated spectral sequence has $E_2$ page  isomorphic to the reduced Khovanov homology of the mirror image $K^m$ of $K$, $E_2\cong Kh^{red}(K^m)$. The spectral sequence is $\ZZ/4$-graded, and the bigrading on Khovanov homology determines the mod 4 grading on the $E_2$ page. Explicitly, a generator of $Kh^{red}_{i,j}(K^m)$ with quantum grading $i$ and homological grading $j$ inherits the  grading $i-j+1\mod{4}$ in the $E_2$ page (Section 8.1 of \cite{KM-khovanov}).  Note that in that reference it is shown  that a generator of {\em unreduced} Khovanov homology in bigrading $(i,j)$  determines a generator of unreduced instanton homology in bigrading $i-j-1$; a shift of 2 occurs when passing to reduced homology.

 The chain complex $FCI^\nat(S^3,K)$ is built from the hypercube of complete unoriented resolutions of a diagram of $K$.  Its construction relies on the fact that the singular instanton homology groups of knots which differ by the unoriented skein relation fit into an exact triangle.  In general, this means that the rank of  $FCI^\nat(S^3,K)$ will be exponentially greater than the rank of any complex $CI^\nat(S^3,K)$ obtained by a non-degenerate perturbation of the Chern-Simons functional.  These considerations also show that for each $i\in \ZZ/4$,
\begin{equation}\label{ineq1}
\rank FCI_i^\nat(S^3,K)\ge \rank Kh^{red}_i(K^m)\ge \rank I^\nat(S^3,K)_i.
\end{equation}
Since the homology of $CI^\nat(S^3,K)$ equals $I^\nat(S^3,K)$, we have the obvious inequality   
 \begin{equation}\label{ineq2}
 \rank CI_i^\nat(S^3,K)\ge   \rank I^\nat(S^3,K)_i, \ \ \text{for each}\  i\in \ZZ/4.
\end{equation}

Kronheimer and Mrowka use excision for instanton homology to show that  $I^\nat(S^3,K)$ is isomorphic,  as a $\ZZ/4$ graded group, to the sutured instanton Floer homology $KHI(K)$ \cite[Proposition 1.4]{KM-khovanov}. This latter invariant was defined by Floer in \cite{Floer} and revisited in  \cite[Section 7]{KMsuture}. It has the advantage of possessing an additional $\ZZ$ grading, and  the graded Euler characteristic with respect to this grading equals the Alexander polynomial \cite{KMalex, Lim}.  It follows that if $\Delta_K(t)=\sum_i a_it^i$ denotes the Alexander polynomial, then 
\begin{equation}\label{ineq3}
\rank I^\nat(S^3,K)\ge \sum_i |a_i|\ge |\sum a_i(-1)^i|=|\det(K)|.
\end{equation}

\noindent On the other hand, the Euler characteristic of Khovanov homology equals the Jones polynomial \cite{Khovanov}
$$J_K(q)=\sum_{i,j} (-1)^j q^{-i} \rank Kh^\red_{i,j}(K).$$
Thus, if we let $J_K(q)=\sum_ib_iq^i$, then we have
\begin{equation}\label{ineq4}
 \rank Kh^{red}(K)\ge \sum_i |b_i|\ge |J_K(-1)|=|\det(K)|.
\end{equation}
These inequalities are obviously useful for understanding the behavior  of Kronheimer and Mrowka's spectral sequence.  For instance, they immediately show that spectral sequence collapses for all 2-bridge knots. This is because     $|\det(K)|$ equals   the rank of $Kh^\red(K)$ for   2-bridge knots.  More generally, the spectral sequence collapses for the same reason for all alternating and  quasi-alternating knots \cite{Khovanov,Lee,MO}. Hence Equations (\ref{ineq1}), (\ref{ineq3}), and (\ref{ineq4}) imply that for these knots, $I^\nat(S^3,K)=Kh^\red(K^m)$. 

\medskip

Similar facts hold for the unreduced theory $I^\sharp(S^3,K)$. In particular the corresponding spectral sequence has $E_2$ term the unreduced Khovanov homology of $K$ with its bigrading appropriately reduced to a mod 4 grading. Thus the calculations we give below can be modified to handle the case of unreduced Khovanov homology and $I^\sharp(S^3,K)$. The corresponding chain complex $CI^\sharp(S^3,K)$ is generated by $R^\sharp_\pi(S^3,K)$, which, using the perturbations described in Theorems \ref{pertvar} and \ref{pertvar2}, has twice as many points as $R_\pi^\nat(S^3,K)$.

\subsection{Remarks on calculations}

The calculations of reduced Khovanov homology we give below were obtained using Dror Bar-Natans' Knot theory Mathematica workbook \cite{BN-Math}.  
 We work over $\QQ$ for the remainder of the article. 
 
 \medskip
 
Use the notation $(a,b,c,d)$ for the $\ZZ/4$ graded vector space $(\QQ^a,\QQ^b, \QQ^c,\QQ^d)$, so, e.g. the rank in grading $2$ is equal to $c$ and in grading 3 is equal to $d$.  More generally, let 
$(a,b,c,d)_e$ denote the result of shifting $(a,b,c,d)$ to the right $e$ slots, so 
e.g. $(0,1,2,3)_3=(1,2,3,0)=(\QQ,\QQ^2,\QQ^3, 0)$.

 \subsubsection{2-bridge knots}    Theorem  \ref{2bridgethm} says that if $K(p/q)$ is a 2-bridge knot,  $CI^\nat(S^3,K(p/q)) $ is generated by $|p|=|\det(K)|$ points. Thus for 2-bridge knots, 
$$CI^\nat(S^3,K(p/q)) =I^\nat(S^3,K)=Kh^\red(K^m)$$
and all differentials in $CI^\nat(S^3,K(p/q)) $ are zero, as are all higher differentials in the spectral sequence from $Kh^\red(K^m)$ to $I^\nat(S^3,K)$.  While these facts could be deduced without the use of Theorem \ref{2bridgethm}, it would be interesting to use these examples to 
 investigate what grading and differential information can be gleaned from   the  intersection diagram (\ref{SVKD}). 

 For example, Figures \ref{fig12}, 
 \ref{fig14},  and \ref{fig72} illustrate the  Trefoil $=K(-3/1)$, with $Kh^\red(K^m)=I^\nat(K)=(1,0,1,1)$, the $(2,5)$ torus knot $=K(-5,1)$, with $Kh^\red(K^m)=I^\nat(K)=(2,1,1,1)$, the Figure 8 knot $=K(-5/3)$, with $Kh^\red(K^m)=I^\nat(K)=(1,1,2,1 )$,  and the knot $7_2=K(-11/5)$, with $Kh^\red(K^m)=I^\nat(K)=(3,2,3,3 )$.

\subsubsection{Knots with simple representation varieties}

We illustrate one sample calculation; all the data below was  obtained by the same method, which is the method described in \cite{KM-filtrations}. 

\subsubsection{The $(3,4)$ torus knot.} 
  The bigraded reduced Khovanov homology has Poincar{\'e} polynomial $Kh^{red}(T_{3,4}^m)=q^{-7}+ q^{-17}t^{-5} + q^{-13}t^{-4} + q^{-13}t^{-3} + q^{-11}t^{-2}$.  This means that there is a generator of homology in bidegree $(-7,0)$, one in $(-17,-5)$ etc.  The induced mod 4 graded group has a generator in degree $-7-0+1=-6\equiv 2\mod 4$, one in degree $-17-(-5)+1\equiv 1$, etc.  Thus the $\ZZ/4$ graded reduced Khovanov homology has ranks equalling $(2,1,1,1)$.

The character variety $\chi(S^3, T_{3,4})$ is illustrated in Figure \ref{fig21}, and is determined by the data (see Section \ref{toruschi}) 
$$(1,7), (5,11), (2,10).$$
There are four points in $R(S^3,T_{3,4})$, labelled $\alpha, x,y,z$ in the figure.
 \begin{figure}
\begin{center}

\def\svgwidth{4.7in}

 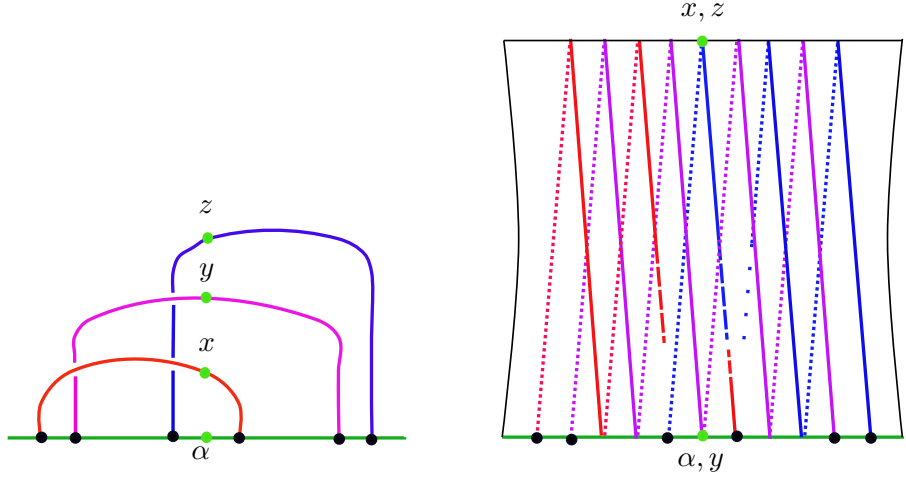

 \caption{$\chi(S^3, T_{3,4})$ and its image in the pillowcase $\chi(\partial N(T_{3,4}))$ \label{fig21}}
\end{center}
\end{figure} 
In terms of $R(S^3,T_{3,4})$ and its image in the  other pillowcase $
R(S^2,\{a,b,c,d\})$,   Proposition \ref{crosssect} shows that representations which lie on ${\bf I}_0$ send $N^q$ to $1$, since $q=4$.  Thus the two traceless representations labelled $\alpha$ and $y$ in Figure \ref{fig21} correspond to the two intersections of $\{\theta=\gamma\}$ with the arc ${\bf I}_0$ in Figure \ref{fig34}: the intersection at the corner is $\alpha$ and the non-abelian intersection in the interior of ${\bf I}_0$ is $y$. The remaining two points, $x$ and $z$ lie on  the intersection of ${\bf I}_+$ and ${\bf I}_-$ with $\{\theta=\gamma\}$.

From Theorem \ref{pertvar} and Corollary \ref{corimage} we see that $x,y$ and $z$ each perturb to  give two generators for $CI^\nat(S^3,T_{3,4})$, which we label $\{x_1, x_2,y_1,y_2,z_1,z_2\}$.  These arise by perturbing a Morse-Bott critical circle, and hence the gradings satisfy $gr(x_1)-gr(x_2)=
gr(y_1)-gr(y_2)=gr(z_1)-gr(z_2)=1$ (after perhaps reindexing).

The last generator $\alpha'$ (the perturbation of $\alpha$) contributes $(1,0,0,0)_a$ to $CI^\nat(S^3,T_{3,4})$.
Thus we see that $CI^\nat(S^3,T_{3,4})$ is a direct sum of graded groups
$$CI^\nat(S^3,T_{3,4})= (1,0,0,0)_a\oplus (1,1,0,0)_{b(x)}\oplus (1,1,0,0)_{b(y)} \oplus (1,1,0,0)_{b(z)}.$$
Using splitting theorems for spectral flow, such as those of \cite{BHKK} (see also the discussion in \cite{KM-filtrations}) one can see from Figure \ref{fig21} that
$b(y)= b(x)+2$ and $b(z)=b(x)+4\equiv b(x)$. Thus  $$CI^\natural(S^3,T_{3,4})=(1,0,0,0)_a\oplus (2,2,1,1)_{b},$$ where $b=b(x)$.  The Alexander polynomial $\Delta_K(t)=t^3+\frac{1}{t^3}-t^2-\frac{1}{t^2}+1$. The sum of the absolute value of its coefficients tells us the rank of $I^\nat(S^3,T_{3,4})$ is at least $5$,  and hence $I^\natural(S^3,T_{3,4})=Kh^{red}(T_{3,4}^m)=(2,1,1,1)$.  Thus the rank of the differential on the complex $CI^\nat(S^3,T_{3,4})$ is one, and the spectral sequence collapses after the $E_2$ page.
 
In this example,  the representation labeled $y$ is  the only non-abelian  binary dihedral representations, and hence corresponds to the representation $y$ of Figure \ref{fig34}.

\subsection{The $(2,n)$ torus knots}

For the $(2,2k+1)$ torus knots, there are $k$ non-abelian arcs in $\chi(S^3,T_{2,2k+1})$, and these are nested. Each arc contains a traceless representation, and the signature of $T_{2,2k+1}$ equals $-2k$.

Thus  $R^\nat_\pi(S^3,K)$ consists of $2k+1$ points which correspondingly generate $CI^\nat(S^3,T_{2,2k+1})$.  Spectral flow considerations show that since the non-abelian arcs in $\chi(S^3,T_{2,2k+1})$ are nested, the grading difference from $\alpha$ to each non-abelian point in $R(S^3,K)$ is  successively $b,b+2,b+4,\cdots$ mod 4 for some integer $b$.  Hence 
\begin{eqnarray*}
I^\nat(S^3,K)=CI^\nat(S^3,T_{2,2k+1})&=&(1,0,0,0)_{a(k)}\oplus_{j=1}^{k} (1,1,0,0)_{b+2j}\\
&=&(1,0,0,0)_{a(k)}\oplus 
 \begin{cases} (\frac{k+1}{2},\frac{k+1}{2},\frac{k-1}{2},\frac{k-1}{2})_{b(k)}& \text{if $k$ is odd,}\\
 (\frac{k}{2},\frac{k}{2},\frac{k}{2},\frac{k}{2})
 & \text{if $k$ is even.}
\end{cases}
\end{eqnarray*}
where $a(k)$ denotes the grading of the generator $\alpha'$, and $b(k)$ is some integer.

On the other hand, one can easily compute $Kh^\red(T_{2,2k+1}^m)$ (e.g., by the method of  \cite[Section 6.2]{KhovanovI})  to find: 

$$Kh^\red(T_{2,2k+1}^m)= (1,0,0,0)\oplus
\begin{cases} (\frac{k-1}{2},\frac{k-1}{2},\frac{k+1}{2},\frac{k+1}{2})& \text{if $k$ is odd,}\\
 (\frac{k}{2},\frac{k}{2},\frac{k}{2},\frac{k}{2})
 & \text{if $k$ is even.}
\end{cases}
$$
\noindent One sees that these values are consistent with the possibility  that $a(k)= \sigma(T_{2,2k+1}) \mod 4$ and  $b(k)=3$ for all $k$.

\subsection{The $(3,n)$ torus knots}

For the $(3,n)$ torus knots with $n\leq 38$, we have

$$CI^\nat(S^3, T_{3, n})=(1,0,0,0)_{a(n)} \oplus
\begin{cases} (2k,2k,2k,2k)& \text{ if~} n=6k+1\\
(2k+1,2k+1,2k,2k)_{b(n)}& \text{ if~} n=6k+2\\
(2k+2,2k+2,2k+1,2k+1)_{b(n)}& \text{ if~} n=6k+4\\
(2k+2,2k+2,2k+2,2k+2)& \text{ if~} n=6k+5
\end{cases}
$$
and
$$Kh^\red( T_{3, n}^m)= 
\begin{cases} (2k+1,2k,2k,2k)& \text{ if~} n=6k+1\\
(2k+1,2k,2k+1,2k+1)& \text{ if~} n=6k+2\\
(2k+2,2k+1,2k+1,2k+1)& \text{ if~} n=6k+4\\
(2k+2,2k+1,2k+2,2k+2)& \text{ if~} n=6k+5.
\end{cases}
$$
For these knots, the sum of the absolute values of the coefficients of the Alexander polynomial $ |\Delta|$ equals the rank of the reduced Khovanov homology. The absolute value of the signature $|\sigma|$ satisfies

$$|\sigma(T_{3,n})|+1=\begin{cases}

 |\Delta(T_{3,n})|&\text{ if } n=6k+1, 6k+2,\\
|\Delta(T_{3,n})|+2&\text{ if } n=6k+4, 6k+5.

\end{cases}$$
Thus $Kh^\red(T_{3,n}^m)=I^\nat(S^3, K_{3,n})$, and when $n=6k+1, 6k+2$, there are no non-zero differentials in the complex $CI^\nat(S^3, T_{3,n})$. When $n=6k+4, 6k+5$, the rank of the differential is one.  In either case there are no non-zero differentials after the $E_2$ page in the Kronheimer-Mrowka spectral sequence.  The data is consistent with the guess $a(n)=\sigma(T_{3,n})\mod 4$ and $b(n)=3$.  The unreduced Khovanov homology is known for all $(3,n)$ torus knots \cite{Turner}. Presumably similar calculations could verify the   formulae above for all $n$.

 \subsection{Other torus knots}  
 Patterns for more complicated torus knots are not as obvious. But  for fixed $m$, the rank $|\sigma|+1$ of $CI^\nat(S^3, T_{m,n})$  and the lower bound $ |\Delta|$  grow linearly in $n$, whereas the rank of $Kh^\red(T_{m,n})$ seems to be growing more quickly for $m\ge 4$. 

 The computations of $CI^\nat$  for torus knots are done by computing their data in the sense of Section \ref{toruschi}, using the MAPLE computer algebra package.  Some examples were listed in Section \ref{toruschi}. The  table below includes the signature  $\sigma$ and the  
sum of the absolute values of the coefficients of the Alexander polynomial $|\Delta|$. We write $A_a$ for $(1,0,0,0)_a$. Whenever    $CI^\nat(K)$ and  $I^\nat(K)$ differ the chain complex $CI^\nat(K)$ has non-trivial differential.
Whenever    $Kh^\red(K)$ and  $I^\nat(K)$ differ the Kronheimer-Mrowka spectral sequence has non-trivial higher differentials.

 \medskip

\medskip

A  torus knot which exhibits interesting $CI^\nat$ is the $(4,5)$ torus knot, with  $(c_i,d_i)$ equal to 
$$(1,9), (7,17), (2,18), (6,14), (11,19), (3,13).$$
The first and fourth arc do not contain traceless representations, hence do not contribute to  $CI^\nat(S^3, T_{4,5})$. The remaining four arcs contain traceless representations, hence $CI^\nat(S^3, T_{4,5})$ is generated by $8+1=9 $ elements.  
As explained in \cite{KM-filtrations}, this example is remarkable because although the Khovanov homology also has 9 generators, the gradings ($CI^\nat=(3,2,2,2)_a$  and $Kh^\red=(2,1,3,3)$) are incompatible with a spectral sequence with no higher differentials. Since $|\Delta|=7$, it follows that the differential on $CI^\nat(S^3, T_{4,5})$ has rank one.   As mentioned in \cite{KM-filtrations}, this is not quite enough to compute $I^\nat(S^3,T_{4,5})$, as both $(2,1,2,2)$ and $(1,1,3,2)$  are compatible.  In \cite{KM-filtrations} they establish that $I^\nat(S^3,T_{4,5})=(2,1,2,2)$ using  the results of  \cite{KMalex,Lim} which identify the coefficients of the Alexander polynomial with 
a kind of Euler characteristic associated to the  generalized eigenspaces of an operator $\mu:KHI(K)\to KHI(K)$.   If, as seems likely, the grading of the generator $\alpha$ occurs in degree $\sigma(K) \mod 4$ (see below) then this would also show that 
$I^\nat(S^3,T_{4,5})=(2,1,2,2)$.

In the spectral sequence for $T_{4,9}$, the rank from the $E_2$ page to the limit drops by at least $8$. For $T_{5,7}$, all differentials in $CI^\nat(S^3,T_{5,7})$ are zero, but in the spectral sequence, the rank drops from 29 for the $E_2$ page  to 17 in the limit. 

For $T_{5,6}$, the ranks of $CI^\nat(S^3, T_{5,6})$ and $Kh^\red(T_{5,6}^m)$ are equal, but the gradings are different so that the spectral sequence necessarily has higher differentials, and hence  $CI^\nat(S^3, T_{5,6})$ also has non-trivial differentials. For this knot, the knot Floer homology groups have rank equal to 9,  and hence further examination of this knot may shed light on the relationship between the three knot invariants.

For the larger values of $(p,q)$ in the table, examples were chosen so that the Alexander polynomial sufficed to conclude that $CI^\nat(S^3,K)$ has no higher differentials. This is because $|\sigma(K)|+1$ equals $|\Delta|$.
 
 \medskip
  
 \begin{tabular}{ |c|c|c|c|c|c|c| }
  \hline                        
  Torus knot & $\sigma $ &  $ |  \Delta | $ & $CI^\nat(K)$ &$ Kh^\red(K^m)$& $I^\nat(K)$ \\
    \hline 
    (4,5)&-8 &7 &$ (3,2,2,2)_a $ &$(2,1,3,3)$&$ $\\   
  (4,7)&-14 &11 &$A_a\oplus(4,4,3,3)_b$, &$(4,4,5,4 )$& $11\leq $rank $\leq15$  \\
  (4,9)&-16 &13 &$ (5,4,4,4)_a $&$(7,6,6,6)$& $13\leq $rank $\leq17$\\ 
  (4,11)&-22 &17 &$A_a\oplus (6,6,5,5)_b$ &$(10,9,9,9 )$& $17\leq $rank $\leq23$\\   
   (4,13)&-24 &19 & $(7,6,6,6)_a  $ &$(12,11,13,13)$& $19\leq \rank \leq 25$\\

 (4,15)&-30&23&$A_a\oplus(8,8,7,7)_b$& $(16,16,17,16)$& $23\leq\rank\leq 31$\\
   (4,17)&-32&25&$A_a\oplus(8,8,8,8)$& $(21,20,20,20)$& $25\leq\rank\leq 33$\\
  (4,19)&-38&29&$A_a\oplus(10,10,9,9)_b$& $(26,25,25,25)$& $29\leq\rank\leq 39$\\
 (4,21)&-40&31&$A_a\oplus(10,10,10,10)$& $(30,29,31,31)$& $31\leq\rank\leq 41$\\
 (4,23)&-46&35&$A_a\oplus(12,12,11,11)_b$& $(36,36,37,36)$& $35\leq\rank\leq 47$\\
 (4,25)&-48&37&$A_a\oplus(12,12,12,12)$& $(43,42,43,42)$& $37\leq\rank\leq 49$\\

   (5,6)&-16 &9 &$ (5,4,4,4)_a  $  &$(5,3,4,5)$& $9\leq $rank $\leq15$\\
   (5,7)&-16 &17 &$ (5,4,4,4)_a  $ &$(8,6,7,8)$& $(5,4,4,4)_a $\\
   (5,8)&-20 &19 &$(6,5,5,5)_a   $ &$(9,8,9,9)$& $19\leq $rank $\leq21$\\
    (5,9)&-24 &15 &$(7,6,6,6)_a   $ &$(10,10,11,10)$& $15\leq $rank $\leq25$\\
     (5,11)&-24 &17 &$(7,6,6,6)_a   $ &$(15,14,14,14)$& $17\leq $rank $\leq25$\\
   (5,12)&-28 &29 &$(8,7,7,7)_a   $ &$(20,19,19,19)$& $(8,7,7,7)_a $\\
    (5,17)&-40 &41 &$(11,10,10,10)_a   $ &$(38,36,37,38)$& $(11,10,10,10)_a $\\
    (5,22)&-52 &53 &$(14,13,13,13)_a   $ &$(62,61,61,61)$& $(14,13,13,13)_a $\\
     (5,117)&-280&281&$ (71,70,70,70)_a$ & ?& $(71,70,70,70)_a$  \\
    (6,7)&-18 & 11 &$A_a\oplus(5,5,4,4)_b   $ &$(7,7,9,8)$& $11\leq$rank$\leq19$\\
    (7,16)&-54 & 55 &$A_a\oplus(14,14,13,13)_b   $ &$?$&$A_a\oplus(14,14,13,13)_b$ \\
    (7,30)&-102 & 103 &$A_a\oplus(26,26,25,25)_b   $ &$?$&$A_a\oplus(26,26,25,25)_b   $ \\
     (9,11)&-48 & 49 &$(13,12,12,12)_a   $ &$?$&$(13,12,12,12)_a   $ \\
     (9,25)&-112 & 111 &$(29,28,28,28)_a  $ &$?$&   $111\leq $rank $\leq 113$ \\
      (9,29)&-128 & 129 &$ (33,32,32,32)_a  $ &$?$&$ (33,32,32,32)_a $ \\
(11,24)&-130 & 131 &$ A_a\oplus(33,33,32,32)_b  $ &$?$&$ A_a\oplus(33,33,32,32)_b  $ \\
(11,31)&-168 & 169 &$(43,42,42,42)_a   $ &$?$&$(43,42,42,42)_a   $ \\
(13,15)&-96 & 97 &$(25,24,24,24)_a   $ &$?$&$(25,24,24,24)_a$ \\
(13,28)&-180 & 181 &$(46,45,45,45)_a   $ &$?$&$(46,45,45,45)_a$ \\
\hline
\end{tabular}

 \medskip
 
 \subsection{Speculation}  The data calculated above is consistent with the conjecture that the generator $\alpha$ of $CI^\nat(K)$ has grading equal to the signature $\sigma(K)$ mod 4.   This is likely true, and a proof should follow from the following outline.  Push a Seifert surface for $K$ into the 4-ball and ambiently surger half a symplectic basis. This yields a 4-manifold $X$ with boundary in which $K$ bounds a disk $D$, such that $\alpha$ extends to $\pi_1(X_0)$  where $X_0=X\setminus N(D)$.   The signature of $(X_0,\alpha)$ gives the signature of $K$. 
 Now glue in $(S^1\times D^2)\times [0,1]$, containing the surfaces $ (S^1\sqcup H)\times[0,1]$   together with the singular bundle data over $W\times[0,1]$ (i.e. cross Figure \ref{natural} with an interval  and glue it to $X_0$ along an annulus in $D$ times an interval). This gives a flat singular cobordism from $(S^3,K)$ to $(S^3,U)$, where $U$ denotes the unknot. The Fredholm index of this flat cobordism (in the sense of \cite[Proposition 4.4]{KM-khovanov}) should be the signature of $K$ by an excision argument.   
 
 The data is  also consistent with the conjecture that for any torus knot, the chain complex $CI^\nat(K)$ has at most one non-trivial differential or, equivalently, that rank$(I^\nat(K))=|\sigma(K)|\pm 1$.

 Looking at the table above, one could also  conjecture that $b=3=-1\mod 4$. This corresponds  to a spectral flow along an arc of flat connections in $\chi(S^3,K)$  starting at the abelian flat connection $\alpha$ traveling towards the trivial connection (i.e. $\alpha_t$ as $t$ decreases from $t=\frac{1}{2}$ in the notation of Definition \ref{SRV}) and changing branches into  the first irreducible  arc encountered.  For general knots  with simple representation varieties $ b$ might be $\mp 1$, according  to the sign of the change in  Levine-Tristram signatures; for torus knots this sign is always positive (see \cite{Herald2}).

Much more ambitiously, one could hope that an Atiyah-Floer conjecture holds in this context, which would describe how to calculate gradings and differentials from  intersection diagrams in the  pillowcase (and perhaps additional data internal to $R(Y_0,K_0)$).  Exploring this topic provides motivation for the problem of describing the spaces $R(Y_0,K_0)$ and their image in the pillowcase for more general   tangles than those that arise from 2-bridge and torus knots.  
It would be interesting to relate our examples and calculations to the approach of Wehrheim and Woodward as alluded to in \cite{WW}.

 \end{document}